\documentclass[a4paper,10pt]{article}

\usepackage{amssymb}
\usepackage{amsmath}
\usepackage{amsthm}

\usepackage{printlen}
\usepackage{enumitem}
\usepackage{graphicx} 
\usepackage[hypertexnames=false]{hyperref}
\usepackage{mathtools}
\usepackage{cleveref}
\usepackage{cleveref-forward}
\usepackage{moreverb,url}
\usepackage{blindtext}
\usepackage[font=small,labelfont=bf]{caption}
\usepackage{csquotes}
\usepackage{xcolor}
\usepackage{orcidlink}
\usepackage{geometry}
\usepackage{xparse} 
\usepackage{color}
\usepackage{bm}
\usepackage{booktabs} 

\theoremstyle{plain}
\newtheorem{theorem}{Theorem}[section]
\crefname{theorem}{theorem}{theorems}
\Crefname{theorem}{Theorem}{Theorems}
\newtheorem{corollary}[theorem]{Corollary}
\crefname{corollary}{corollary}{corollaries}
\Crefname{corollary}{Corollary}{Corollaries}
\newtheorem{lemma}[theorem]{Lemma}
\crefname{lemma}{lemma}{lemmata}
\Crefname{lemma}{Lemma}{Lemmata}
\newtheorem{proposition}[theorem]{Proposition}
\crefname{proposition}{proposition}{propositions}
\Crefname{proposition}{Proposition}{Propositions}

\theoremstyle{definition}
\newtheorem{definition}[theorem]{Definition}
\crefname{definition}{definition}{definitions}
\Crefname{definition}{Definition}{Definitions}
\newtheorem{example}[theorem]{Example}
\crefname{example}{example}{examples}
\Crefname{example}{Example}{Examples}

\theoremstyle{remark}
\newtheorem{remark}[theorem]{Remark}
\crefname{remark}{remark}{remarks}
\Crefname{remark}{Remark}{Remarks}

\newenvironment{myex}
{\pushQED{\qed}\example}
{\popQED\endexample}

\AddToHook{env/corollary/begin}{\crefalias{theorem}{corollary}}
\AddToHook{env/lemma/begin}{\crefalias{theorem}{lemma}}
\AddToHook{env/proposition/begin}{\crefalias{theorem}{proposition}}
\AddToHook{env/example/begin}{\crefalias{theorem}{example}}

\newcommand{\none}{^{k+1}}
\newcommand{\nonehalf}{^{k+1/2}}
\newcommand{\n}{^{k}}
\newcommand{\discrete}[1]{\overline{#1}}
\newcommand{\discretetimestep}[1]{{#1^{k,k+1}}}
\newcommand{\costate}{z}
\newcommand{\diag}{\mathrm{diag}}
\newcommand{\discreteE}{\overline{E}}
\newcommand{\discreteinput}{\discretetimestep{u}}
\newcommand{\discreteJ}{\overline{J}}
\newcommand{\discreteoutput}{\discretetimestep{y}}
\newcommand{\discreteR}{\overline{R}}
\newcommand{\discreteB}{\overline{B}}
\newcommand{\discretedissipation}{\discretetimestep{W_{\mathrm{diss}}}}
\newcommand{\discretestate}{\state}
\newcommand{\discretecostate}{\overline{\costate}}
\newcommand{\discretecostateone}{\discretetimestep{\costate_1}}
\newcommand{\discretecostateoneq}{\discretetimestep{\costate_{1,q}}}
\newcommand{\discretecostateonev}{\discretetimestep{\costate_{1,v}}}
\newcommand{\finaltime}{t_{\mathrm{end}}}
\newcommand{\hamiltonian}{\mathcal{H}}
\newcommand{\statespace}{\mathcal{X}}
\newcommand{\cont}{\mathcal{C}}
\newcommand{\ntimesteps}{N}
\newcommand{\R}{\mathbb{R}}
\newcommand{\N}{\mathbb{N}}
\newcommand{\rank}{\mathrm{rank}}
\newcommand{\state}{x}
\newcommand{\timeinterval}{\mathbb{T}}
\newcommand{\specified}[1]{#1_1}
\renewcommand{\d}{{\,\rm  d}}

\newcommand{\transp}{^\top}
\newcommand{\ntransp}{^{-\top}}

\renewcommand{\phi}{\varphi}

\newcommand{\DG}[1][]{\ensuremath{\overline{\nabla}#1}}
\renewcommand{\epsilon}{\varepsilon}
\newcommand{\dd}[2]{\ensuremath{\frac{\mathrm{d}#1}{\mathrm{d}#2}}}
\newcommand{\ddt}{\dd{}{t}}

\newcommand{\gradient}{\nabla}
\newcommand{\jacobian}{\mathrm{D}}
\newcommand{\discretejacobian}[1][]{\ensuremath{\overline{\mathrm{D}}#1}}

\DeclarePairedDelimiter{\set}{\{}{\}}
\DeclarePairedDelimiter{\pset}{(}{)}
\DeclarePairedDelimiter{\norm}{\lVert}{\rVert}

\definecolor{color1}{RGB}{230, 159, 0}
\definecolor{color2}{RGB}{86, 180, 233}
\definecolor{color3}{RGB}{204, 121, 167}
\definecolor{color4}{RGB}{0, 158, 115}
\definecolor{color5}{RGB}{0, 114, 178}
\definecolor{color6}{RGB}{213, 94, 0}
\definecolor{color7}{RGB}{240, 228, 66}
\definecolor{colorblack}{RGB}{0, 0, 0}          

\pgfdeclareradialshading[sphere color]{ballshading}{
	\pgfpoint{3mm}{3mm}}{color(0mm)=(sphere color!30!white);
	color(4mm)=(sphere color!75!white);
	color(8mm)=(sphere color!70!black);
	color(12mm)=(sphere color!50!black);
	color(15mm)=(black)}

\usetikzlibrary{
	decorations.pathmorphing,
	arrows.meta,
	arrows,shadings,decorations.markings,decorations
}
\usetikzlibrary{calc}
\usetikzlibrary{matrix}
\usetikzlibrary{shapes.geometric}

\usepackage{pgfplots}
\pgfplotsset{compat = newest}
\usepackage{pgfplotstable}
\usepackage{tikz}
\pgfkeys{/pgf/number format/.cd,fixed,precision=4}

\newlength\figH
\newlength\figW

\usepackage[isbn=false,url=false,giveninits=true,maxnames=6]{biblatex}
\bibliography{bib}
	
\NewDocumentCommand{\RMnew}{om}{{\color{blue}\IfNoValueTF{#1}{#2}{\sout{#1}\;#2}}}

\begin{document}

\author{}
\title{Discrete gradient methods for port-Hamiltonian differential-algebraic equations}

\maketitle


\begin{center}
	\begin{minipage}{0.45\textwidth}
		\centerline{\scshape Philipp L. Kinon}
		\medskip
		{\footnotesize
			\centerline{Institute of Mechanics}
			\centerline{Karlsruhe Institute of Technology (KIT), Germany}
			\centerline{philipp.kinon@kit.edu}
		}
	\end{minipage}
	\begin{minipage}{0.45\textwidth}
		\centerline{\scshape Riccardo Morandin}
		\medskip
		{\footnotesize
			\centerline{Institute of Analysis and Numerics}
			\centerline{OvGU Magdeburg, Germany}
			\centerline{riccardo.morandin@ovgu.de}
		}
	\end{minipage}

	\bigskip

	\centerline{\scshape Philipp Schulze}
	\medskip
	{\footnotesize

		\centerline{Institute of Mathematics}
		\centerline{Technische Universit\"at Berlin, Germany}
		\centerline{pschulze@math.tu-berlin.de}
	}
	\bigskip

\end{center}
\begin{abstract}
	\noindent Discrete gradient methods are a powerful tool for the time discretization of dynamical systems, since they are structure-preserving regardless of the form of the total energy. In this work, we discuss the application of discrete gradient methods to the system class of nonlinear port-Hamiltonian differential-algebraic equations - as they emerge from the port- and energy-based modeling of physical systems in various domains.
	We introduce a novel numerical scheme tailored for semi-explicit differential-algebraic equations and further address more general settings using the concepts of discrete gradient pairs and Dirac-dissipative structures.
	Additionally, the behavior under system transformations is investigated and we
	demonstrate that under suitable assumptions port-Hamiltonian differential-algebraic equations admit a representation which consists of a parametrized port-Hamiltonian semi-explicit system and an unstructured equation.
	Finally, we present the application to multibody system dynamics and discuss numerical results to demonstrate the capabilities of our approach.



	\vskip .3truecm

	\noindent 	{\bf Keywords}:
	Port-Hamiltonian systems, differential-algebraic equations, structure-preserving discretization, time integration methods, discrete gradients

	\vskip .3truecm

	\noindent 	{\bf AMS(MOS)}:
	34A09, 65L80, 65P10, 70E55, 93C10

\end{abstract}


\section{Introduction}
\noindent\emph{Port-Hamiltonian} (pH) systems have gained significant importance in various research areas, with a particular focus on the modeling, simulation, and control of dynamical systems \cite{duindam_2009_modeling,schaft_2014_porthamiltonian}. PH systems offer a valuable framework for analyzing complex problems, where the complexity may arise from multi-physical interactions, non-trivial domains, and various nonlinearities. One of the key advantages of the pH representation is its explicit description of power interfaces, known as ports, which facilitate power-preserving interconnections between submodules. Thus, this approach simplifies the modular composition of models, which often leads to the presence of algebraic constraints. Correspondingly,
the governing equations at hand are \emph{differential-algebraic equations} (DAEs), also known as \emph{descriptor systems} in the context of control theory. If the system has in addition a pH structure, we speak of \emph{port-Hamiltonian differential-algebraic equations} (pHDAEs).
A definition for linear time-varying pHDAEs was provided in \cite{beattie_2018_linear} and a full, nonlinear generalization has been provided in \cite{mehrmann_2019_structurepreserving}. An important subclass consists of semi-explicit pHDAEs, see e.g. \cite[Eq.~3.16]{vanderschaft_2013_porthamiltonian}, where local representations of implicit port-Hamiltonian DAEs are discussed.
In \cite{SchM20,vanderschaft_2018_generalized}, the Hamiltonian as a backbone of pH systems is replaced by Lagrangian subspaces or submanifolds to define generalized pHDAEs.

In general, discretizing a structured dynamical system, such as a pH system, can result in the loss of its continuous-time properties, potentially leading to numerical solutions that exhibit unphysical behavior (see, for example, \cite[Ch.~1]{hairer_2006_geometric}). One way to mitigate this issue is by employing a structure-preserving time discretization scheme, as the system's properties are often embedded in the algebraic or geometric structure of the original continuous-time model. Examples of such systems include gradient \cite{HirS74}, Hamiltonian \cite{arnold_1989_mathematical}, and, particularly relevant to this work, pH systems. Structure-preserving time discretization approaches for Hamiltonian systems have been widely studied, with \cite{hairer_2006_geometric} offering a general overview. Notably, the development of structure-preserving discretization methods has been driven by computational mechanics \cite{betsch_2016_structurepreserving,gonzalez_2000_exact,gonzalez_1999_mechanical,kinon_2023_structurepreserving,simo_1992_discrete}, where variational integrators \cite{betsch_2016_structurepreserving,leyendecker_2008_variational,marsden_2001_discrete} represent an important discretization approach within the group of symplectic methods \cite{leimkuhler_1994_symplectic}. An interesting approach is also given by time finite element methods (see, e.g. \cite{betsch_2001_conservation,betsch_2002_conservation,betsch_2000_conservation,EggHS21,MayBetsch2025}). Structure-preserving techniques for other system classes are for example explored in \cite{JueST19,KunM23,Oet18,simoes_2023_discrete}.

Compared to those works, the structure-preserving time discretization of pH systems is still a relatively young field.
When performing numerical integration of pH systems, it is essential to account for the energy exchange through the ports resulting in the presence of a power balance equation.
The following developments have been made in recent years:
\begin{itemize}
	\item In \cite{KotL19}, the authors show that certain collocation methods can achieve an exact power balance at the discrete level, provided that the total energy function, the \emph{Hamiltonian}, is a quadratic function of the state. This result is further extended to descriptor systems in \cite{mehrmann_2019_structurepreserving}.
	\item Structure-preserving discretization approaches based on Petrov–Galerkin projections have been proposed in \cite{EggHS21,GieKT24} and are closely connected to the aforementioned time finite element approaches. Although these methods can provide continuous solutions also between discrete points in time, and one can obtain arbitrarily high convergence rates, they require the numerical approximation of integrals in time. Not being able to integrate these formulas sufficiently accurately can lead to the loss of the desired convergence and conservation properties \cite{betsch_2000_conservation}.
	      Moreover, this numerical quadrature imposes additional numerical costs for the emanating schemes.
	\item In several recent works, e.g.~\cite{BarDFG23,BarDFGM25,MonM25}, the authors consider splitting schemes that separate the energy-conserving and dissipative parts of the dynamics.
	      While this approach can achieve high order convergence and seems quite promising, to the best of our knowledge, it has been so far only applied to linear port-Hamiltonian systems with quadratic Hamiltonian.
	\item Another approach consists in dropping the requirement for an exact time-discrete power balance, while focusing on minimizing its violation, for example by refining adaptively the time grid of the discretization, see e.g.~\cite{BarS25}.
\end{itemize}

\noindent Contrary to these approaches, the present work pursues a \emph{discrete gradient} approach, which achieves exact time-discrete power balances also for non-quadratic Hamiltonians.
Additionally, the implementation of such schemes is comparably simple and straightforward.
While most of the known discrete gradient schemes are restricted to second order convergence rates, there are recent developments to obtain higher accuracy as well (see \cite{eidnes_2022_order} and the references therein).
Another notable work \cite{sato_2019_linear} deals with DAEs with a gradient structure and constant descriptor matrix.

Most of the approaches in the literature for pH systems \cite{aoues_2017_hamiltonian,celledoni2017energy,falaize_2016_passive,FroGLM24,goren-sumer_2008_gradient,kinon_2023_discrete,MorMMN19}, which achieve an exact power balance at the discrete level for general Hamiltonians, share the characteristic that they focus on \emph{pH ordinary-differential equations}, where the gradient of the Hamiltonian explicitly appears in the system equations. A challenge with applying methods like discrete gradient techniques to more general systems as introduced in \cite{mehrmann_2019_structurepreserving} lies in the fact that the gradient of the Hamiltonian in general only appears implicitly in the system equations. The development of discrete gradient pairs \cite{schulze_structure_2023} has recently addressed this issue.

In contrast to the works focusing on ordinary differential equations, we want to generalize the application field of discrete gradient methods to pHDAEs with possibly state-dependent descriptor matrices, as introduced in \cite{mehrmann_2019_structurepreserving}.
The primary contributions of this work are outlined in the following:

\begin{enumerate}[label=(\roman*)]
	\item Discussion of discrete gradient pairs for general pHDAEs along with a corresponding time integration approach, see \Cref{sec_version2}.
	\item Development of a tangible discrete gradient method for \emph{semi-explicit} pHDAEs, see \Cref{sec_version1}. This already covers many application problems.
	\item Discussion of a method based on an alternative representation of pHDAEs, see \Cref{sec_version3}.
	\item In-depth analysis of relations between the proposed methods and their behavior under coordinate transformations, see \Cref{sec_connections}.
\end{enumerate}

\smallskip\noindent
The remainder of this work is structured as follows: Preliminary basics are recapitualed in \Cref{sec_Preliminaries}, including the definition of pHDAEs and discrete gradients. In \Cref{sec:semi-explicitPHDAEs} we focus on a certain class of pHDAEs, namely semi-explicit pHDAEs.
We then introduce new methods for the numerical integration of pHDAEs using discrete gradients in \Cref{sec_methods}.
We analyze the relationship between the proposed methods as well as their behavior under system transformations in \Cref{sec_connections}.
\Cref{sec_examples} is entirely devoted to the application of our approaches to multibody systems, including numerical experiments. Conclusions and a brief outlook are given in \Cref{sec_Conclusion}.

\subsection{Notation}

\noindent We denote by $\N$ the positive natural numbers and by $\N_0$ the natural numbers including zero.
With $I_n\in\R^{n,n}$ or simply $I$ we denote the identity matrix and with $0$ the zero matrix or vector. We mostly assume that the dimension should become clear from the context.
For every matrix $A\in\mathbb R^{n,m}$ or vector $v\in\mathbb R^n=\mathbb R^{n,1}$ we denote by $A\transp\in\mathbb R^{m,n}$ and $v\transp\in\mathbb R^{1,n}$ their corresponding transposes. Additionally, we sometimes abbreviate less important, unspecified terms by \lq\lq $\star$\rq\rq, to enhance the reader's focus.

We denote by $\cont(X,Y)$ the continuous functions between two topological spaces $X$ and $Y$.
For $k\in\N_0\cup\set{\infty}$ we denote by $\cont^k(\statespace_1,\statespace_2)$ the $k$-times continuously differentiable functions from $\statespace_1$ to $\statespace_2$, where typically $\statespace_1\subseteq\R^n$ and $\statespace_2\subseteq\R^m$ are open subsets for some $n,m\in\N$.

If $f\in\cont^1(\statespace,\R)$ with $\statespace\subseteq\R^n$ open, we denote by $\gradient{f} \in\cont(\statespace,\R^n)$ the gradient of $f$, intended as a column vector function.
If $F\in\cont^1(\statespace,\R^m)$ with $\statespace\subseteq\R^n$ open, we denote by $\jacobian{F} \in\cont(\statespace,\R^{m,n})$ the Jacobian of $F$, intended as a matrix function whose rows transposed are the gradients of the entries of $F$.
Furthermore, given a partition $F=(F_1,\ldots,F_m)$ for the function and $\state=(\state_1,\ldots,\state_r)$ of the state variable, with $\state_i=(\state_{i,1},\ldots,\state_{i,n_i})\in\R^{n_i}$ for $i=1,\ldots,r$, we denote the corresponding partial gradients and partial Jacobians as
\[
	\gradient_{x_i} f =
	\begin{bmatrix}
		\frac{\partial f}{\partial{x_{i,1}}} \\
		\vdots                               \\
		\frac{\partial f}{\partial{x_{i,n_i}}}
	\end{bmatrix}, \qquad
	\jacobian_{x_i} F =
	\begin{bmatrix}
		\frac{\partial F_1}{\partial{x_{i,1}}} & \cdots & \frac{\partial F_1}{\partial{x_{i,n_i}}} \\
		\vdots                                 & \ddots & \vdots                                   \\
		\frac{\partial F_m}{\partial{x_{i,1}}} & \cdots & \frac{\partial F_m}{\partial{x_{i,n_i}}}
	\end{bmatrix},
\]
such that in particular
\begin{equation} \label{def_jacobian}
	\gradient f =
	\begin{bmatrix}
		\gradient_{x_1} f \\ \vdots \\ \gradient_{x_r} f
	\end{bmatrix}, \qquad
	\jacobian F =
	\begin{bmatrix}
		\jacobian_{x_1}F & \cdots & \jacobian_{x_r}F
	\end{bmatrix}.
\end{equation}
\noindent Additionally, the derivative with respect to time $t$ deserves its own notation, which is $\dot x\coloneqq\frac{\d x}{\d t}$.

If $f:\mathcal X\to\mathcal Y$ and $g:\mathcal Y\to\mathcal Z$ are two maps, we denote as usual with $g \circ f:\mathcal X\to\mathcal Z$ their composition, i.e., $g \circ f(x) = g(f(x))$.
When $\discrete{g}:\mathcal Y\times\mathcal Y\to\mathcal Z$, we sometimes abuse the notation and write $\discrete{g} \circ f:\mathcal X\times\mathcal X\to\mathcal Z$ to denote the map $\discrete g \circ f(x,x')=\discrete g(f(x),f(x'))$.

For every matrix function $A\in\cont(\statespace,\R^{m,n})$, we denote by $A\transp\in\cont(\statespace,\R^{n,m})$ its pointwise transpose $A\transp(x)=A(x)\transp$.
If furthermore $m=n$ and $A$ is pointwise invertible, we usually denote by $A^{-1}\in\cont(\statespace,\R^{n,n})$ its pointwise inverse $A^{-1}(\state)=A(\state)^{-1}$, instead of the inverse map, unless otherwise specified.
We also introduce the short notation $A\ntransp$ for $(A^{-1})\transp=(A\transp)^{-1}$.
Given a subset $\mathcal V\subseteq\R^n$, we denote by $\mathrm{span}(\mathcal V)\subseteq\R^n$ the smallest linear subspace of $\R^n$ containing $\mathcal V$, and by
\[
	\mathcal V^\perp = \set{ v \in \R^n \mid v\perp w \text{ for all }w\in\mathcal V }
\]
its orthogonal complement. When $\mathcal V$ consists of only one vector $v\in\R^n$, we simply write $\mathrm{span}(v)$ and $v^\perp$ instead of $\mathrm{span}(\set{v})$ and $\set{v}^\perp$.
Given a subset $\statespace\subseteq\R^{n_1}\times\R^{n_2}$, we usually denote by $\pi_1:\statespace\to\R^{n_1}$ and $\pi_2:\statespace\to\R^{n_2}$ the corresponding orthogonal projections, i.e., $\pi_1(x_1,x_2)=x_1$ and $\pi_2(x_1,x_2)=x_2$ for all $(x_1,x_2)\in\statespace$.

\section{Preliminaries}\label{sec_Preliminaries}

\subsection{Differential-algebraic equations}

\noindent\emph{Differential-algebraic equations} are systems of the form
\begin{equation}\label{eq:DAE_general}
	F\pset*{ t, x, \frac{\textrm{d}x}{\textrm{d}t}, \ldots , \frac{\textrm{d}^kx}{\textrm{d}t^k} } = 0
\end{equation}
for some map $F:\mathcal D_F\to\R^m$, where $t\in \timeinterval \subseteq\R$ denotes the time variable, $x\in \mathcal{X} \subseteq \R^n$ the state variable, and $\mathcal D_F\subseteq\R^{1+(k+1)n}$ is the domain of $F$.
Here $n$ is the dimension of the state variable, $m$ the number of equations, and $k$ is the order of the DAE.
Typically, the domain of $F$ is of the form $\mathcal D_F=\timeinterval\times\statespace\times\R^{kn}$, where $\timeinterval\subseteq\R$ is an open (possibly unbounded) interval and $\statespace\subseteq\R^n$ is an open subset, while the solutions of \eqref{eq:DAE_general} are to be found in $\cont^k(\timeinterval,\statespace)$.

We are particularly interested in first order quasilinear DAEs, i.e., equations of the form
\begin{equation}\label{eq:DAE_quasilinear}
	E(t,x)\dot x = f(t,x),
\end{equation}
see e.g. \cite{rabier_1994_impasse,steinbrecher_2006_numerical}, for some maps $E:\mathcal D_E\to\R^{m,n}$ and $f:\mathcal D_f\to\R^{m}$, where $\mathcal D_E,\mathcal D_f\subseteq\R^{1+n}$.
In particular, if we had $n=m$ and $E$ were pointwise invertible, then \eqref{eq:DAE_quasilinear} would be equivalent to
$
	\dot x = E(t,x)^{-1}f(t,x),
$
which is a system of first order ordinary differential equations (ODEs). However, when this property is not satisfied, the system might include algebraic constraints and be under- or overdetermined.
This presents several challenges, both in the study of the existence and uniqueness of solutions and in the time discretization of the system, see e.g.~\cite{kunkel_2006_differentialalgebraic}.
In particular, dedicated numerical methods are often necessary.

Enriching a DAE with input and output variables $u\in\R^p$ and $y\in\R^q$ we obtain a \emph{descriptor system}
\begin{equation}
	\begin{split}
		E(t,x)\dot x & = f(t,x,u), \\
		y            & = g(t,x,u),
	\end{split}
\end{equation}
for some maps $E:\mathcal D_E\to\R^{m,n}$, $f:\mathcal D_f\to\R^{m}$, and $g:\mathcal D_g\to\R^{q}$, where $\mathcal D_E\subseteq\R^{1+n}$ and $\mathcal D_f,\mathcal D_g\subseteq\R^{1+n+p}$.
In applications, the input $u$ is typically a given fixed time-varying function, a state feedback, or an output feedback.

\subsection{Port-Hamiltonian descriptor systems}

\noindent In this paper we focus on time-invariant port-Hamiltonian descriptor systems.
We introduce first the concept of gradient pair, which will replace the gradient of the Hamiltonian in the equations.

\begin{definition}
	Let $\statespace\subseteq\R^n$ be an open set and let $\hamiltonian\in\cont^1(\statespace,\R)$, $E\in\cont(\statespace,\R^{n,n})$, and $\costate\in\cont(\statespace,\R^n)$.
	We say that $(E,\costate)$ is a \emph{gradient pair} for $\hamiltonian$ if
	\begin{equation}\label{constitutive_equation}
		E(\state)\transp \costate(\state) = \gradient \hamiltonian(\state)
	\end{equation}
	holds for all $\state\in\statespace$.
\end{definition}

\noindent Port-Hamiltonian descriptor systems are then defined as follows.

\begin{definition}[see also \cite{mehrmann_2019_structurepreserving}]\label{def:pHDAE}
	Consider a time interval $\timeinterval=[0,\finaltime]$ with $\finaltime > 0$ and an open state space $\statespace\subseteq\R^n$.
	A \emph{time-invariant port-Hamiltonian descriptor system}, in short \emph{pHDAE}, is a descriptor system of the form
	\begin{equation}\label{eq:pHDAE}
		\begin{split}
			E(x)\dot x & = \pset[\big]{J(x)-R(x)}z(x) + B(x)u, \\
			y          & = B(x)\transp z(x),
		\end{split}
	\end{equation}
	together with a \emph{Hamiltonian} $\hamiltonian\in\cont^1(\statespace,\R)$, where $E,J,R\in\cont(\statespace,\R^{n,n})$, $B\in\cont(\statespace,\R^{n,m})$, and $z\in\cont(\statespace,\R^n)$ satisfy the properties $J(x)=-J(x)\transp$, $R(x)=R(x)\transp\succeq 0$ for all $\state\in\statespace$, and $(E,\costate)$ is a gradient pair for $\hamiltonian$.
	Here $E,J,R$ are called the \emph{descriptor, structure}, and \emph{dissipation} matrix functions, respectively, and $z$ is called the \emph{co-state} function.
\end{definition}

\begin{remark}
	In this work we formally consider only systems without a feedthrough term in the output equation. Nevertheless, our results can be easily adapted for systems with feedthrough, i.e.,~replacing the output equation with $y=C(x)\transp z(x) + D(x)u$ for some matrix functions $C,D$ and requiring some additional dissipative structure involving $R,B,C,D$, see e.g.~definitions in \cite{mehrmann_2019_structurepreserving,mehrmann_2023_control}.
\end{remark}

\begin{remark}
	Although the system from \Cref{def:pHDAE} could emerge from a change of variables of a pH ODE system with state $\hat{x}$ and Jacobian $\jacobian \tilde{x}(x) = E(x)$ inducing $z(x) = \gradient \widetilde{\hamiltonian}(\tilde{x}(x))$, the presented framework additionally covers many more cases.
\end{remark}

\noindent Note that in \Cref{def:pHDAE} the input and output variables, usually taken as functions in $\cont(\timeinterval,\R^m)$, have the same size.
In fact, the product $y\transp u$ typically has the same physical dimension as power.
In particular, one can easily verify (see e.g.~\cite{mehrmann_2019_structurepreserving}) that every pHDAE of the form \eqref{eq:pHDAE} satisfies the \emph{power balance equation} (PBE)
\begin{equation}\label{eq:PBE}
	\ddt \hamiltonian\pset[\big]{x(t)} = -z\pset[\big]{\state(t)}\transp R\pset[\big]{\state(t)}z\pset[\big]{\state(t)} + y(t)\transp u(t)
\end{equation}
and the dissipation inequality
\begin{equation}\label{eq:dissipationInequality}
	\ddt \hamiltonian\pset[\big]{x(t)} \leq y(t)\transp u(t),
\end{equation}
along every solution $(\state,u,y)$ of \eqref{eq:pHDAE}, for all $t\in\timeinterval$.
Note that the PBE and the dissipation inequality can be reinterpreted in integral form as
\begin{equation}\label{eq:PBE_integral}
	\hamiltonian\pset[\big]{x(t_1)} - \hamiltonian\pset[\big]{x(t_0)} = \int_{t_0}^{t_1}\pset[\big]{-z(t)\transp R(t)z(t) + y(t)\transp u(t)}\textrm{d}t
\end{equation}
and
\begin{equation}\label{eq:dissipationInequality_integral}
	\hamiltonian\pset[\big]{x(t_1)} - \hamiltonian\pset[\big]{x(t_0)} \leq \int_{t_0}^{t_1}y(t)\transp u(t)\,\textrm{d}t
\end{equation}
respectively, for every $t_0,t_1\in\timeinterval,\ t_0\leq t_1$.

Since the PBE and dissipation inequality are fundamental properties satisfied by every pH system, there is much effort in the literature \cite{EggHS21,FroGLM24,GieKT24,kinon_2023_discrete,KotL19,mehrmann_2019_structurepreserving,schulze_structure_2023} in developing time-discretization schemes to preserve them on a discrete level. This is also be the focus of this paper.

\subsection{Discrete gradients}

\noindent Discrete gradients are a popular tool for generating structure-preserving integration methods for dynamical systems \cite{gonzalez_1996_time,hairer_2006_geometric,McLQR99}. A general definition is as follows.

\begin{definition}[Discrete gradients, see \cite{hairer_2006_geometric}] \label{def_discrete_grad}
	Given a function $f\in\cont^1(\statespace,\R)$ with $\statespace\subseteq\R^n$ open, a \emph{discrete gradient} for $f$ is any vector function $\DG f\in\cont(\statespace\times\statespace,\R^n)$ that satisfies the properties
	\begin{enumerate}[label=(\roman*)]
		\item \label{itm:directionality} $\DG f(\state,\state') \transp (\state' - \state) = f(\state') - f(\state)$ for all $x,x'\in\statespace$,
		\item \label{itm:consistency} $\DG f\left(\state,\state\right) = \gradient f\left(\state\right)$ for all $x\in\statespace$,
	\end{enumerate}
	where \ref{itm:directionality} is referred to as \textit{directionality} and \ref{itm:consistency} as \textit{consistency} condition.
\end{definition}

\noindent Especially the directionality property will be handy later on for the design of structure-preserving discretizations.
The following definition provides an example for a discrete gradient,
which can yield a symmetric method of second order accuracy, as it represents a second-order approximation to the exact gradients.

\begin{definition}[Gonzalez discrete gradient, see \cite{gonzalez_1996_time}]
	For a given function $f \in \cont^1(\statespace, \R) $ with $\statespace\subseteq\mathbb R^n$ convex open subset, its Gonzalez (or midpoint) discrete gradient $\DG f \in \cont(\statespace\times\statespace, \R^n) $ is defined by
	\begin{equation} \label{eq:DD-Gonzalez}
		\DG f(\state,\state') =
		\begin{dcases}
			\gradient f\left(\tfrac{\state + \state'}{2}\right) + \frac{f(\state') - f(\state) -  \gradient f \left(
			\tfrac{\state + \state'}{2}\right) \transp (\state' - \state)  }{ ||\state' - \state||^2 } (\state' - \state) & \quad \text{if} \ \state' \neq \state , \\
			\gradient f\left(\state\right)                                                                                & \quad \text{otherwise} .
		\end{dcases}
	\end{equation}
	Notably, the Gonzales discrete gradient is determined by the directionality condition together with its action on the orthogonal complement $(x'-x)^\perp$, that is,
	$
		\DG f(\state,\state')\transp z = \gradient f(\tfrac{\state+\state'}{2})\transp z
	$
	for all $\state,\state'\in\statespace$ and $z\in(\state'-\state)^\perp$.
\end{definition}
\noindent Note that for the special case of polynomial functions with degree of at most two, the Gonzalez discrete gradient is equivalent to a midpoint evaluation of the analytical gradient.
Next, the concept of discrete gradients may also be generalized to vector-valued functions.

\begin{definition}[Discrete Jacobians, see {\cite[Def.~3.3]{McLQR99}}] \label{def_discrete_jac}
	Given a vector-valued function $F\in\cont^1(\statespace,\R^m)$ with $\statespace\subseteq\R^n$ open, a \emph{discrete Jacobian} for $F$ is any matrix function $\discretejacobian F\in\cont(\statespace\times\statespace,\R^{m,n})$ that satisfies the directionality and consistency properties
	\begin{enumerate}[label=(\roman*)]
		\item $\discretejacobian F(\state,\state') (\state' - \state) = F(\state') - F(\state)$ for all $x,x'\in\statespace$,
		\item $\discretejacobian F\left(\state,\state\right) = \jacobian F\left(\state\right)$ for all $x\in\statespace$.
	\end{enumerate}
\end{definition}
\noindent As pointed out in \cite{McLQR99}, a discrete Jacobian $\discretejacobian F$ may be equivalently characterized by the fact that all of its rows are discrete gradients of the corresponding component functions of $F$.
In a similar notation as in \eqref{def_jacobian}, we write
\begin{equation*}
	\DG F =
	\begin{bmatrix}
		\DG_{x_1} F \\ \vdots \\ \DG_{x_r} F
	\end{bmatrix}, \qquad
	\discretejacobian{F} =
	\begin{bmatrix}
		\discretejacobian_{x_1}F & \cdots & \discretejacobian_{x_r}F
	\end{bmatrix}
\end{equation*}
for partial discrete derivatives and a partition $\state=(\state_1,\ldots,\state_r)$ of the state variable.
In particular, as long as $\statespace\subseteq\R^n$ is convex, we define the \emph{Gonzalez discrete Jacobian} of a differentiable vector field $F\in\cont^1(\statespace,\R^m)$ as
\begin{equation*}\label{eq:GonzalezJacobian}
	\discretejacobian{F}(\state,\state') =
	\begin{dcases}
		\jacobian{F}\pset*{\tfrac{\state+\state'}{2}} + \frac{F(\state')-F(\state)-\jacobian{F}\pset*{\tfrac{\state+\state'}{2}}(\state'-\state)}{\norm{\state'-\state}^2}(\state'-\state)\transp & \text{if }\state'\neq\state, \\
		\jacobian{F}(\state)                                                                                                                                                                      & \text{otherwise,}
	\end{dcases}
\end{equation*}
which is again determined by the directionality condition together with $\discretejacobian F(\state,\state')z=\jacobian F(\tfrac{\state+\state'}{2})z$ for all $\state,\state'\in\statespace$ and $z\in(\state'-\state)^\perp$.

\begin{remark}\label{rem:existence_of_DG}
	For the construction of classical discrete gradients or discrete Jacobians, some assumptions on the state space $\statespace$, like its convexity, are usually necessary.
	However, in general the existence of discrete gradients is actually independent from the structure of $\statespace$.
	For example, replacing $\gradient f(\frac{\state+\state'}{2})$ by $\gradient f(\state)$ or $\gradient f(\state')$ in \eqref{eq:DD-Gonzalez} yields a discrete gradient regardless of the structure of $\statespace$, although its usefulness for discretization is unclear.
\end{remark}

\noindent Let us come back to discrete gradients and observe the following property.

\begin{lemma}\label{lem:specifiedDG}
	Let $f\in\cont^1(\statespace,\R)$ with $\statespace=\statespace_1\times\statespace_2$, where $\statespace_1\subseteq\R^{n_1}$ and $\statespace_2\subseteq\R^{n_2}$ are open and $\statespace_2$ is convex, let us partition $x=(x_1,x_2)\in\R^{n_1}\times\R^{n_2}$, and suppose that $\gradient_{x_2}f=0$ holds everywhere in $\statespace$.
	Then there is $\specified{f}\in\cont^1(\statespace_1,\R)$ such that $\specified{f}(x_1)=f(x_1,x_2)$ and $\gradient\specified{f}(x_1)=\gradient_{x_1}f(x_1,x_2)$ for every $(x_1,x_2)\in\statespace$, or in short $\specified{f}\circ\pi_1=f$ and $\gradient\specified{f}\circ\pi_1=\gradient_{x_1}f$.
	Let now $\DG\specified{f}$ be a discrete gradient for $\specified{f}$ and $\DG f =(\DG\specified{f}\circ\pi_1,0) : \statespace\times\statespace \to \R^{n_1}\times\R^{n_2}$, i.e.,
	\[
		\DG f (x,x') =
		\begin{bmatrix}
			\DG\specified{f}\pset[\big]{x_1,x_1'} \\ 0
		\end{bmatrix}
	\]
	for every $x=(x_1,x_2),x'=(x_1',x_2')\in\statespace$. Then $\DG f$ is a discrete gradient for $f$.
\end{lemma}

\begin{proof}
	The interested reader is referred to \Cref{appendix_lemma2_7}
	for some detailed lines showing that there is $\specified{f}\in\cont^1(\statespace_1,\R)$ such that $\specified{f}(x_1)=f(x_1,x_2)$ and $\gradient\specified{f}(x_1)=\gradient_{x_1}f(x_1,x_2)$ for every $(x_1,x_2)\in\statespace$.
	We now show that $\DG f$ is a discrete gradient for $f$. In fact, for every $x=(x_1,x_2),x'=(x_1',x_2')\in\statespace$ it holds that
	\[
		\DG f(x,x) =
		\begin{bmatrix}
			\DG{\specified{f}}(x_1,x_1) \\ 0
		\end{bmatrix}
		=
		\begin{bmatrix}
			\gradient{\specified{f}}(x) \\ 0
		\end{bmatrix}
		=
		\begin{bmatrix}
			\gradient_{x_1}f(x) \\ \gradient_{x_2}f(x)
		\end{bmatrix}
		= \gradient{f}(x)
	\]
	and
	\begin{align*}
		\DG f(x,x')\transp(x'-x)  =
		\begin{bmatrix}
			\DG\specified{f}(x_1,x_1') \\ 0
		\end{bmatrix}\transp
		\begin{bmatrix}
			x_1'-x_1 \\ x_2'-x_2
		\end{bmatrix} & = \DG\specified{f}(x_1,x_1')\transp(x_1'-x_1)                             \\
		                     & = \specified{f}(x_1')- \specified{f}(x_1) = f(x') - f(x). \qedhere
	\end{align*}
\end{proof}

\begin{remark}\label{rem:specify}
	We note that the previous lemma is still true when replacing the assumption that $\statespace$ has the form $\statespace_1\times\statespace_2$ with convex $\statespace_2$ by the weaker assumption that there exist an open set $\widetilde\statespace\subseteq\R^{n_1}\times\R^{n_2}$ and a diffeomorphism $\varphi=(\varphi_1,\varphi_2):\widetilde\statespace\to\statespace$ such that $\pi_2(\widetilde\statespace)$ is convex and $\varphi_1:\pi_1(\widetilde\statespace)\to\pi_1(\statespace)$ is well-defined.
	However, in order to keep the setting simple, in this paper we will focus on the case where $\statespace=\statespace_1\times\statespace_2$ with convex $\statespace_2$, with the awareness that this setting can be extended.
	As we will discuss in \Cref{rem:local}, this assumption is not restrictive, as long as we are comfortable with restricting the state space $\statespace$ to appropriately small open neighborhoods and working locally, which is suitable for the goal of time discretization.
\end{remark}

\noindent Discrete gradients have been applied successfully to the time discretization of pH ODEs, see e.g.~\cite{celledoni2017energy,FroGLM24,kinon_2023_discrete}.
Here, we want to tackle pHDAEs as described in \Cref{def:pHDAE}.
This brings with it the striking challenge that the gradient of the Hamiltonian, which is supposed to be approximated with a discrete gradient, appears only implicitly within the relation \eqref{constitutive_equation} and is not directly part of the DAEs \eqref{eq:pHDAE}, which govern the dynamics of the system. Particularly for singular descriptor matrices, this leads to a non-invertible relation to the co-state function. In this context the recent work \cite{schulze_structure_2023} proposed the notion of \emph{discrete gradient pairs}, which we regard to be helpful throughout the present work.

\begin{definition}[Discrete gradient pair, see \cite{schulze_structure_2023}]
	\label{def:discGradPair}
	Let $(E,\costate)$ be a gradient pair for $\hamiltonian$.
	We call $(\overline{E},\overline{\costate})\in \cont(\statespace\times \statespace,\R^{n,n})\times \cont(\statespace\times \statespace,\R^n)$ a \emph{discrete gradient pair} for $(\hamiltonian,E,\costate)$ if the following conditions are satisfied.
	\begin{enumerate}[label=(\roman*)]
		\item \label{itm:discGradPH4}$\overline{\costate}(\state,\state')\transp\overline{E}(\state,\state')(\state'-\state) = \hamiltonian(\state')-\hamiltonian(\state)$ for all $(\state',\state)\in\statespace\times \statespace$,
		\item \label{itm:discGradPH2}$\overline{E}(\state,\state) = E(\state)$ for all $\state\in\statespace$,
		\item \label{itm:discGradPH3}$\overline{\costate}(\state,\state) = \costate(\state)$ for all $\state\in\statespace$.
	\end{enumerate}
\end{definition}
\noindent These conditions essentially yield that $\discreteE\transp\discretecostate$ is a discrete gradient, see \Cref{def_discrete_grad}.
Property \ref{itm:discGradPH4} can be interpreted as the directionality condition, while conditions \ref{itm:discGradPH2} and \ref{itm:discGradPH3} ensure the consistency condition for this specific discrete gradient.
As it has become obvious from the previous definitions in this section, the property of \emph{consistency} is rather crucial. We therefore make the following statement.
\begin{definition}\label{def:consistency}
	Given two functions $F \in \cont(\statespace,\R^n)$ and $\overline{F} \in \cont(\statespace\times \statespace,\R^n)$, we call $\overline{F}$ a \emph{consistent} approximation or discretization of $F$ if
	\begin{equation}
		\overline{F}(\state,\state) = F(\state) \quad  \text{for all} \ \state\in\statespace .
	\end{equation}
\end{definition}

\begin{remark}\label{rem:local}
	Assuming the convexity of $\statespace$ is in practice not restrictive.
	Since discrete gradients and other consistent approximations are used for time discretization, it can be usually assumed that they will only be evaluated for arbitrarily close $\state,\state'\in\statespace$, up to reducing the time step accordingly.
	Then, for every $\state\in\statespace$ we can restrict them to $\statespace_0\times\statespace_0$, where $\statespace_0$ is an appropriate open neighborhood of $\state$, which can be selected to have even stronger structure, like being a ball for the $\infty$-norm on $\R^n$.
	This choice in particular ensures that $\statespace_0$ is convex and can be written in the form $\statespace_1\times\statespace_2$ for every partition of the state variable $\state=(\state_1,\state_2)$.
\end{remark}

\noindent Having discussed basic notions of pHDAEs and discrete gradients, we stress that a special class of pHDAEs is pivotal in this work, see the upcoming section.

\section{Semi-explicit port-Hamiltonian descriptor systems}
\label{sec:semi-explicitPHDAEs}
\noindent We now specify that the pHDAE under investigation is semi-explicit
. This subclass already covers many applications and will be the
starting point
pivotal
for derivations of corresponding time integration methods.
We start by introducing the related concept of semi-explicit gradient pairs.

\begin{definition}
	\label{def:semi-explicit_gradient_pair}
	Let $(E,\costate)$ be a gradient pair for $\hamiltonian$.
	We say that $(E,\costate)$ is \emph{semi-explicit} if $\statespace=\statespace_1\times\statespace_2$ with $\statespace_1\subseteq\R^{n_1}$ and $\statespace_2\subseteq\R^{n_2}$ open and $\statespace_2$ convex 
	, and $E=\diag(E_{11},0)$ for some pointwise invertible matrix function $E_{11}\in\cont(\statespace,\R^{n_1,n_1})$.
\end{definition}
\noindent Semi-explicit gradient pairs  satisfy the following property.

\begin{lemma}\label{lem:semiExplicitGradientPair}
	Let $(E,\costate)$ be a semi-explicit gradient pair for $\hamiltonian$. Then there exists $\specified\hamiltonian\in\cont^1(\statespace_1,\R)$ such that $\specified{\hamiltonian}\circ\pi_1=\hamiltonian$ and $\gradient\specified\hamiltonian\circ\pi_1=\gradient_{\state_1}\hamiltonian$, i.e.,
	\begin{align}\label{split_hamiltonian}
		\specified{\hamiltonian}(\state_1) =  \hamiltonian(\state_1,\state_2), \qquad
		\gradient{\specified{\hamiltonian}}(\state_1) = \gradient_{\state_1}\hamiltonian(\state_1,\state_2)
	\end{align}
	for all $\state_1\in\statespace_1,\ \state_2\in\statespace_2$.
	In particular, the gradient pair property \eqref{constitutive_equation} is determined by
	\begin{equation}\label{eq:semiExplicitGradientPair}
		\gradient\specified\hamiltonian(\state_1) = E_{11}(\state_1,\state_2)\transp \costate_1(\state_1,\state_2),
	\end{equation}
	for all $\state_1\in\statespace_1,\ \state_2\in\statespace_2$, where $\costate=(\costate_1,\costate_2)$ is the corresponding partition of $\costate$.
\end{lemma}

\begin{proof}
	Due to the structure of $E$, the gradient pair property \eqref{constitutive_equation} can be written as
	$\gradient_{\state_1}\hamiltonian = E_{11}\transp\costate_1,\ \gradient_{\state_2}\hamiltonian = 0$.
	The latter equation implies with \Cref{lem:specifiedDG} that there exists $\hamiltonian_1\in\cont^1(\statespace_1,\R)$ satisfying \eqref{split_hamiltonian}, while the former is immediately reinterpreted as \eqref{eq:semiExplicitGradientPair}.
\end{proof}

\noindent In the context of pHDAEs, we often call $\hamiltonian_1$ the \emph{specified Hamiltonian}.
This motivates the following definition.

\begin{definition}\label{def:semiexp_pHDAE}
	Consider
	a state space $\statespace=\statespace_1\times\statespace_2\subseteq\R^n$ with $\statespace_1\subseteq\R^{n_1}$ open and $\statespace_2\subseteq\R^{n_2}$ open convex, and let us partition the state $\state=(\state_1,\state_2)\in\statespace$ accordingly.
	A \emph{semi-explicit pHDAE} is a port-Hamiltonian descriptor system in the sense of \Cref{def:pHDAE} with $E=\diag(E_{11},0)$, where $E_{11}\in\cont(\statespace,\R^{n_1,n_1})$ is pointwise invertible. In particular it admits the form
	\begin{equation} \label{block_pHDAE}
		\begin{split}
			\begin{bmatrix}
				E_{11}(\state) & 0 \\ 0 & 0
			\end{bmatrix} \begin{bmatrix}
				              \dot{\state}_1 \\ \dot{\state}_2
			              \end{bmatrix} & = \left( J(\state) - R(\state )\right) \begin{bmatrix}
				                                                                     \costate_1(\state) \\ \costate_2(\state)
			                                                                     \end{bmatrix} + B(x) u , \\
			y                                & = B(\state)\transp
			\begin{bmatrix}
				\costate_1(\state) \\ \costate_2(\state)
			\end{bmatrix} ,
		\end{split}
	\end{equation}
	together with a specified Hamiltonian $\hamiltonian_1\in\cont^1(\statespace_1,\R)$ that satisfies the gradient pair property \eqref{eq:semiExplicitGradientPair} and conforms with \Cref{lem:semiExplicitGradientPair}.
\end{definition}

\noindent Note that systems of the form \eqref{block_pHDAE} have also been considered in \cite{morandin_phd_2019}, where the application of partitioned Runge-Kutta schemes for their time discretization was considered.
We now illuminate the abovementioned definition by exploring some examples.

\begin{myex}[Constrained input-output pH systems in classical form]\label{ex_constrained_IO_PHS}
	The above framework naturally includes all systems which are covered by the standard notion of pH systems in \emph{constrained input-output representation} (see e.g. \cite[Eq.~2.154]{duindam_2009_modeling} or \cite[Eq.~4.44]{vanderschaft_2013_porthamiltonian}) described by local coordinates $\tilde{x}$ satisfying
	\begin{equation*}
		\begin{aligned}
			\dot{\tilde x} & = \left( \widetilde{J}(\tilde{x}) - \widetilde{R}(\tilde{x}) \right) \gradient \widetilde{\hamiltonian}(\tilde{x}) + g(\tilde{x}) u + b(\tilde{x}) \lambda, \\
			y              & = g(\tilde{x})\transp \gradient \widetilde{\hamiltonian}(\tilde{x}),                                                                                        \\
			0              & = b(\tilde{x})\transp \gradient \widetilde{\hamiltonian}(\tilde{x}),
		\end{aligned}
	\end{equation*}
	with $\state = (\state_1,\state_2) = (\tilde{x},\lambda)$, $\specified{\hamiltonian}(x_1) = \widetilde{\hamiltonian}(\tilde{x})$, $E_{11} = I$, $z_1= \gradient \widetilde{\hamiltonian}(\tilde{x})$, $z_2 =\lambda$, $B(x)\transp = [g(\tilde{x})\transp, 0]$ and
	\[
		J(x) -R(x) = \begin{bmatrix}
			\widetilde{J}(\tilde{x}) - \widetilde{R}(\tilde{x}) &  & b(\tilde{x}) \\
			-b(\tilde{x})\transp                                &  & 0
		\end{bmatrix} .
		\qedhere
	\]
\end{myex}

\begin{myex}[Nonlinear multibody systems]\label{ex_mbs_shorter}
	It can be shown that the governing equations for nonlinear multibody systems fit well into the above framework of semi-explicit pHDAEs. The equations of motion are given as
	\begin{align*} 
		\begin{bmatrix}
			I & 0 & 0 \\ 0 & M & 0 \\ 0 & 0 & 0
		\end{bmatrix} \begin{bmatrix}
			              \dot{q} \\ \dot{v} \\ \dot{\lambda}
		              \end{bmatrix} & = \left( \begin{bmatrix}
			                                       0  & I                  & 0                      \\
			                                       -I & -R_{\mathrm{R}}(q) & -\jacobian g(q)\transp \\ 0 & \jacobian g(q) & 0
		                                       \end{bmatrix} \right) \begin{bmatrix}
			                                                             \gradient V(q) \\ v \\ \lambda
		                                                             \end{bmatrix} + \begin{bmatrix}
			                                                                             0 \\ I \\ 0
		                                                                             \end{bmatrix} u , \\
		y                                   & = \begin{bmatrix}
			                                        0 & I & 0
		                                        \end{bmatrix} \begin{bmatrix}
			                                                      \gradient V(q) \\ v \\ \lambda
		                                                      \end{bmatrix} .
	\end{align*}
	The Hamiltonian
	\[
		\hamiltonian(\state) = \frac{1}{2} v\transp M v + V(q) = T(v) + V(q)
	\]
	denotes the total energy.
	Verifying that $E\transp\costate(\state) = \gradient \hamiltonian(\state)$ holds true is straightforward. For more details, especially concerning an introduction of the unknowns, see \Cref{modelling_mbs}.
\end{myex}

\begin{myex}[Synchronous machine]\label{ex_synchro}
	Let us consider a synchronous machine, modeled as described e.g.~in \cite{kundur_1994_power}, and interpreted as a pH system like in \cite{fiaz_2013_porthamiltonian}.
	After a change of variables, which is detailed in \Cref{appendix_synchro}, such that we obtain $x=(I,p,\theta) \in \R^8$, the governing equations can be found in a suitable representation
	\begin{subequations}
		\begin{align}\label{eq:syncMachAlt_dyn}
			\begin{bmatrix}
				L(\theta) & 0 & L'(\theta)I \\ 0 & 1 & 0 \\ 0 & 0 & 1
			\end{bmatrix}
			\begin{bmatrix}
				\dot{I} \\ \dot p \\ \dot\theta
			\end{bmatrix}
			 & =
			\begin{bmatrix}
				-R_{s,r} & 0 & 0 \\ 0 & -d & -1 \\ 0 & 1 & 0
			\end{bmatrix}
			\begin{bmatrix}
				I         \\
				J_r^{-1}p \\
				\frac{1}{2}I^\top L'(\theta)I
			\end{bmatrix}
			+
			\begin{bmatrix}
				I_3 & 0 & 0 \\ 0 & e_1 & 0 \\ 0 & 0 & 1 \\ 0 & 0 & 0
			\end{bmatrix}
			\begin{bmatrix}
				V_s \\ V_f \\ \tau
			\end{bmatrix}, \\ \label{eq:syncMachAlt_out}
			\begin{bmatrix}
				I_s \\ I_f \\ \omega
			\end{bmatrix}
			 & =
			\begin{bmatrix}
				I_3 & 0 & 0 & 0 \\ 0 & e_1\transp & 0 & 0 \\ 0 & 0 & 1 & 0
			\end{bmatrix}
			\begin{bmatrix}
				I         \\
				J_r^{-1}p \\
				\frac{1}{2}I^\top L'(\theta)I
			\end{bmatrix}  .
		\end{align}
	\end{subequations}
	Here $e_1\in\R^{3}$ denotes the first vector of the standard basis of $\R^3$,
	$I\in\R^6$ contains the current in the stator and rotor
	, $p \in \R$ represents the angular momentum of the rotor, $\theta \in \R$ the angle of the rotor, and $R_{s,r}\coloneqq\diag(R_s,R_r)\succ 0$, where $R_s,R_r\in\R^{3,3}$ are positive diagonal matrices representing the stator and rotor resistances. Additionally, $d>0$ is the mechanical friction, $V_s,I_s\in\R^3$ are the three-phase stator terminal voltages and currents, $V_f,I_f\in\R$ are the rotor field winding voltage and current, $\tau,\omega\in\R$ are the mechanical torque and angular velocity, $J_r>0$ is the rotational inertia of the rotor, $L:\R\to\R^{6,6}$ is the inductance matrix, usually assumed to be $\cont^\infty$, pointwise symmetric positive definite, and periodic of period $2\pi$, and $L'$ denotes its first derivative.
	Note that $V_s,V_f,\tau$ are interpreted as the input variables of the system, while $I_s,I_f,\omega$ as the corresponding output variables.
	The system is completed by the Hamiltonian
	\begin{equation*} \label{eq_syncMach_Ham}
		\hamiltonian(I,p,\theta) = \frac{1}{2}I\transp L(\theta)I + \frac{1}{2J_r}p^2 ,
	\end{equation*}
	which easily verifies $E(\state)\transp \costate(\state) = \gradient \hamiltonian(\state)$.

	While one might argue that \eqref{eq:syncMachAlt_dyn} is not really a DAE, since $E$ is pointwise invertible, this representation has potential advantages. For example, the inductance matrix $L(\theta)$ does not appear under inversion, unlike in the original example from \Cref{appendix_synchro}.
	Furthermore, synchronous machines can be components in complex interconnected systems, e.g.~in the modeling of power networks, typically resulting in actual DAEs anyway due to the application of Kirchhoff's laws.
\end{myex}
\noindent In the upcoming section we focus on the discretization of pHDAEs - as discussed both in \Cref{def:pHDAE,def:semiexp_pHDAE}.

\section{Structure-preserving time discretization} \label{sec_methods}

\noindent We start by discussing integration methods for general pHDAEs of the form \eqref{eq:pHDAE} in \Cref{sec_version2}.
Here the concept of discrete gradient pairs will be of central importance.
We continue with the discretization of semi-explicit pHDAEs of the form \eqref{block_pHDAE} in \Cref{sec_version1}, yielding a tangible time stepping method. Lastly, we discuss an alternative approach based on a different modeling ansatz, see \Cref{sec_version3}.

\subsection{Discrete gradient pair methods for general pHDAEs} \label{sec_version2}

\noindent Consider a pHDAE of the form \eqref{eq:pHDAE} and a temporal grid $0=t^0<t^1<\ldots<t^\ntimesteps=\finaltime$ with $N$ time intervals of constant time step size $h = t\none - t\n$ for $k=0,\ldots,\ntimesteps-1$. We consider uniform time grids for the sake of brevity
and propose the scheme
\begin{equation} \label{block_pHDAE_timestepping}
	\begin{aligned}
		\discreteE(\discretestate\n,\discretestate\none) (\discretestate\none - \discretestate\n) & = h \pset[\big]{ \discreteJ(\discretestate\n,\discretestate\none)-\discreteR(\discretestate\n,\discretestate\none)} \discretecostate(\discretestate\n,\discretestate\none) + h\discreteB(\discretestate\n,\discretestate\none) \discreteinput , \\
		\discreteoutput                                                                           & = \discreteB(\discretestate\n,\discretestate\none)\transp \discretecostate(\discretestate\n,\discretestate\none) .
	\end{aligned}
\end{equation}
for $k = 0, \ldots, \ntimesteps -1$.

In \eqref{block_pHDAE_timestepping}, we define discrete approximations of the state
$\discretestate\n \approx \state(t\n)$
assuming that also $\state\n \in \statespace$ for sufficiently small time steps. The matrices $\discreteE, \discreteJ, \discreteR, \discreteB$ are arbitrary consistent approximations of the matrix functions (see \Cref{def:consistency}), still satisfying $\discreteJ=-\discreteJ\transp$ and $\discreteR=\discreteR\transp \succeq 0$ pointwise.
Moreover, $\discreteinput$ is not necessarily the evaluation of the (possibly discontinuous) input function at $t\n$, but at some point within the time interval of interest or an average value of it.
Correspondingly, the discrete-time output $\discreteoutput$ is an approximation for $y(t)$  for the whole time step interval.

Most importantly,
we require that $(\overline{E},\overline{\costate})$ is a discrete gradient pair for $(\hamiltonian,E,\costate)$ in the sense of \Cref{def:discGradPair}. Finding such a discrete gradient pair is not trivial, but we will study how to construct one in certain special cases in \Cref{sec_connections}.
For self-containedness of this work, we show that the usage of discrete gradient pairs yields an energy-consistent time integration.
\begin{theorem}\label{theorem_block_energy_general}
	Scheme \eqref{block_pHDAE_timestepping} yields an energy-consistent approximation of the time-continuous power balance \eqref{eq:PBE} given by
	\begin{equation}
		\begin{aligned}
			{\hamiltonian}(\discretestate\none) - {\hamiltonian}(\discretestate\n) & = - h\discretecostate(\discretestate\n,\discretestate\none) \transp \discreteR(\discretestate\n,\discretestate\none)  \discretecostate(\discretestate\n,\discretestate\none)  + h (\discreteoutput)\transp \discreteinput \\
			                                                                       & \leq h (\discreteoutput)\transp \discreteinput .
		\end{aligned}
	\end{equation}
\end{theorem}
\begin{proof}
	Combining the directionality property \ref{itm:discGradPH4} of the discrete gradient pair with \eqref{block_pHDAE_timestepping} one obtains
	\begin{align*}
		{\hamiltonian}(\discretestate\none) - {\hamiltonian}(\discretestate\n)
		 & = \discretecostate(\discretestate\n,\discretestate\none)\transp \discreteE(\discretestate\n,\discretestate\none)  (\discretestate\none - \discretestate\n)
		\\
		 & = h \discretecostate(\discretestate\n,\discretestate\none)\transp \pset[\big]{\discreteJ(\discretestate\n,\discretestate\none)-\discreteR(\discretestate\n,\discretestate\none)} \discretecostate(\discretestate\n,\discretestate\none) \\
		 &
		\qquad + h \discretecostate(\discretestate\n,\discretestate\none)\transp \discreteB(\discretestate\n,\discretestate\none) \discreteinput                                                                                                   \\
		 &
		= - h \discretecostate(\discretestate\n,\discretestate\none) \transp \discreteR(\discretestate\n,\discretestate\none) \discretecostate(\discretestate\n,\discretestate\none) + h (\discreteoutput)\transp \discreteinput \leq h (\discreteoutput)\transp \discreteinput,
	\end{align*}
	which is the desired result.
\end{proof}

\subsection{Discrete gradient method for semi-explicit pHDAEs}\label{sec_version1}

\noindent Consider now a semi-explicit pHDAE of the form \eqref{block_pHDAE}.
Since the partitioning of the state and the block matrix structure allow for a straightforward approach using discrete gradients, we will obtain a concrete time stepping method in this section.
Particularly, the semi-explicit gradient pair property \eqref{eq:semiExplicitGradientPair}
allows for a direct approximation of $\costate_1$ in terms of the specified Hamiltonian.
Essentially, the proposed method can be written just like equations \eqref{block_pHDAE_timestepping}, which have to be completed by the additional constraint
\begin{equation}\label{block_pHDAE_const_discrete}
	\discreteE_{11}(\state\n,\state\none)\transp \discretecostate_1(\state\n,\state\none) = \DG\specified\hamiltonian(\state\n,\state\none).
\end{equation}
We now choose $\DG\specified\hamiltonian\in\cont(\statespace_1\times\statespace_1,\R^{n_1})$ to be a discrete gradient of the specified Hamiltonian $\specified\hamiltonian$ and $\discreteE_{11}\in\cont(\statespace\times\statespace,\R^{n_1,n_1})$ to be a consistent discretization of $E_{11}$.
This allows to determine uniquely $\discretecostate_1$ as a function of $\state\n,\state\none$, as long $\discreteE_{11}$ is ensured to be invertible within our search scope.

This is for example the case if we choose the midpoint approximation $\discreteE_{11}(\state,\state')\coloneqq E_{11}(\frac{\state+\state'}{2})$, since $E_{11}$ is invertible in the convex space $\statespace$.
More in general, any consistent approximation $\discreteE_{11}$ will be invertible for sufficiently close $\discretestate\n,\discretestate\none$.

While there is in general no guarantee that a discrete matrix function is pointwise invertible, cf.~\Cref{ex:counterexampleForInvertibleDiscreteJacobianOfADiffeomorphism} in the appendix,
we expect to achieve this condition up to refining the time grid sufficiently.
For the sake of simplicity, we introduce the following assumption.
\begin{equation}\label{ass:invertible_discreteE11}\tag{A1}
	\text{$\discreteE_{11}$ is pointwise invertible on $\statespace\times\statespace$}.
\end{equation}
Furthermore, since $\discretecostate=(\discretecostate_1,\discretecostate_2)$ is not given as part of a gradient pair anymore,
we will choose $\discretecostate_2$ as a consistent discretization of the time-continuous function $\costate_2$.

In a more detailed fashion, also highlighting the partitioned state, we rewrite \eqref{block_pHDAE_timestepping} and \eqref{block_pHDAE_const_discrete} combined as
\begin{subequations} \label{block_pHDAE_timestepping_2}
	\begin{align}
		\begin{bmatrix}
			\, \discreteE_{11}(\discretestate\n,\discretestate\none) & 0 \\ 0 & 0
		\end{bmatrix}
		\begin{bmatrix}
			\discretestate_1\none - \discretestate_1\n \\
			\discretestate_2\none - \discretestate_2\n
		\end{bmatrix}
		                                                                                 & = h \pset[\big]{\discreteJ(\discretestate\n, \discretestate\none)-\discreteR(\discretestate\n, \discretestate\none)}
		\begin{bmatrix}
			\discretecostateone \\ \discretecostate_2(\discretestate\n,\discretestate\none)
		\end{bmatrix} \notag                                                                                                                                               \\
		                                                                                 & \qquad + h\discreteB(\discretestate\n, \discretestate\none) \discreteinput ,                                                              \\
		\discreteoutput                                                                  & = \discreteB(\discretestate\n, \discretestate\none)\transp \begin{bmatrix}
			                                                                                                                                              \discretecostateone \\ \discretecostate_2(\discretestate\n,\discretestate\none)
		                                                                                                                                              \end{bmatrix} ,
		\\
		\discreteE_{11}(\discretestate\n, \discretestate\none)\transp\discretecostateone & =  \DG \specified{\hamiltonian}(\discretestate_1\n, \discretestate_1\none) ,  \label{block_pHDAE_timestepping_2_4}
	\end{align}
\end{subequations}
which can be solved for the unknowns $(\discretestate_1\none, \discretestate_2\none, \discretecostateone, \discreteoutput)$ in each time step (assuming that a solution exists).
Note that $\discretecostateone$, which here replaces the uniquely determined function $\discretecostate_1$, is considered as an unknown of the time-discrete system, whereas $\discretecostate_2$ is a
consistent discretization of $\costate_2$.
This scheme extends the discrete gradient method from \cite{kinon_2023_porthamiltonian,kinon_2023_discrete} to semi-explicit pHDAE systems with the specific structure of the descriptor matrix $E$.

\begin{theorem}\label{theorem_block_energy}
	Scheme \eqref{block_pHDAE_timestepping_2} yields an energy-consistent approximation of the time-continuous power balance \eqref{eq:PBE} given by
	\begin{equation}
		\begin{aligned}
			\hamiltonian(\discretestate\none) - \hamiltonian(\discretestate\n) & = - h
			\begin{bmatrix}
				\discretecostateone \\ \discretecostate_2(\discretestate\n,\discretestate\none)
			\end{bmatrix}
			\transp\discreteR(\discretestate\n,\discretestate\none)
			\begin{bmatrix}
				\discretecostateone \\ \discretecostate_2(\discretestate\n,\discretestate\none)
			\end{bmatrix}
			+ h (\discreteoutput)\transp \discreteinput                                \\ & \leq h (\discreteoutput)\transp \discreteinput .
		\end{aligned}
	\end{equation}

\end{theorem}
\begin{proof}
	Combining the directionality property of the discrete gradient $\DG\specified{\hamiltonian}$ with \eqref{block_pHDAE_timestepping_2} one obtains
	\begin{align*}
		\hamiltonian(\discretestate\none) - \hamiltonian(\discretestate\n) & =  \specified{\hamiltonian}(\discretestate_1\none) - \specified{\hamiltonian}(\discretestate_1\n) = \DG\specified{\hamiltonian}(\discretestate\n,\discretestate\none) \transp (\discretestate\none-\discretestate\n)                                \\
		                                                                   & = (\discretecostateone)\transp \discreteE_{11}(\discretestate\n,\discretestate\none)  (\discretestate_1\none - \discretestate_1\n)	= \begin{bmatrix}
			                                                                                                                                                                                                          \discretecostateone \\ \discretecostate_2(\discretestate\n,\discretestate\none)
		                                                                                                                                                                                                          \end{bmatrix}\transp \discreteE(\discretestate\n,\discretestate\none) (\discretestate\none - \discretestate\n) \\
		                                                                   & = h\begin{bmatrix}
			                                                                        \discretecostateone \\ \discretecostate_2(\discretestate\n,\discretestate\none)
		                                                                        \end{bmatrix}\transp \pset[\big]{\discreteJ(\discretestate\n,\discretestate\none)-\discreteR(\discretestate\n,\discretestate\none)} \begin{bmatrix}
			                                                                                                                                                                                                            \discretecostateone \\ \discretecostate_2(\discretestate\n,\discretestate\none)  \end{bmatrix}               \\
		                                                                   & \qquad + h \begin{bmatrix} \discretecostateone \\ \discretecostate_2(\discretestate\n,\discretestate\none) \end{bmatrix}\transp \discreteB(\discretestate\n,\discretestate\none) \discreteinput                                                     \\  & = - h \begin{bmatrix}
			\discretecostateone \\ \discretecostate_2(\discretestate\n,\discretestate\none) \end{bmatrix}\transp \discreteR(\discretestate\n,\discretestate\none) \begin{bmatrix}  \discretecostateone \\ \discretecostate_2(\discretestate\n,\discretestate\none)   \end{bmatrix} + h (\discreteoutput)\transp \discreteinput \\
		                                                                   & \leq h (\discreteoutput)\transp \discreteinput,
	\end{align*}
	which is the desired result.
\end{proof}

\noindent The semi-explicit discrete gradient method introduced in this section can of course be applied to constrained input-output pH systems as introduced in \Cref{ex_constrained_IO_PHS}, always achieving the desired exact PBE.
While there is in general no guarantee that the algebraic constraints are satisfied exactly by the discrete solution, specific implementation choices may allow to
meet additional requirements stemming from the particular application problem. This is shown for the example of nonlinear multibody systems in \Cref{sec_mbs_integrator}.

\subsection{A third approach based on the Dirac-dissipative representation}\label{sec_version3}

\noindent Under the provision that a newly introduced quantity $f$ satisfies $f=z(x)$, a pHDAE determined by
\eqref{constitutive_equation} and \eqref{eq:pHDAE}
can be given in terms of its \emph{Dirac-dissipative representation} (shorthand \emph{DDR}-pHDAE, see \cite{morandin_phd_2019}) governed by the equations
\begin{equation}
	\label{eq_DDR}
	\begin{bmatrix}
		\gradient \hamiltonian(x) \\
		0                         \\
		y
	\end{bmatrix}+
	\begin{bmatrix}
		0    &  & -E(x)\transp        &  & 0         \\
		E(x) &  & J(\state)-R(\state) &  & B(\state) \\
		0    &  & -B(\state)\transp   &  & 0
	\end{bmatrix}
	\begin{bmatrix}
		-\dot{\state} \\
		f             \\
		u
	\end{bmatrix} = 0 .
\end{equation}
Now, we discuss a method that
can be linked to
previous works for ODE systems \cite{kinon_2023_porthamiltonian,kinon_2024_generalized,kinon_2023_discrete}. Given a DDR-pHDAE \eqref{eq_DDR}, the DDR-method
governs time-stepping via
	{
		\small\setlength\arraycolsep{2pt}
		\begin{equation}
			\label{eq_DDR_discrete}
			\begin{bmatrix}
				\, \DG \hamiltonian(\discretestate\n, \discretestate\none) \\
				0                                                          \\
				\discreteoutput
			\end{bmatrix}+
			\begin{bmatrix}
				0                                                &  & - \discreteE(\discretestate\n,\discretestate\none)\transp                                         &  & 0                                                \\
				\discreteE(\discretestate\n,\discretestate\none) &  & \discreteJ(\discretestate\n,\discretestate\none)-\discreteR(\discretestate\n,\discretestate\none) &  & \discreteB(\discretestate\n,\discretestate\none) \\
				0                                                &  & -\discreteB(\discretestate\n,\discretestate\none) \transp                                         &  & 0
			\end{bmatrix}
			\begin{bmatrix}
				-\frac{1}{h}(\discretestate\none - \discretestate\n) \\
				\discretetimestep{f}                                 \\
				\discreteinput
			\end{bmatrix} = 0 .
		\end{equation}
	}
or written out
\begin{subequations}
	\begin{equation} \label{pHDAE_timestepping}
		\begin{aligned}
			\discreteE(\discretestate\n,\discretestate\none) (\discretestate\none - \discretestate\n) & = h \pset[\big]{ \discreteJ(\discretestate\n,\discretestate\none)-\discreteR(\discretestate\n,\discretestate\none) }  \discretetimestep{f}     + h\discreteB(\discretestate\n,\discretestate\none) \discreteinput , \\
			\discreteoutput                                                                           & = \discreteB(\discretestate\n,\discretestate\none)\transp  \discretetimestep{f}    ,
		\end{aligned}
	\end{equation}
	as well as
	\begin{align} \label{eqn_colsp}
		\discreteE(\discretestate\n,\discretestate\none) \transp  \discretetimestep{f} & = \DG \hamiltonian(\discretestate\n, \discretestate\none) .
	\end{align}
\end{subequations}
Therein, $ \discretetimestep{f}    $ are discrete-time approximations of the co-state quantities.
Additionally, we have borrowed definitions from \Cref{sec_version2} concerning the discrete state and matrices.

\begin{theorem}
	Scheme \eqref{eq_DDR_discrete} yields an energy-consistent approximation of the time-continuous power balance \eqref{eq:PBE} given by
	\begin{equation}\label{eq:discDDR_PBE}
		\hamiltonian(x\none) - \hamiltonian(x\n)
		= -h(\discretetimestep{f})^\top\discreteR(x\n,x\none)\discretetimestep{f} + h(\discreteoutput)^\top\discreteinput
		\leq h(\discreteoutput)^\top\discreteinput.
	\end{equation}
\end{theorem}

\begin{proof}
	The equation is obtained by left-multiplying \eqref{eq_DDR_discrete} with $[-\frac{1}{h}(x\none-x\n)^\top, (\discretetimestep{f})^\top, (\discreteinput)^\top]^\top$.
	The inequality immediately follows from $\discreteR\succeq 0$ holding pointwise.
\end{proof}

\noindent It is in general unclear whether these equations can be solved simultaneously for the unknowns $(\discretestate\none$, $\discretetimestep{f}$, $\discreteoutput)$ in each time step.
While for pointwise invertible $\discreteE$ one can at least recover $\discretetimestep{f}$ as a function of $\discretestate\none$, and rewrite the discrete system only in terms of $\state\n$ and $\state\none$,
for an arbitrary, non-invertible $\discreteE$ further
analysis is
required.
In particular, \eqref{pHDAE_timestepping} could be underdetermined even if the original DAE was regular, thus it might be necessary to introduce additional constraints. We show this with an example.

\begin{myex}\label{exm:DDR}
	Consider the regular linear semi-explicit pHDAE of index 1 given by
	\begin{equation}\label{eq:exm:DDR:cont}
		\begin{bmatrix}
			1 & 0 \\ 0 & 0
		\end{bmatrix}
		\begin{bmatrix}
			\dot x_1 \\ \dot x_2
		\end{bmatrix}
		=
		\begin{bmatrix}
			0 & 1 \\ -1 & -1
		\end{bmatrix}
		\begin{bmatrix}
			x_1 \\ x_2
		\end{bmatrix},
	\end{equation}
	together with its DDR \eqref{eq_DDR}, which reads
	\begin{equation}\label{eq:exm:DDR:contDDR}
		\begin{bmatrix}
			x_1 \\ 0 \\ 0 \\ 0
		\end{bmatrix}
		+
		\begin{bmatrix}
			0 & 0 & -1 & 0 \\ 0 & 0 & 0 & 0 \\ 1 & 0 & 0 & 1 \\ 0 & 0 & -1 & -1
		\end{bmatrix}
		\begin{bmatrix}
			-\dot x_1 \\ -\dot x_2 \\ f_1 \\ f_2
		\end{bmatrix}
		= 0
	\end{equation}
	with $f=\costate(\state)$. The corresponding Hamiltonian is given by $\hamiltonian(\state) = \specified{\hamiltonian}(\state_1) = \frac{1}{2} \state_1^2$.
	Discretizing \eqref{eq:exm:DDR:contDDR} with the Gonzalez discrete gradient and discarding the trivial parts of the equation yields
	\begin{equation}\label{eq:exm:DDR:discDDR}
		\discretetimestep{f_1} = -\discretetimestep{f_2} = \frac{\discretestate_1\n + \discretestate_1\none}{2}, \qquad
		\discretestate_1\none = \discretestate_1\n - h \frac{\discretestate_1\n + \discretestate_1\none}{2},
	\end{equation}
	which is equivalent to reducing \eqref{eq:exm:DDR:cont} to $\dot\state_1=-\state_1$ and solving this subsystem with the implicit midpoint method.
	However, $\discretestate_2$ remains undetermined, since the DDR-method discarded the connection between $f_2$ and $\state_2$.

	One possible solution is to use our original knowledge from \eqref{eq:exm:DDR:cont} and define $\discretestate_2\none=-\discretestate_1\none$, conforming with the algebraic condition $\state_1+\state_2=0$.
	Another possibility would be to observe that $(f_1,f_2)=\costate(\state)=(\state_1,\state_2)$ and define $\discretestate_2\none=\discretetimestep{f_2}=-\frac{1}{2}(\discretestate_1\n+\discretestate_1\none)$.
	Note that both these ideas are based on a priori knowledge of the equation structure.

	A more robust and generalizable approach would be to select a consistent discretization $\discretecostate$ for $\costate$. Since imposing $\discretetimestep{f}=\discretecostate(\discretestate\n,\discretestate\none)$ might make the system overdetermined, we choose $(\discretestate\none,\discretetimestep{f})$ instead so that it satisfies \eqref{eq:exm:DDR:discDDR} while minimizing $\norm{\discretetimestep{f}-\discretecostate(\discretestate\n,\discretestate\none)}$.
	We investigate three different choices for $\discretecostate$:
	\begin{enumerate}[label=(\roman*)]
		\item If $\discretecostate(\discretestate\n,\discretestate\none)=\discretestate\none$, then we obtain again $\discretestate_2\none=\discretetimestep{f_2}=-\frac{1}{2}(\discretestate_1\n+\discretestate_1\none)$ for all $k\geq 0$.
		\item If $\discretecostate(\discretestate\n,\discretestate\none)=\discretestate\n$, then $\discretestate_2\none$ does not appear in $\norm{\discretetimestep{f}-\discretecostate(\discretestate\n,\discretestate\none)}$. However, since $\discretestate_2\n$ appears, we obtain $\discretestate_2\n=\discretetimestep{f_2}=-\frac{1}{2}(\discretestate_1\n+\discretestate_1\none)$ for $1\leq k\leq N-1$ (and additionally $k=0$ if we allow to redefine $\state_2^0$).
		      Note that in this case $\state_2\n$ is to be computed after $\state_1\none$, since the iteration defining $\state_1\none$ is independent of $\state_2\n$. However, the final state $\state_2^N$ remains undefined. In fact, this definition suggests that $\state_2\n$ actually approximates $\state_2(t\n+\frac{h}{2})$ instead of $\state_2(t\n)$, thus justifying redefining $\state_2^0$ and stopping at $\state_2^{N-1}$.
		\item If $\discretecostate(\discretestate\n,\discretestate\none)=\frac{1}{2}(\discretestate\n+\discretestate\none)$, then we obtain $\discretestate_2\none=-\discretestate_1\n-\discretestate_1\none-\discretestate_2\n$. If the initial condition $\discretestate^0$
		      satisfies $\discretestate_2^0=-\discretestate_1^0$,
		      then $\discretestate_2\none=-\discretestate_1\none$ holds inductively for all $k\geq 0$.
		      \qedhere
	\end{enumerate}

\end{myex}

\noindent Let us emphasize that the choice of additional constraints does not affect the power balance equation, which remains satisfied by \eqref{pHDAE_timestepping} and \eqref{eqn_colsp}. We also refer to \cite[Ex.~7.4.1]{morandin_phd_2019} where analogous deductions are made in the context of Galerkin projection schemes.

In the case where $E$ is singular, the question arises whether there exists a discrete gradient of $\hamiltonian$, which ensures that $\DG \hamiltonian(x, x')$ is in the column space of $\discreteE(x,x')\transp$, i.e.,
\begin{equation} \label{eqn_colsp2}
	\DG \hamiltonian(x, x')\in\mathrm{colsp}(\discreteE(x,x')\transp),
\end{equation}
for all $x,x'\in\R^n$.
This ensures that \eqref{eqn_colsp} can be solved for $\discretetimestep{f}$, although not necessarily uniquely.
Further details and a corresponding counterexample can be found in \Cref{appendix_counter_example}.
In the next section we focus on the connections between the discrete methods introduced in this section.

\section{Connecting the dots}
\label{sec_connections}

\noindent In the following, we discuss the proposed methods in more detail.
To this end, it is demonstrated in \Cref{sec_link1} that the semi-explicit discrete gradient method from \Cref{sec_version1} is equivalent to special cases of the discrete gradient pair approach from \Cref{sec_version2} and the DDR approach from \Cref{sec_version3}.
Additionally, the behavior of the proposed schemes under system transformations is analyzed in \Cref{subsec_systemTransformation}. In \Cref{sec:existence_semi_explicit} we shed more light on the existence of semi-explicit representations of pHDAEs.

\subsection{Relations between the presented methods in the semi-explicit setting}\label{sec_link1}

\noindent First, we observe that scheme \eqref{block_pHDAE_timestepping_2} for semi-explicit pHDAEs of the form \eqref{block_pHDAE}
corresponds to an underlying discrete gradient pair, see \Cref{def:discGradPair}. This is stated in the following theorem and corollary.

\begin{theorem}\label{thm_link1}
	Let $(E,\costate)$ be a semi-explicit gradient pair for $\hamiltonian$ in the sense of \Cref{def:semi-explicit_gradient_pair} and let $\costate=(\costate_1,\costate_2)$ be split correspondingly.
	Furthermore, let $\discreteE_{11}\in\cont(\statespace\times\statespace,\R^{n_1,n_1})$ and $\discretecostate_2\in\cont(\statespace\times\statespace,\R^{n_2})$ be consistent discretizations of $E_{11}$ and $\costate_2$, respectively, suppose that $\discreteE_{11}$ satisfies the assumption \ref{ass:invertible_discreteE11}, and let $\DG\specified\hamiltonian$ be a discrete gradient for the specified Hamiltonian $\specified\hamiltonian\in\cont^1(\statespace_1,\R)$.
	Then $(\discrete E,\discretecostate)$ with
	\begin{equation}\label{eq:DGP_link1}
		\discreteE = \begin{bmatrix}
			\discreteE_{11} & 0 \\ 0 & 0
		\end{bmatrix}, \qquad
		\discretecostate =
		\begin{bmatrix}
			\discreteE_{11}\ntransp(\DG\specified\hamiltonian\circ\pi_1) \\
			\discretecostate_2
		\end{bmatrix}
	\end{equation}
	is a discrete gradient pair for $(\hamiltonian,E,\costate)$.
\end{theorem}
\begin{proof}
	Property \ref{itm:discGradPH2} in \Cref{def:discGradPair} is part of our hypotheses. We proceed to show that the properties \ref{itm:discGradPH4} and \ref{itm:discGradPH3} are also valid. In fact, it holds that
	\begin{align*}
		\overline z(x,x) & =
		\begin{bmatrix}
			\discreteE_{11}(\state,\state)\ntransp\DG\specified{\hamiltonian}(\state_1,\state_1) \\
			\discretecostate_2(\state,\state)
		\end{bmatrix}
		=
		\begin{bmatrix}
			E_{11}(\state)\ntransp\gradient\specified{\hamiltonian}(\state_1) \\
			\costate_2(\state)
		\end{bmatrix}
		=
		\begin{bmatrix}
			\costate_1(\state) \\ \costate_2(\state)
		\end{bmatrix}
		= \costate(\state)
	\end{align*}
	and
	\begin{align*}
		\overline\costate(x,x')\transp\overline{E}(x,x')(x'-x) & = \overline\costate_1(x,x')\transp\overline{E}_{11}(x,x')(x'_1-x_1) =                                                                                      \\
		                                                       & = \DG\specified\hamiltonian(x_1,x_1')\transp(x_1'-x_1) = \specified\hamiltonian(x_1') - \specified\hamiltonian(x_1) = \hamiltonian(x') - \hamiltonian(x) ,
	\end{align*}
	for all $\state=(\state_1,\state_2),\state'=(\state_1',\state_2')\in\statespace$.
\end{proof}

\begin{corollary}\label{cor_link1}
	Consider a semi-explicit pHDAE of the form \eqref{block_pHDAE}, let $\DG\hamiltonian_1$, $\discreteE_{11}$, $\discretecostate_2$, $\discreteE$, and $\discretecostate$ be defined as in \Cref{thm_link1}, and let us fix consistent discretizations for $J$, $R$, and $B$. Then the semi-explicit discrete gradient method applied with $\DG\hamiltonian_1$ governed by \eqref{block_pHDAE_timestepping_2} and the discrete gradient pair method \eqref{block_pHDAE_timestepping} applied with $(\discreteE,\discretecostate)$ yield the same solution.
\end{corollary}
\begin{proof}
	The claim immediately follows by construction, since
	\[
		\discretecostate(\discretestate\n,\discretestate\none) =
		\begin{bmatrix}
			\discreteE_{11}(\discretestate\n,\discretestate\none)\ntransp\DG\specified{\hamiltonian}(\discretestate_1\n,\discretestate_1\none) \\ \discretecostate_2(\discretestate\n,\discretestate\none)
		\end{bmatrix}
		=
		\begin{bmatrix}
			\discretetimestep{\costate_1} \\ \discretecostate_2(\discretestate\n,\discretestate\none)
		\end{bmatrix}. \qedhere
	\]
\end{proof}
\noindent We will now see that the discrete gradient method applied to semi-explicit pHDAEs can be equivalently reinterpreted as a specific DDR-method.
\begin{theorem}
	Under the same assumptions as in \Cref{thm_link1}, the
	semi-explicit discrete gradient method \eqref{block_pHDAE_timestepping_2}
	yields the same one-step method as the DDR-method \eqref{eq_DDR_discrete} with the completing constraint $\discretetimestep{f_2}=\discretecostate_2(\discretestate\n,\discretestate\none)$.
\end{theorem}
\begin{proof}
	Due to the structure of the system, the DDR-method applied to \eqref{block_pHDAE} yields the one-step method
	\begin{equation}\label{eq:DDR_semiexplicit}
		\begin{bmatrix}
			\DG\specified\hamiltonian(\discretestate_1\n,\discretestate_1\none) \\ 0 \\ 0 \\ 0 \\ \discreteoutput
		\end{bmatrix}
		+
		\begin{bmatrix}
			0               & 0 & -\discreteE_{11}\transp         & 0                               & 0            \\
			0               & 0 & 0                               & 0                               & 0            \\
			\discreteE_{11} & 0 & \discreteJ_{11}-\discreteR_{11} & \discreteJ_{12}-\discreteR_{12} & \discreteB_1 \\
			0               & 0 & \discreteJ_{21}-\discreteR_{21} & \discreteJ_{22}-\discreteR_{22} & \discreteB_2 \\
			0               & 0 & -\discreteB_1\transp            & -\discreteB_2\transp            & 0
		\end{bmatrix}
		\begin{bmatrix}
			-\frac{1}{h}(\discretestate_1\none-\discretestate_1\n) \\ -\frac{1}{h}(\discretestate_2\none-\discretestate_2\n) \\ \discretetimestep{f_1} \\ \discretetimestep{f_2} \\ \discreteinput
		\end{bmatrix}
		= 0,
	\end{equation}
	where the arguments $(\discretestate\n,\discretestate\none)$ have been omitted for simplicity.
	In particular, the first equation of \eqref{eq:DDR_semiexplicit} yields
	\[
		\discreteE_{11}(\discretestate\n,\discretestate\none)\transp\discretetimestep{f_1} = \DG\specified{\hamiltonian}(\discretestate\n,\discretestate\none),
	\]
	thus it is equivalent to the equation \eqref{block_pHDAE_timestepping_2_4}, up to replacing $\discretetimestep{f_1}$ with $\discretetimestep{\costate_1}$.
	Then, since the second equation of \eqref{eq:DDR_semiexplicit} is trivial, we can remove it.
	Finally, by replacing $\discretetimestep{f_2}$ with $\discretecostate_2(\discretestate\n,\discretestate\none)$ due to the stated constraint, we get exactly \eqref{block_pHDAE_timestepping_2}.
\end{proof}

\subsection{Behavior of the presented methods under system transformations}\label{subsec_systemTransformation}

\noindent We are now interested in studying how the methods introduced in \Cref{sec_methods} behave under structure-preserving system transformations.
Our first motivation is to construct discrete gradient pairs for general gradient pairs. One possibility would be to transform the general system into an equivalent semi-explicit one, apply \Cref{thm_link1}, and then apply the inverse transformation.
However, this requires to understand more accurately the behavior of discrete gradient pairs under invertible transformations.
Another motivation is to understand whether applying the same methods under different coordinates yields different results.

We start by formalizing what is meant by \emph{structure-preserving system transformations}.
In fact, given a pHDAE of the form \eqref{eq:pHDAE}, a diffeomorphism $\varphi\in\cont^1(\widetilde\statespace,\statespace)$, and a pointwise invertible matrix function $U\in\cont(\widetilde\statespace,\R^{n,n})$, we call the pair $(\varphi,U)$ an \emph{(invertible) system transformation}.
This is motivated by the fact that we can obtain an equivalent system by applying the change of variables $\state=\varphi(\tilde\state)$ and left-multiplication of the first equation of \eqref{eq:pHDAE} by $U(\tilde\state)\transp$.
In fact, this transformation yields the new system
\begin{equation}\label{eq:pHDAE_transf}
	\begin{split}
		\widetilde E(\tilde\state)\dot{\tilde\state} & = \pset[\big]{ \widetilde J(\tilde\state) - \widetilde R(\tilde\state) } \tilde\costate(\tilde\state) + \widetilde{B}(\tilde\state)u, \\
		y                                            & = \widetilde{B}(\tilde\state) \tilde\costate(\tilde\state),
	\end{split}
\end{equation}
where $\widetilde E=U\transp(E\circ\varphi)\jacobian{\varphi}$, $\widetilde J=U\transp(J\circ\varphi)U$, $\widetilde R=U\transp(R\circ\varphi)U$, $\tilde\costate=U^{-1}(\costate\circ\varphi)$, and $\widetilde B=U\transp(B\circ\varphi)$.
Remarkably, $(\widetilde E,\tilde\costate)$ is a gradient pair for $\widetilde\hamiltonian=\hamiltonian\circ\varphi$, and the system \eqref{eq:pHDAE_transf} is a pHDAE with Hamiltonian $\widetilde\hamiltonian$, see \cite[Thm.~1]{mehrmann_2019_structurepreserving} for more details.
Furthermore, we observe that, if $(\varphi,U)$ is an invertible system transformation, then $(\varphi^{-1},U^{-1})$ is also an invertible system transformation.
In particular, applying $(\varphi^{-1},U^{-1})$ to the transformed system \eqref{eq:pHDAE_transf} we obtain again the original system \eqref{eq:pHDAE}.
This motivates calling $(\varphi^{-1},U^{-1})$ the \emph{inverse} of the system transformation $(\varphi,U)$.
We will discuss the composition and inversion of system transformations further in \Cref{app_systemTransformation}.

\begin{remark}
	Note that, if the original system is an ODE and we want to ensure that \eqref{eq:pHDAE_transf} is also an ODE, we need to choose $U=(\jacobian\varphi)\ntransp$, while in the more general case of DAEs this requirement is unnecessary and the choice of $U$ is free.
\end{remark}

\noindent Since $(\widetilde E,\tilde\costate)$ only depends on the gradient pair $(E,\costate)$ and on the system transformation $(\varphi,U)$, we
deduce that system transformations can be applied directly to (discrete) gradient pairs.
This leads to the following result.

\begin{theorem}\label{thm:discreteGradientPairChainRule}
	Let $(E,\costate)$ be a gradient pair for $\hamiltonian$, let $(\varphi,U)$ be a system transformation.
	Then
	\begin{equation}
		(\widetilde E,\tilde\costate) = \pset[\big]{U\transp(E\circ\varphi)\jacobian\varphi , U^{-1}(\costate\circ\varphi)}
	\end{equation}
	is a gradient pair for $\widetilde\hamiltonian=\hamiltonian\circ\varphi$, which we call the \emph{gradient pair transformed from $(E,\costate)$ via $(\varphi,U)$}.
	Furthermore, let $(\discreteE,\discretecostate)$ be a discrete gradient pair for $(\hamiltonian,E,\costate)$, let $\discretejacobian{\varphi}$ be a discrete Jacobian for $\varphi$, and let $\discrete{U} \in \cont(\widetilde\statespace\times\widetilde\statespace,\R^{n,n})$ be a pointwise invertible consistent discretization for $U$.
	Then $(\widehat E,\hat\costate)$ with
	\begin{equation}\label{eq:DGP_chainRule}
		\widehat E = \discrete U\transp(\discreteE\circ\varphi)\discretejacobian\varphi, \qquad
		\hat\costate = \discrete U^{-1}(\discretecostate\circ\varphi)
	\end{equation}
	is a discrete gradient pair for $(\widetilde\hamiltonian,\widetilde{E},\tilde\costate)$.
\end{theorem}

\begin{proof}
	Analogously to what was proven in \cite{mehrmann_2019_structurepreserving} in the case of pHDAEs, we have
	\[
		\widetilde E\transp\tilde\costate
		= (\jacobian\varphi)\transp (E\circ\varphi)\transp U U^{-1} (\costate\circ\varphi)
		= (\jacobian\varphi)\transp (E\transp\costate\circ\varphi)
		= (\jacobian\varphi)\transp (\gradient\hamiltonian\circ\varphi)
		= \gradient\widetilde\hamiltonian,
	\]
	thus $(\widetilde E,\tilde\costate)$ is a gradient pair for $\widetilde\hamiltonian$.
	Concerning the second statement, for every $\tilde\state,\tilde\state'\in\widetilde{\statespace}$ we have
	\begin{align*}
		\hat\costate(\tilde\state,\tilde\state')\transp \widehat{E}(\tilde\state,\tilde\state') (\tilde\state'-\tilde\state)
		 & = \discretecostate\pset[\big]{\varphi(\tilde\state),\varphi(\tilde\state')}\transp \discreteE\pset[\big]{\varphi(\tilde\state),\varphi(\tilde\state')} \discretejacobian{\varphi}(\tilde\state,\tilde\state') (\tilde\state'-\tilde\state) \\
		 & = \discretecostate\pset[\big]{\varphi(\tilde\state),\varphi(\tilde\state')}\transp \discreteE\pset[\big]{\varphi(\tilde\state),\varphi(\tilde\state')} \pset[\big]{ \varphi(\tilde\state') - \varphi(\tilde\state) }                       \\
		 & = \hamiltonian\pset[\big]{\varphi(\tilde\state')} - \hamiltonian\pset[\big]{\varphi(\tilde\state)}
		= \widetilde\hamiltonian(\tilde\state') - \widetilde\hamiltonian(\tilde\state).
	\end{align*}
	It is furthermore clear that $\widehat E(\tilde\state,\tilde\state)=\widetilde E(\tilde\state)$ and $\hat\costate(\tilde\state,\tilde\state)=\tilde\costate(\tilde\state)$ hold for all $\tilde\state\in\widetilde\statespace$.
\end{proof}

\begin{remark}
	Note that, for $(E,\costate)=(I_n,\gradient\hamiltonian)$, $(\discrete{E},\discretecostate)=(I_n,\DG\hamiltonian)$ with $\DG\hamiltonian$ being a discrete gradient of $\hamiltonian$, and $U=\discrete{U}=I_n$, \Cref{thm:discreteGradientPairChainRule} yields that $(\discretejacobian\varphi,\DG\hamiltonian\circ\varphi)$ is a discrete gradient pair for $(\widetilde\hamiltonian,I_n,\gradient\widetilde\hamiltonian)$, and therefore
	\begin{equation}\label{eq:DG_chainRule}
		\DG\widetilde\hamiltonian = (\discretejacobian\varphi)\transp(\DG\hamiltonian\circ\varphi)
	\end{equation}
	is a discrete gradient for $\widetilde\hamiltonian$.
	This is the well-known chain rule for discrete derivatives, see \cite[Prop.~3.4]{McLQR99}.
\end{remark}

\noindent We also deduce from \Cref{thm:discreteGradientPairChainRule} the following result, which fulfills our first motivation mentioned at the beginning of this section.

\begin{corollary}\label{cor:DGP_fromSemiExplicit}
	Let $(E,\costate)$ be a gradient pair for $\hamiltonian$, and suppose that there exists an invertible system transformation $(\varphi,U)$ that maps it into a semi-explicit gradient pair.
	Then $(\hamiltonian,E,\costate)$ admits a discrete gradient pair.
\end{corollary}

\begin{proof}
	In this proof we employ the fact that discrete gradients and discrete Jacobians exist regardless of the structure of the state space, as discussed in \Cref{rem:existence_of_DG}.
	Let $\widetilde\hamiltonian=\hamiltonian\circ\varphi\in\cont^1(\widetilde\statespace,\R)$ be the transformed Hamiltonian, and let $\DG\specified{\widetilde\hamiltonian}$ be any discrete gradient for the specified Hamiltonian $\specified{\widetilde\hamiltonian}$.
	We apply \Cref{thm_link1} to construct a discrete gradient pair $(\widehat E,\hat\costate)$ for $(\widetilde\hamiltonian,\widetilde E,\tilde\costate)$.
	Then, we construct a discrete gradient pair $(\discreteE,\discretecostate)$ for $(\hamiltonian,E,\costate)$ by applying \Cref{thm:discreteGradientPairChainRule} to the discrete gradient pair $(\widehat E,\hat\costate)$ via the inverse system transformation $(\varphi^{-1},U^{-1})$.
	For that, any consistent discretization of $U^{-1}$ (e.g.~the midpoint discretization) and any discrete Jacobian for $\varphi^{-1}$ can be employed.
\end{proof}

\noindent We would now like to understand whether applying the same numerical methods under different coordinate systems yields different results.
We start by studying the discrete gradient pair scheme \eqref{block_pHDAE_timestepping}, in the form of the following theorem.

\begin{theorem}\label{thm:DGP_scheme_invariant}
	Consider a pHDAE of the form \eqref{eq:pHDAE}, let $(\discreteE,\discretecostate)$ be a discrete gradient pair for $(\hamiltonian,E,\costate)$, and let $\discreteJ,\discreteR,\discreteB$ be consistent approximations for $J,R,B$, respectively, such that $\discreteJ=-\discreteJ\transp$ and $\discreteR=\discreteR\transp\succeq 0$ pointwise.
	Let $(\varphi,U)$ be an invertible system transformation, let $\discretejacobian\varphi$ be a discrete Jacobian for $\varphi$, and let $\discrete U$ be a pointwise invertible consistent approximation for $U$.
	Then the discrete gradient pair scheme \eqref{block_pHDAE_timestepping} applied to the original system, with the discrete gradient pair $(\discreteE,\discretecostate)$ and the consistent approximations $\discreteJ,\discreteR,\discreteB$, is equivalent to the same scheme applied to the system transformed via $(\varphi,U)$, with the discrete gradient pair $(\widehat E,\hat\costate)$ defined as in \eqref{eq:DGP_chainRule} and the consistent approximations $\widehat J=\discrete{U}\transp(\discreteJ\circ\varphi)\discrete{U}$, $\widehat R=\discrete{U}\transp(\discreteR\circ\varphi)\discrete{U}$, and $\widehat B=\discrete{U}\transp(\discreteB\circ\varphi)$, up to the change of variables $\varphi$.
\end{theorem}

\begin{proof}
	In this proof we will often omit the arguments $(\state\n,\state\none)$ and $(\tilde\state\n,\tilde\state\none)$, for the sake of readability.
	It is clear that $\widehat J$, $\widehat R$, and $\widehat B$ are consistent approximations for the coefficients $\widetilde J=U\transp(J\circ\varphi)U$, $\widetilde R=U\transp(R\circ\varphi)U$, and $\widetilde B=U\transp(B\circ\varphi)$ of the transformed system, and that $\widehat J=-\widehat J\transp$ and $\widehat R=\widehat R\transp\succeq0$ hold pointwise. Thus, the discrete gradient pair scheme applied on the transformed system is well-defined.

	Let now $\state^0$, $(\state\none,\discreteinput,\discreteoutput)$ for $k=0,\ldots,N-1$ denote a solution of the discrete gradient pair scheme applied to the original system, and let $\tilde\state\n=\varphi^{-1}(\state\n)$ for $k=0,\ldots,N$.
	Then we have
	\begin{align*}
		\widehat E(\tilde\state\none-\tilde\state\n) - h \pset[\big]{ (\widehat J-\widehat R)\hat\costate + \widehat B\discreteinput }
		                                                 & = \discrete{U}\transp\discrete{E}(\discretejacobian\varphi)(\tilde\state\none-\tilde\state\n)                                                                                                 \\
		                                                 & \qquad - h\pset[\big]{ (\discrete{U}\transp\discreteJ\discrete{U}-\discrete{U}\transp\discreteR\discrete{U})\discrete{U}^{-1}\discretecostate + \discrete{U}\transp\discreteB\discreteinput } \\
		                                                 & = \discrete{U}\transp \pset[\Big]{ \discrete{E}(\state\none-\state\n) - h\pset[\big]{(\discreteJ-\discreteR)\discretecostate + \discreteB\discreteinput} } = 0,                               \\
		\discreteoutput - \widehat{B}\transp\hat\costate & = \discreteoutput - \discreteB\transp\discrete{U}\,\discrete{U}^{-1}\discretecostate = \discreteoutput - \discreteB\transp\discretecostate = 0.
	\end{align*}
	Thus, $\tilde\state^0$, $(\tilde\state\none,\discreteinput,\discreteoutput)$ for $k=0,\ldots,N-1$ is a solution of the scheme applied to the transformed system.

	Analogously, if $\tilde\state^0$, $(\tilde\state\none,\discreteinput,\discreteoutput)$ for $k=0,\ldots,N-1$ is a solution of the transformed discrete system, and we define $\state\n=\varphi(\tilde\state\n)$ for $k=0,\ldots,N$, then
	\begin{align*}
		\discreteE(\state\none-\state\n) - h\pset[\big]{(\discreteJ-\discreteR)\discretecostate + \discreteB\discreteinput}
		                                                    & = \discrete{U}\ntransp \pset[\Big] { \discrete{U}\transp\discrete{E}(\discretejacobian\varphi)(\tilde\state\none-\tilde\state\n)                                                                \\
		                                                    & \qquad - h\pset[\big]{ (\discrete{U}\transp\discreteJ\discrete{U}-\discrete{U}\transp\discreteR\discrete{U})\discrete{U}^{-1}\discretecostate + \discrete{U}\transp\discreteB\discreteinput } } \\
		                                                    & = \discrete{U}\ntransp \pset[\Big]{ \widehat E(\tilde\state\none-\tilde\state\n) - h \pset[\big]{ (\widehat J-\widehat R)\hat\costate + \widehat B\discreteinput } } = 0,                       \\
		\discreteoutput - \discreteB\transp\discretecostate & = \discreteoutput - \discreteB\transp\discrete{U}\,\discrete{U}^{-1}\discretecostate = \discreteoutput - \widehat{B}\transp\hat\costate = 0,
	\end{align*}
	such that $\state^0$, $(\state\none,\discreteinput,\discreteoutput)$ for $k=0,\ldots,N-1$ is a solution of the original discrete system.
\end{proof}

\noindent Next, we analyze the semi-explicit discrete gradient scheme \eqref{block_pHDAE_timestepping_2}. For that purpose, we first have to investigate which system transformations preserve the semi-explicit structure.

\begin{proposition}\label{lem:semiExplicitStructurePreserving}
	Let $(E,\costate)$ be a gradient pair for $\hamiltonian$,
	let $(\varphi,U)$ be an invertible system transformation from another open state space $\widetilde\statespace=\widetilde\statespace_1\times\widetilde\statespace_2$, where $\widetilde\statespace_2$ is convex, and let
	us split $\varphi=(\varphi_1,\varphi_2)$ and $U=\left[\begin{smallmatrix}U_{11} & U_{12} \\ U_{21} & U_{22}\end{smallmatrix}\right]$ accordingly.
	Then the transformed gradient pair is semi-explicit if and only if $\jacobian_{\tilde\state_2}\varphi_1=0$ and $U_{12}=0$.
\end{proposition}

\begin{proof}
	Denoting by $(\widetilde E,\tilde\costate)$ the transformed gradient pair, we have
	\begin{align*}
		\widetilde E = U\transp(E\circ\varphi)\jacobian\varphi & =
		\begin{bmatrix}
			U_{11} & U_{12} \\ U_{21} & U_{22}
		\end{bmatrix}\transp
		\begin{bmatrix}
			E_{11}\circ\varphi & 0 \\ 0 & 0
		\end{bmatrix}
		\begin{bmatrix}
			\jacobian_{\tilde\state_1}{\varphi_1} & \jacobian_{\tilde\state_2}{\varphi_1} \\ \jacobian_{\tilde\state_1}{\varphi_2} & \jacobian_{\tilde\state_2}{\varphi_2}
		\end{bmatrix} \\
		                                                       & =
		\begin{bmatrix}
			U_{11}\transp (E_{11}\circ\varphi) \jacobian_{\tilde\state_1}{\varphi_1} & U_{11}\transp (E_{11}\circ\varphi) \jacobian_{\tilde\state_2}{\varphi_1} \\
			U_{12}\transp (E_{11}\circ\varphi) \jacobian_{\tilde\state_1}{\varphi_1} & U_{12}\transp (E_{11}\circ\varphi) \jacobian_{\tilde\state_2}{\varphi_1}
		\end{bmatrix}.
	\end{align*}
	Note that, since $U$ and $\jacobian\varphi$ are pointwise invertible, $\rank(\widetilde E)=\rank(E)$ pointwise, and therefore
	$(\widetilde E,\tilde\costate)$ is a semi-explicit gradient pair
	if and only if $\widetilde E_{ij}=U_{1i}\transp (E_{11}\circ\varphi) \jacobian_{\tilde\state_j}{\varphi_1}$ is invertible for $i=j=1$ and zero otherwise.

	Suppose first that $\widetilde E$ has the wished structure.
	Since $\widetilde E_{11}$ is invertible, so are $U_{11}$ and $\jacobian_{\tilde\state_1}\varphi_1$.
	Then we deduce from $\widetilde E_{12},\widetilde E_{21}=0$ that $U_{12},\jacobian_{\tilde\state_2}\varphi_1=0$.
	Suppose now that $U_{12},\jacobian_{\tilde\state_2}\varphi_1=0$.
	Then it is clear that $\widetilde E_{ij}=0$ for $(i,j)\neq(1,1)$.
	Furthermore, since $U$ and $\jacobian\varphi$ are pointwise invertible and block lower triangular, we deduce that $U_{11}$ and $\jacobian_{\tilde\state_1}\varphi_1$ are also pointwise invertible, and so is $\widetilde E_{11}$.
\end{proof}

\noindent We also need the following result, which provides conditions ensuring that the existence of a specified Hamiltonian is preserved when transforming the system.

\begin{lemma}\label{lem:DG_specChainRule}
	Let $\varphi$ be a diffeomorphism as in \Cref{lem:semiExplicitStructurePreserving} satisfying $\jacobian_{\tilde\state_2}\varphi_1=0$.
	Then, the following assertions hold.
	\begin{enumerate}[label=(\roman*)]
		\item There is a diffeomorphism $\varphi_{11}\in\cont^1(\widetilde\statespace_1,\statespace_1)$ such that
		      $\varphi_{11}\circ\pi_1=\varphi_1$ and $\jacobian\varphi_{11}\circ\pi_1=\jacobian_{\tilde\state_1}\varphi_1$.
		\item Let $\hamiltonian\in\cont^1(\statespace,\R)$ be such that $\gradient_{\state_2}\hamiltonian=0$, let $\specified\hamiltonian\in\cont^1(\statespace_1,\R)$ be its specified function defined as in \Cref{lem:specifiedDG}, let $\DG\specified\hamiltonian$ be a discrete gradient for $\specified\hamiltonian$, and let $\discretejacobian\varphi_{11}$ be a discrete Jacobian for $\varphi_{11}$.
		      Then $\widetilde\hamiltonian=\hamiltonian\circ\varphi$ has a specified function $\specified{\widetilde\hamiltonian}\in\cont^1(\widetilde\statespace_1,\R)$, and
		      \begin{equation}\label{eq:DG_specChainRule}
			      \DG\specified{\widetilde\hamiltonian} = (\discretejacobian\varphi_{11})\transp (\DG\specified\hamiltonian\circ\varphi_{11})
		      \end{equation}
		      is a discrete gradient for $\specified{\widetilde\hamiltonian}$.
	\end{enumerate}
\end{lemma}

\begin{proof}
	Let us prove the two statements separately:
	\begin{enumerate}[label=(\roman*)]
		\item It is clear that $\varphi_{11}\in\cont^1(\widetilde\statespace_1,\statespace_1)$ satisfying
		      $\varphi_{11}\circ\pi_1=\varphi_1$ and $\jacobian\varphi\circ\pi_1=\jacobian_{\tilde\state_1}\varphi_1$
		      exists because of \Cref{lem:specifiedDG}.
		      Let now $\psi=(\psi_1,\psi_2)=\varphi^{-1}$.
		      Since $\jacobian\psi=(\jacobian\varphi\circ\psi)^{-1}$ has the same block lower triangular structure as $\jacobian\varphi$, we construct $\psi_{11}\in\cont^1(\statespace_1,\widetilde\statespace_1)$ analogously to $\varphi_{11}$, and deduce that $\varphi_{11}\circ\psi_{11}$ and $\psi_{11}\circ\varphi_{11}$ are both the identity.
		      Since $\jacobian\varphi_{11}$ inherits from $\jacobian_{\tilde\state_1}\varphi_1$ (and therefore from $\jacobian\varphi$) the property of being pointwise invertible, we conclude that $\varphi_{11}$ is indeed a diffeomorphism.
		\item Since
		      \[
			      \gradient\widetilde\hamiltonian = (\jacobian\varphi)\transp (\gradient\hamiltonian\circ\varphi) =
			      \begin{bmatrix}
				      \jacobian_{\tilde\state_1}\varphi_1\transp & \jacobian_{\tilde\state_1}\varphi_2\transp \\ 0 & \jacobian_{\tilde\state_2}\varphi_2\transp
			      \end{bmatrix}
			      \begin{bmatrix}
				      \gradient_{\state_1}\hamiltonian\circ\varphi \\ 0
			      \end{bmatrix}
			      =
			      \begin{bmatrix}
				      \jacobian_{\tilde\state_1}\varphi_1\transp(\gradient_{\state_1}\hamiltonian\circ\varphi) \\ 0
			      \end{bmatrix},
		      \]
		      the specified function $\specified{\widetilde\hamiltonian}$ is well-defined, and clearly satisfies $\specified{\widetilde\hamiltonian}=\specified\hamiltonian\circ\varphi_{11}$.
		      The fact that $\DG\specified{\widetilde\hamiltonian}$ is a discrete gradient for $\specified{\widetilde\hamiltonian}$ follows immediately from the chain rule \eqref{eq:DG_chainRule}. \qedhere
	\end{enumerate}
\end{proof}

\noindent We can now prove the following result, which states that applying the semi-explicit discrete gradient scheme from \Cref{sec_version1} to a semi-explicit system and to a corresponding transformed semi-explicit system leads to equivalent time-discrete systems.

\begin{theorem}\label{thm:SEDG_scheme_invariant}
	Consider a semi-explicit pHDAE of the form \eqref{block_pHDAE}, let $\DG\specified\hamiltonian$ be a discrete gradient for the specified Hamiltonian $\specified\hamiltonian$, and let $\discrete E_{11},\discretecostate_2,\discreteJ,\discreteR,\discreteB$ be consistent approximations for $E_{11},\costate_2,J,R,B$, respectively, such that assumption \ref{ass:invertible_discreteE11} is satisfied, and $\discreteJ=-\discreteJ\transp$ and $\discreteR=\discreteR\transp\succeq 0$ hold pointwise.
	Let $(\varphi,U)$ be an invertible system transformation preserving the semi-explicit structure as in \Cref{lem:semiExplicitStructurePreserving}, let $\discretejacobian\varphi_{11}$ be a pointwise invertible discrete Jacobian of $\varphi_{11}$ as specified in \Cref{lem:DG_specChainRule}, and let $\discrete U=\Big[\begin{smallmatrix}\discrete U_{11} & 0 \\ \discrete U_{21} & \discrete U_{22}\end{smallmatrix}\Big]$ be a pointwise invertible consistent approximation for $U$.
	Then the semi-explicit discrete gradient scheme \eqref{block_pHDAE_timestepping_2} applied to the original system, with the discrete gradient $\DG\specified\hamiltonian$ and the consistent approximations $\discrete E_{11},\discretecostate_2,\discreteJ,\discreteR,\discreteB$, is equivalent to the same scheme applied to the system transformed via $(\varphi,U)$, with the discrete gradient $\DG\specified{\widetilde\hamiltonian}$ defined as in \eqref{eq:DG_specChainRule} and the consistent approximations $\widehat E_{11}=\discrete{U}_{11}\transp(\discreteE_{11}\circ\varphi)(\discretejacobian\varphi_{11}\circ\pi_1)$, $\hat\costate_2 = \discrete{U}_{22}^{-1} \pset[\big]{ (\discretecostate_2\circ\varphi) - \discrete U_{21}\widehat E_{11}\ntransp(\DG\specified{\widetilde\hamiltonian}\circ\pi_1)}$, $\widehat J=\discrete{U}\transp(\discreteJ\circ\varphi)\discrete{U}$, $\widehat R=\discrete{U}\transp(\discreteR\circ\varphi)\discrete{U}$, and $\widehat B=\discrete{U}\transp(\discreteB\circ\varphi)$, up to the change of variables $\varphi$.
\end{theorem}

\begin{proof}
	We first show that the semi-explicit discrete gradient scheme applied to the transformed system is well-defined.
	To this end, we observe that the transformed system is semi-explicit by assumption, $\DG\specified{\widetilde\hamiltonian}$ is a discrete gradient of $\specified{\widetilde\hamiltonian}$ because of \Cref{lem:DG_specChainRule}, $\widehat E_{11}$, $\widehat J$, $\widehat R$, and $\widehat B$ are consistent approximations of $\widetilde E_{11}=U_{11}\transp(E_{11}\circ\varphi)\jacobian_{\tilde\state_1}\varphi_1 = U_{11}\transp(E_{11}\circ\varphi)(\jacobian\varphi_{11}\circ\pi_1)$, $\widetilde J=U\transp(J\circ\varphi)U$, $\widetilde R=U\transp(R\circ\varphi)U$, and $\widetilde B=U\transp(B\circ\varphi)$, respectively, and $\widehat J=-\widehat J\transp$ and $\widehat R=\widehat R\transp\succeq0$ hold pointwise.
	Furthermore, note that $\widehat E_{11}$ satisfies assumption \ref{ass:invertible_discreteE11} and that, since
	\begin{align*}
		\tilde\costate_2
		 & = [0,I_{n_2}]U^{-1}(\costate\circ\varphi)
		= U_{22}^{-1}\pset[\big]{ (\costate_2\circ\varphi) - U_{21}U_{11}^{-1}(\costate_1\circ\varphi) }                                                     \\
		 & = U_{22}^{-1}\pset[\big]{ (\costate_2\circ\varphi) - U_{21}U_{11}^{-1}(E_{11}\circ\varphi)\ntransp(\gradient_{\state_1}\hamiltonian\circ\varphi)}
		= U_{22}^{-1}\pset[\big]{ (\costate_2\circ\varphi) - U_{21}\widetilde E_{11}\ntransp(\gradient\specified{\widetilde\hamiltonian}\circ\pi_1) },
	\end{align*}
	$\hat\costate_2$ is a consistent approximation of $\tilde\costate_2$.

	Let us now construct a discrete Jacobian $\discretejacobian\varphi$ of $\varphi$, in such a way that $\discretejacobian_{\tilde\state_1}\varphi_1=\discretejacobian\varphi_{11}\circ\pi_1$.
	This can be simply done by choosing a discrete Jacobian $\discretejacobian\varphi_2$ of $\varphi_2=\pi_2\circ\varphi$ and defining
	\[
		\discretejacobian\varphi =
		\begin{bmatrix}
			\discretejacobian\varphi_{11}\circ\pi_1 & 0 \\ \discretejacobian_{\tilde\state_1}\varphi_2 & \discretejacobian_{\tilde\state_2}\varphi_2
		\end{bmatrix},
	\]
	which, as it can be easily verified, fulfills the discrete Jacobian definition.

	To prove the statement of the theorem it is then sufficient to combine \Cref{cor_link1} and \Cref{thm:DGP_scheme_invariant}.
	In fact, the semi-explicit discrete gradient scheme applied to the original system is equivalent to the discrete gradient pair scheme applied to the same system with the discrete gradient pair $(\discreteE,\discretecostate)$ defined as in \eqref{eq:DGP_link1}, i.e.,
	\[
		\discreteE =
		\begin{bmatrix}
			\discreteE_{11} & 0 \\ 0 & 0
		\end{bmatrix}, \qquad
		\discretecostate =
		\begin{bmatrix}
			\discreteE_{11}\ntransp (\DG\specified\hamiltonian\circ\pi_1) \\
			\discretecostate_2
		\end{bmatrix},
	\]
	where $\pi_1$ here is to be intended as the projection of $\statespace$ onto $\statespace_1$.
	This discrete system is then equivalent up to the change of variables $\varphi$ to the one yielded by the discrete gradient pair scheme applied to the transformed system with the discrete gradient pair $(\widehat E,\hat\costate)$ defined as in \Cref{thm:DGP_scheme_invariant}, i.e.,

	\begin{align*}
		\widehat E   & = \discrete U\transp (\discreteE\circ\varphi) \discretejacobian\varphi =
		\begin{bmatrix}
			\discrete U_{11}\transp (\discreteE_{11}\circ\varphi) (\discretejacobian\varphi_{11}\circ\pi_1) & 0 \\ 0 & 0
		\end{bmatrix}
		=
		\begin{bmatrix}
			\widehat E_{11} & 0 \\ 0 & 0
		\end{bmatrix},                                                             \\
		\hat\costate & = \discrete{U}^{-1}(\discretecostate\circ\varphi) =
		\begin{bmatrix}
			\discrete U_{11}^{-1}(\discreteE_{11}\circ\varphi)\ntransp(\DG\specified\hamiltonian\circ\pi_1\circ\phi) \\
			\discrete{U}_{22}^{-1} \pset[\big]{ (\discretecostate_2\circ\varphi) - \discrete U_{21}\discrete U_{11}^{-1}(\discreteE_{11}\circ\varphi)\ntransp(\DG\specified\hamiltonian\circ\pi_1\circ\phi) }
		\end{bmatrix}
		=
		\begin{bmatrix}
			\widehat{E}_{11}\ntransp (\DG\specified{\widetilde\hamiltonian}\circ\pi_1) \\
			\hat\costate_2
		\end{bmatrix}
		,
	\end{align*}
	and the consistent approximations $\widehat J,\widehat R,\widehat B$.
	Finally, due to the structure of $(\widehat E,\hat\costate)$, we can apply again \Cref{cor_link1} and conclude that the latest discrete system is equivalent to the one obtained by applying the semi-explicit discrete gradient scheme to the transformed system with the discrete gradient $\DG\specified{\widetilde\hamiltonian}$ and the consistent approximations $\widehat E_{11}$, $\hat\costate_2$, $\widehat J$, $\widehat R$, and $\widehat B$.
\end{proof}

\begin{remark}\label{rem:DDR_systemTransformation}
	A similar result holds for the DDR-method applied to a given DDR-pHDAE \eqref{eq_DDR}. On the one hand, its corresponding DDR-pHDAE transformed via $(\varphi,U)$ is
	\begin{equation}
		\begin{bmatrix}
			(\jacobian{\varphi})\transp(\gradient\hamiltonian\circ\varphi) \\ 0 \\ y
		\end{bmatrix}
		+
		\begin{bmatrix}
			0                                         & -\jacobian\varphi\transp (E\circ\varphi)\transp U & 0                        \\
			U\transp (E\circ\varphi) \jacobian\varphi & U\transp \pset[\big]{(J-R)\circ\varphi} U         & U\transp (B\circ\varphi) \\
			0                                         & (B\circ\varphi)\transp U                          & 0
		\end{bmatrix}
		\begin{bmatrix}
			-\dot{\tilde\state} \\ \tilde f \\ u
		\end{bmatrix}
		= 0,
	\end{equation}
	together with $\tilde f=U^{-1}(\costate\circ\varphi)$.
	On the other hand, by applying the DDR-method \eqref{eq_DDR_discrete} for a fixed discrete gradient $\DG\hamiltonian$ of $\hamiltonian$ and some given consistent approximations $\discreteE$, $\discreteJ$, $\discreteR$, and $\discreteB$ of $E$, $J$, $R$, and $B$, respectively, and multiplying the first block row by $(\discretejacobian\varphi)\transp$ and the second block row by $\discrete U\transp$, where $\discretejacobian\varphi$ is a discrete Jacobian for $\varphi$ and $\discrete U$ is a consistent approximation of $U$, and replacing $\state\none-\state\n=\discretejacobian\varphi(\tilde\state\n,\tilde\state\none)(\tilde\state\none-\tilde\state\n)$ and $\discretetimestep{f}=\discrete{U}(\tilde\state\n,\tilde\state\none)\discretetimestep{\tilde f}$ with $\tilde\state\n=\varphi^{-1}(\state\n)$ and $\tilde\state\none=\varphi^{-1}(\state\none)$, we obtain the one-step method

	\begin{equation}\label{eq:transfDiscDDR}
		\begin{bmatrix}
			(\discretejacobian\varphi)\transp(\DG\hamiltonian\circ\varphi) \\ 0 \\ \discreteoutput
		\end{bmatrix}
		+
		\begin{bmatrix}
			0                                                                     & -\discretejacobian\varphi\transp (\discreteE\circ\varphi)\transp \discrete{U}     & 0                                          \\
			\discrete{U}\transp (\discreteE\circ\varphi) \discretejacobian\varphi & \discrete{U}\transp \pset[\big]{(\discreteJ-\discreteR)\circ\varphi} \discrete{U} & \discrete U\transp(\discreteB\circ\varphi) \\
			0                                                                     & (\discreteB\circ\varphi)\transp \discrete U                                       & 0
		\end{bmatrix}
		\begin{bmatrix}
			-\frac{\tilde\state\none-\tilde\state\n}{h} \\ \discretetimestep{\tilde f} \\ \discreteinput
		\end{bmatrix}
		= 0,
	\end{equation}
	where the arguments $(\tilde\state\n,\tilde\state\none)$ have been omitted to keep the notation short.
	Note that, since $\DG\widetilde\hamiltonian=(\discretejacobian\varphi)\transp(\DG\hamiltonian\circ\varphi)$ is a discrete gradient for $\widetilde\hamiltonian=\hamiltonian\circ\varphi$ because of the chain rule, and $\widehat E=\discrete{U}\transp (\discreteE\circ\varphi) \discretejacobian\varphi$, $\widehat J=\discrete{U}\transp(\discreteJ\circ\varphi)\discrete{U}$, $\widehat R = \discrete{U}\transp(\discreteR\circ\varphi)\discrete{U}$, and $\widehat B=\discrete{U}\transp(\discreteB\circ\varphi)$ are consistent approximations of the correspondent coefficients of the transformed system \eqref{eq:pHDAE_transf}, the one-step method \eqref{eq:transfDiscDDR} is equivalent to an appropriate DDR-method applied to the transformed system.

	This construction immediately shows that every solution of the original discrete DDR system is uniquely mapped into a solution of the transformed discrete DDR system.
	However, to be able to uniquely associate to every solution of the transformed discrete DDR system one solution of the original discrete DDR system, we need to invert the construction, which requires the discrete Jacobian $\discretejacobian\varphi(\state\n,\state\none)$ to be invertible.

	Similarly to our considerations leading to
	Assumption \ref{ass:invertible_discreteE11} we note that, for small enough time steps, this is guaranteed by consistency, since $\discretejacobian\varphi(\state\n,\state\n)=\jacobian\varphi(\state\n)$ is invertible for every $\state\n\in\statespace$. However, there is in general no guarantee that the discrete Jacobian of a diffeomorphism is pointwise invertible, cf.~\Cref{ex:counterexampleForInvertibleDiscreteJacobianOfADiffeomorphism}.

	It remains to discuss how the additional constraints used in the DDR-methods change under system transformations. This strongly depends on the specific form of these constraints, which can be quite diverse.
	For example, if in the original coordinates we have a constraint of the form $F(\state\n,\state\none,\discretetimestep{f})=0$ for some function $F:\statespace\times\statespace\times\R^m\to\R^p$, then to obtain an equivalent one-step method the corresponding constraint in the new coordinates would be $\widetilde F(\tilde\state\n,\tilde\state\none,\discretetimestep{\tilde f})=0$ with
	\[
		\widetilde F(\tilde\state\n,\tilde\state\none,\discretetimestep{\tilde f}) = F(\varphi(\tilde\state\n),\varphi(\tilde\state\none),\discrete U(\tilde\state\n,\tilde\state\none)\discretetimestep{\tilde f}).
	\]
	Similarly, if in the original coordinates we require $\norm{\discretetimestep{f}-\discretecostate(\state\n,\state\none)}$ to be minimal for some fixed consistent discretization $\discretecostate$ of $\costate$, in the new coordinates we would minimize
	\[
		\norm{\discretetimestep{f}-\discretecostate(\state\n,\state\none)} = \norm{\discrete{U}(\tilde\state\n,\tilde\state\none)(\discretetimestep{\tilde f}-\hat\costate(\tilde\state\n,\tilde\state\none))}
	\]
	instead, where $\hat\costate=\discrete U^{-1}(\discretecostate\circ\varphi)$ is a consistent discretization of $\tilde\costate$.
\end{remark}

\noindent In this subsection, we focused on how system transformations affect the proposed time discretization schemes. In particular, we have proven that, given a system transformation that turns a pHDAE into semi-explicit form, one can always construct a discrete gradient pair for the original system. The conditions for the existence of such a transformation will be the focus of the next subsection.

\subsection{On the existence of a semi-explicit representation}
\label{sec:existence_semi_explicit}

\noindent \Cref{cor:DGP_fromSemiExplicit} requires the existence of an invertible system transformation $(\varphi,U)$ which brings the system to semi-explicit form.
In the special case where $E$ is constant, such a transformation can be obtained based on a singular value decomposition (SVD) of $E$, as detailed in the following proposition.

\begin{proposition}
	\label{prop:discreteGradientPair_constantE}
	Let $(E,\costate)$ be a gradient pair for $\hamiltonian\in\cont^1(\statespace,\R)$, assume that $E\in\R^{n,n}$ is constant, and
	let $E=U\Sigma V\transp$ be an SVD of $E$.
	Furthermore, let $U=[U_1,U_2]$, $V=[V_1,V_2]$, and $\Sigma=\diag(\Sigma_1,0)$ with $U_1,W_1\in\R^{n,r}$, $\Sigma_1\in\R^{r,r}$, and $r=\rank(E)$.
	Then the following statements hold:
	\begin{enumerate}[label=(\roman*)]
		\item $(\varphi,U)$ with $\varphi(\state)=V\state$ is an invertible system transformation that maps $(E,\costate)$ into a semi-explicit gradient pair.
		\item The specified function $\specified{\widetilde\hamiltonian}:\widetilde\statespace_1\to\R$ of $\widetilde\hamiltonian=\hamiltonian\circ\varphi$ satisfies $\specified{\widetilde\hamiltonian}(\tilde\state_1)=\hamiltonian(V_1\tilde\state_1)$.
		\item If $\DG\specified{\widetilde\hamiltonian}$ is a discrete gradient for $\specified{\widetilde\hamiltonian}$ and $\hat\costate_2$ is a consistent approximation for $U_2\transp(\costate\circ\varphi)$, then $(E,\discretecostate_2)$ with $\discretecostate_2=U_1\Sigma_1^{-1}\DG\specified{\widetilde\hamiltonian}+U_2\hat\costate$ is a discrete gradient pair for $(\hamiltonian,E,\costate)$.
	\end{enumerate}
\end{proposition}
\begin{proof}
	Straightforward calculations yield that the result follows from \Cref{thm_link1} and \Cref{thm:discreteGradientPairChainRule}.
\end{proof}

\noindent%
\Cref{prop:discreteGradientPair_constantE} can find its use in several application instances. We demonstrate this in the following example.
\begin{myex}[Multibody system with singular mass matrix]\label{ex_singularM}
	Using \Cref{prop:discreteGradientPair_constantE}, one can design a discrete gradient pair with corresponding integration method for
	a multibody system example with singular mass matrix from \cite[Sec.~5, Ex.~3]{udwadia_2006_explicit} (see also \cite{kinon_2024_conserving}) as depicted in \Cref{fig:federmasse_sketch_massspring_energy}, with masses $m_1$ and $m_2$.
	Due to the presence of redundant coordinates, the descriptor matrix is given by
	$ E =\diag{(I_{3 , 3}, M, 0)}$ with singular but constant mass matrix
	\begin{equation} \label{singular_M}
		M = \begin{bmatrix}
			m_1 &  & 0   &  & 0   \\
			0   &  & m_2 &  & m_2 \\
			0   &  & m_2 &  & m_2
		\end{bmatrix} .
	\end{equation}
	Performing the steps as shown above yields the desired discrete gradient pair. More details can be found in \Cref{appendix_mass_spring}.
\end{myex}
\begin{figure}[bth]
	\centering
	\hspace{4mm}
	\includegraphics{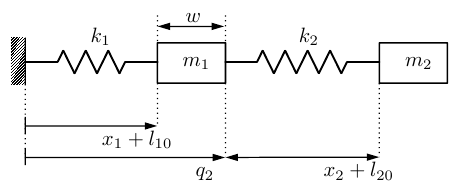}
	\caption{Mass-spring system}
	\label{fig:federmasse_sketch_massspring_energy}
\end{figure}
\noindent In the general case of a state-dependent $E$, one might wonder whether it is possible to transform any pHDAE \eqref{eq:pHDAE} into an equivalent one in the semi-explicit form \eqref{block_pHDAE}.
We start by deducing the following necessary condition:
\begin{theorem} \label{tmp_theorem}
	Let $(E,\costate)$ be a gradient pair for $\hamiltonian$, and suppose that there exists an invertible system transformation $(\varphi,U)$ that maps it into a semi-explicit gradient pair. Then $E$ has constant rank.

\end{theorem}
\begin{proof}
	Since $\varphi$ is a diffeomorphism, $\jacobian\varphi$ is pointwise invertible, thus we deduce that
	\[
		\rank \pset[\big]{ E(\state) }
		= \rank\pset[\big]{ U\pset{\tilde\state}\transp E(\state) \jacobian\varphi(\tilde\state) }
		= \rank\pset[\big]{ \widetilde{E}\pset{\tilde\state} } = n_1
	\]
	for all $\state\in\statespace$, where $\tilde\state=\varphi^{-1}(\state)$.
\end{proof}
\noindent We investigate now whether this condition is also sufficient.
A simple extension of \cite[Thm.~3.9]{KunM23} (see \Cref{thm:nullSpaceDec} in the appendix)
shows that $E$ being continuous and constant rank is sufficient to find pointwise unitary $U,V\in\cont(\statespace,\R^{n,n})$ satisfying
\[
	U\transp EV =
	\begin{bmatrix}
		E_{11} & 0 \\ 0 & 0
	\end{bmatrix}
\]
for some pointwise invertible matrix function $E_{11}$, at least locally.
Unfortunately, since such $V$ is not necessarily the Jacobian of a diffeomorphism $\varphi$, we cannot exploit this result directly.
However, we are still able to provide a crucial local canonical form for a gradient pair $(E,\costate)$, as long as $E$ is analytic and has constant rank.=

\begin{theorem}\label{thm:analyticGradientPair}
	Let $(E,\costate)$ be a gradient pair for $\hamiltonian\in\cont^1(\statespace,\R)$, and suppose that $E\in\cont(\statespace,\R^{n,n})$ is analytic and has constant rank. Then $(E,\costate)$ is locally equivalent to a gradient pair $(\widetilde E,\tilde\costate)$ for $\widetilde\hamiltonian\in\cont^1(\widetilde\statespace,\R)$, such that $\widetilde\hamiltonian$ admits a specified function $\specified{\widetilde\hamiltonian}\in\cont^1(\pi_1(\widetilde\statespace),\R)$ (in the sense of \Cref{lem:specifiedDG}), and
	\begin{equation}
		\widetilde E =
		\begin{bmatrix}
			I_p & 0 & 0 \\ 0 & 0 & 0 \\ 0 & E_{32} & I_{r-p}
		\end{bmatrix}, \qquad
		\tilde\costate =
		\begin{bmatrix}
			\gradient\specified{\widetilde\hamiltonian}\circ\pi_1 \\ \costate_2 \\ 0
		\end{bmatrix},
	\end{equation}
	where $\pi_1:\R^n\to\R^p$ denotes the projection onto the first $p$ coordinates.
\end{theorem}
\begin{proof}
	See \Cref{appendix_to_semiexplicit}.
\end{proof}

\noindent The canonical form in \Cref{thm:analyticGradientPair} allows us then to find a canonical form for pHDAEs, which allows to split them into a structured semi-explicit DAE and an additional unstructured DAE.

\begin{corollary}\label{cor:pHDAE_to_SE}
	Consider a pHDAE of the form \eqref{eq:pHDAE}, and suppose that the descriptor matrix $E$ is analytic and has constant rank.
	Then the system is locally equivalent to the combination of a parametrized semi-explicit pHDAE of the form
	\begin{subequations}
		\label{eq:canonicalAndUnstruc}
		\begin{equation}\label{eq:canonicalSemiExplicit}
			\begin{aligned}
				\begin{bmatrix}
					I_{n_1} & 0 \\ 0 & 0
				\end{bmatrix}
				\begin{bmatrix}
					\dot x_1 \\ \dot x_2
				\end{bmatrix}
				  & = \pset[\big]{J(x,\theta)-R(x,\theta)}
				\begin{bmatrix}
					\gradient\specified{\hamiltonian}(x_1) \\ z_2(x,\theta)
				\end{bmatrix}
				+
				\begin{bmatrix}
					B_1(x,\theta) \\ B_2(x,\theta)
				\end{bmatrix}
				u,                                         \\
				y & =
				\begin{bmatrix}
					B_1(x,\theta)\transp & B_2(x,\theta)\transp
				\end{bmatrix}
				\begin{bmatrix}
					\gradient\specified{\hamiltonian}(x_1) \\ z_2(x,\theta)
				\end{bmatrix}
			\end{aligned}
		\end{equation}
		with specified Hamiltonian depending only on $x_1$, and an unstructured DAE for the parameter $\theta$ given by
		\begin{equation} \label{eq:unstruc_part}
			\dot \theta + E_{32}(x,\theta)\dot x_2 = A_{31}(x,\theta)\gradient \specified{\hamiltonian}(x_1) + A_{32}(x,\theta)z_2(x,\theta) + B_3(x,\theta)u,
		\end{equation}
	\end{subequations}
	with state $x=(x_1,x_2)$, up to invertible system transformations.
\end{corollary}

\begin{proof}
	Because of \Cref{thm:analyticGradientPair} we can assume, up to restricting the state space to an appropriate open neighborhood and applying a certain invertible system transformation, that $\hamiltonian$ admits a specified Hamiltonian $\specified{\hamiltonian}\in\cont^1(\pi_1(\statespace),\R)$ and
	\[
		E =
		\begin{bmatrix}
			I_p & 0 & 0 \\ 0 & 0 & 0 \\ 0 & E_{32} & I_{r-p}
		\end{bmatrix}, \qquad
		\costate =
		\begin{bmatrix}
			\gradient\specified\hamiltonian\circ\pi_1 \\ \costate_2 \\ 0
		\end{bmatrix}.
	\]
	The pHDAE then can be written as
	\begin{align*}
		\begin{bmatrix}
			\dot x_1 \\ 0 \\ E_{32}(x)\dot x_2 + \dot x_3
		\end{bmatrix}
		  & =
		\begin{bmatrix}
			J_{11}(x)-R_{11}(x) & J_{12}(x)-R_{12}(x) & J_{13}(x)-R_{13}(x) \\
			J_{21}(x)-R_{21}(x) & J_{22}(x)-R_{22}(x) & J_{23}(x)-R_{23}(x) \\
			J_{31}(x)-R_{31}(x) & J_{32}(x)-R_{32}(x) & J_{33}(x)-R_{33}(x)
		\end{bmatrix}
		\begin{bmatrix}
			\gradient\specified{\hamiltonian}(x_1) \\ z_2(x) \\ 0
		\end{bmatrix}
		+
		\begin{bmatrix}
			B_1(x) \\ B_2(x) \\ B_3(x)
		\end{bmatrix}
		u,    \\
		y & =
		\begin{bmatrix}
			B_1(x)\transp & B_2(x)\transp & B_3(x)\transp
		\end{bmatrix}
		\begin{bmatrix}
			\gradient\specified{\hamiltonian}(x_1) \\ z_2(x) \\ 0
		\end{bmatrix}.
	\end{align*}
	Note that, since $z_3=0$, the third block column of $J-R$ can be arbitrarily modified without affecting the solutions of the system.
	Therefore, $A_{31}=J_{31}-R_{31}$ and $A_{32}=J_{32}-R_{32}$ are actually unstructured, and the system can be equivalently interpreted as \eqref{eq:canonicalAndUnstruc}, up to relabeling $\state\coloneqq(\state_1,\state_2)$ and $\theta\coloneqq\state_3$.

	We finally note that, up to restricting $\statespace$ further, we can assume that it is an open convex of the form $\statespace_1\times\statespace_2\times\statespace_3$ with $\statespace_1\subseteq\R^{n_1}$, $\statespace_2\subseteq\R^{n_2}$, and $\statespace_3\subseteq\R^{n_3}$.
\end{proof}

\begin{remark}

	In the proof of \Cref{cor:pHDAE_to_SE} we reinterpreted part of the state as a time-varying parameter.
	While this choice might seem arbitrary, it allows us to highlight and exploit the underlying semi-explicit pHDAE structure. In fact,
	while the full system \eqref{eq:canonicalAndUnstruc} is not a semi-explicit pHDAE, the fact that the subsystem \eqref{eq:canonicalSemiExplicit} is a parametrized semi-explicit pHDAE allows to apply a structure-preserving time-discretization scheme by approximating $\gradient\specified{\hamiltonian}(x_1)$ by a corresponding discrete gradient.
	Additionally, one is free to choose an approximation of the unstructured part \eqref{eq:unstruc_part}, which however does not spoil the discrete time power balance equation.
	It should be emphasized that the derivation of the system \eqref{eq:canonicalAndUnstruc} requires a suitable system transformation, which may be difficult to obtain in practice.
\end{remark}

\noindent While we leave further in-depth analyses for future research, we subsequently highlight the  applicability of our proposed approach to a mechanical problem class.

\section{Application to multibody system dynamics}\label{sec_examples}

\noindent Let us consider the example of nonlinear and constrained multibody systems (see \Cref{ex_mbs_shorter}). We discuss the modeling as a semi-explicit pHDAE in \Cref{modelling_mbs}, showcase the application of a discrete gradient method in \Cref{sec_mbs_integrator} and present a numerical experiment in \Cref{sec_mbs_experiment}.

\subsection{Modeling multibody systems as semi-explicit pHDAEs} \label{modelling_mbs}

\noindent The class of nonlinear multibody systems with redundant coordinates $q \in \cont (\timeinterval,\mathcal{Q})$ fits well into the semi-explicit framework \eqref{block_pHDAE}. More details on the derivations of the following equations may be found for example in the textbook \cite[Ch.~1]{holm_2009_geometric}.
The configuration space $\mathcal{Q}$ is typically a differential manifold, but it can also be regarded as an open subset of $\R^d$ up to switching to local coordinates, where the dimension $d$ of $\mathcal{Q}$ determines the number of coordinates.
Correspondingly, admissible velocities $v=\dot{q}$ are elements of the tangent space $T_q\mathcal{Q}$ defined through the presence of holonomic constraints $g\in\cont^1(\mathcal{Q},\R^m)$, and can be reinterpreted in local coordinates as vectors in $\R^d$.
Since
\begin{align} \label{eq_pos_constraint}
	g \pset[\big]{ q(t) } = 0
\end{align}
gives rise to the kinematic (i.e.~velocity level) constraints (sometimes in the MBS community referred to as \textit{hidden constraints}), admissible velocities need to satisfy
\begin{align}\label{eq_vel_constraint}
	\jacobian g \pset[\big]{ q(t) } v(t) = 0 .
\end{align}
These constraints are enforced by means of Lagrange multipliers $\lambda \in \cont(\timeinterval, \R^m)$, which now represent the purely algebraic states, i.e., $\state_2 = \lambda$. Correspondingly, with $\state_1 = (q,v)$ one defines the non-quadratic Hamiltonian as
\begin{align} \label{hamiltonian_mbs_detail}
	\hamiltonian(\state) =\specified{\hamiltonian}(\state_1) = \frac{1}{2} v\transp M v + V(q),
\end{align}
where the first term represents the kinetic energy with the symmetric and positive-definite mass matrix $M \in \R^{d,d}$ and $V \in \cont(\mathcal{Q},\R)$
denotes an arbitrary potential energy. The emanating potential forces are derived by taking the partial derivative with respect to the coordinates, i.e., $F_{\mathrm{p}} =  \gradient V(q)$. Additionally, we consider velocity-dependent viscous dissipation governed by the Rayleigh dissipation function $G(q,v) = \frac{1}{2} v\transp R_{\mathrm{R}}(q) v$, where $R_{\mathrm{R}}(q) \in \cont(\mathcal{Q},\R^{d,d})$ is a symmetric and positive semi-definit dissipation matrix. The non-potential forces appearing in the balance of linear momentum are obtained through differentiation, i.e., $F_{\mathrm{np}}(q,v) = - \nabla_{v} G(q,v) = - R_{\mathrm{R}}(q) v $. This eventually yields the equations of motion as index-2 DAEs given by
\begin{subequations}
	\begin{align}
		\dot{q}  & = v                            ,                                               \\
		M \dot v & = - \gradient V(q) - R_{\mathrm{R}}(q) v - \jacobian g(q)\transp\lambda + u  , \\
		0        & = \jacobian g(q) v ,
	\end{align}
\end{subequations}
where $u$ represents external input loads.
These equations can be brought into the semi-explicit pHDAE representation \eqref{block_pHDAE} by rewriting them as
\begin{subequations}
	\begin{align} \label{block_MBS_detailled}
		\begin{bmatrix}
			I & 0 & 0 \\ 0 & M & 0 \\ 0 & 0 & 0
		\end{bmatrix} \begin{bmatrix}
			              \dot{q} \\ \dot{v} \\ \dot{\lambda}
		              \end{bmatrix} & = \left( \begin{bmatrix}
				                                       0  & I                  & 0                      \\
				                                       -I & -R_{\mathrm{R}}(q) & -\jacobian g(q)\transp \\ 0 & \jacobian g(q) & 0
			                                       \end{bmatrix} \right) \begin{bmatrix}
			                                                             \gradient V(q) \\ v \\ \lambda
		                                                             \end{bmatrix} + \begin{bmatrix}
			                                                                             0 \\ I \\ 0
		                                                                             \end{bmatrix} u , \\
		y                                   & = \begin{bmatrix}
			                                        0 & I & 0
		                                        \end{bmatrix} \begin{bmatrix}
			                                                      \gradient V(q) \\ v \\ \lambda
		                                                      \end{bmatrix} .
	\end{align}
\end{subequations}
The verification that $E\transp\costate(\state) = \gradient \hamiltonian(\state)$ holds true is straightforward. Moreover, the system output collocated with the input forces coincides with the velocity, i.e.,~$y=v$.

Note that the pH formulation of the multibody system dynamics is characterized by explicitely accounting for the hidden velocity constraints \eqref{eq_vel_constraint} instead of the constraints on position level \eqref{eq_pos_constraint}. Care has to be taken when it comes to the numerical discretization in order to avoid the violation of the constraints on position level during simulations (\textit{drift-off}).

\subsection{Structure-preserving time integration of multibody systems} \label{sec_mbs_integrator}

\noindent For the time discretization of \eqref{block_MBS_detailled}, we propose the application of the semi-explicit discrete gradient method \eqref{block_pHDAE_timestepping_2} with additional specifications, leading to the discrete time mapping
\begin{subequations}
	\label{MBS_integrator_DG}
	\begin{align}
		\begin{bmatrix}
			I & 0 & 0 \\ 0 & M & 0 \\ 0 & 0 & 0
		\end{bmatrix} \begin{bmatrix}
			              {q}\none - {q}\n \\ {v}\none - {v}\n \\ {\lambda}\none -{\lambda}\n
		              \end{bmatrix} & = h \begin{bmatrix}
			                                  0  & I                                   & 0                                           \\
			                                  -I & -R_{\mathrm{R}}({q}\nonehalf)       & -\discretejacobian g({q}\n,{q}\none)\transp \\
			                                  0  & \discretejacobian g({q}\n,{q}\none) & 0
		                                  \end{bmatrix} \begin{bmatrix}
			                                                \discretecostateoneq \\ \discretecostateonev \\ \discretetimestep{\lambda}
		                                                \end{bmatrix} + h \begin{bmatrix}
			                                                                  0 \\ I \\ 0
		                                                                  \end{bmatrix} \discreteinput , \label{MBS_integrator_DG_a} \\
		\begin{bmatrix}
			\discretecostateoneq \\ \discretecostateonev
		\end{bmatrix}                        & =  \begin{bmatrix}
			                                          I & 0 \\ 0 & M
		                                          \end{bmatrix}\ntransp  \DG \specified{\hamiltonian}(x_1\n,x_1\none) ,
		\\
		\discreteoutput                                                     & = \discretecostateonev ,
	\end{align}
\end{subequations}
where $q\nonehalf=\frac{1}{2}(q\n+q\none)$ is the midpoint, and $\discretejacobian{g}$ is a discrete Jacobian for $g$.
Here we discretized the Rayleigh dissipation term using the implicit midpoint rule, but we emphasize that any other consistent approximation which preserves the positive semi-definiteness of $R_{\mathrm{R}}$ would be suitable as well.
For the multipliers we make the choice $\discretetimestep{\lambda} := \lambda\none$ such that no appropriate initialization for $\lambda^0$ is required.

As already mentioned, one might wonder about the drift-off effect.
Since we approximate $\jacobian{g}$ with a discrete Jacobian, combining
the first and third row of \eqref{MBS_integrator_DG_a} yields
\begin{align} \label{eq_no_drift}
	g(q\none) - g(q\n)
	= \discretejacobian g(q\n,q\none)(q\none-q\n)
	= h \discretejacobian g(q\n,q\none) \discretecostateonev
	= 0,
\end{align}
and therefore the drift-off vanishes, as long as the initial condition satisfies $g(q^0)=0$.
Thus,%
this scheme not only yields energy consistency in terms of \Cref{theorem_block_energy}, but also prevents the drift-off effect.

\begin{remark}\label{rem_velocity_constraint_discrete}
	The choice of using a discrete Jacobian to approximate $\jacobian{g}$ in general only guarantees that the velocity constraint \eqref{eq_vel_constraint} itself is satisfied approximately. For an energy-consistent multibody system integrator, which captures constraints both on position and on velocity level exactly, the interested reader is referred to \cite[Sec.~5]{kinon_2023_structurepreserving}.
\end{remark}

\begin{remark}\label{rem_gequiv}
	For lossless systems, one might be additionally interested in preserving momentum maps, like the angular momentum. This can be achieved by using \emph{G-equivariant} discrete gradients, see \cite[Ch.~3.7]{gonzalez_1996_time}.
\end{remark}

\begin{figure}[b]
	\centering
	\includegraphics{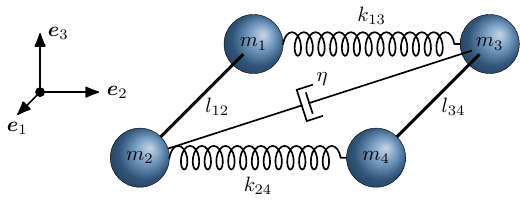}
	\caption{Four-particle system}
	\label{fig:4Psystem_sketch}
\end{figure}

\subsection{Numerical experiment} \label{sec_mbs_experiment}
\subsubsection{Problem description}

\noindent The four-particle system depicted in \Cref{fig:4Psystem_sketch} has been adapted from the literature \cite{gonzalez_1999_mechanical,kinon_2023_structurepreserving} and extended to include dissipation. The configuration of the system is characterized by the coordinate vector $q=(q_1,q_2,q_3,q_4) \in \R^{12}$ comprising the Cartesian coordinates of four masses $m_i$, $i=1,...,4$ in three dimensions.
Two nonlinear springs give rise to the potential function
\begin{equation} \label{eq:potential-4mass-system}
	V(q) = \frac{1}{2} k_{13} \pset[\big]{ \norm{q_3 - q_1}^2 - 1 }^2
	+ \frac{1}{2} k_{24} \pset[\big]{ \norm{q_4 - q_2}^2 - 1 }^2 ,
\end{equation}
with the spring stiffness parameters $k_{13}$ and $k_{24}$.
The mass matrix is block diagonal, i.e. $M = \diag\{m_1 I, m_2 I, m_3 I, m_4 I\}$. Additionally, we consider configuration-dependent viscous dissipation in terms of the Rayleigh dissipation function
\begin{align}
	G(q,v) = \frac{1}{2} \eta(q) v_{\mathrm{rel}}^2, \qquad
	v_{\mathrm{rel}} = \norm{v_3 -v_2},
\end{align}
where $\eta(q) = \eta_0 (1 + \alpha q_{\mathrm{rel}}^2) \geq 0$ is the dynamic viscosity parameter and
$q_{\mathrm{rel}}=\norm{q_3-q_2}$.
We have also introduced $\eta_0 >0 $ and $\alpha >0$ as constant parameters.
This leads to the dissipation matrix
\begin{align}
	R_{\mathrm{R}}(q) & = \eta(q) \begin{bmatrix}
		                              0 & 0  & 0  & 0 \\
		                              0 & I  & -I & 0 \\
		                              0 & -I & I  & 0 \\
		                              0 & 0  & 0  & 0
	                              \end{bmatrix}.
\end{align}
There are two rigid bars  connecting two masses, respectively, leading to the constraints on position level given by 
\begin{equation} \label{ex_constraints}
	g_1(q) = \frac{1}{2}\pset[\big]{ \norm{q_2-q_1}^2 - 1 } =0, \qquad
	g_2(q) = \frac{1}{2}\pset[\big]{ \norm{q_4-q_3}^2 - 1 } =0.
\end{equation}
In the numerical simulations the initial conditions
\begin{equation}
	\begin{aligned}
		{q}_1^0 & = \begin{bmatrix} 0 , 0 , 0 \end{bmatrix}\transp  \ , \ {q}_2^0 = \begin{bmatrix} 1 , 0 , 0 \end{bmatrix}\transp  \ , \  {q}_3^0 = \begin{bmatrix} 0 , 1 , 0 \end{bmatrix}\transp  \  , \  {q}_4^0 = \begin{bmatrix} 1 , 1 , 0 \end{bmatrix}\transp  ,          \\
		{v}_1^0 & = \begin{bmatrix} 0 , 0 , 0 \end{bmatrix}\transp \ , \ {v}_2^0 = \begin{bmatrix} 0 , 0 , 0 \end{bmatrix}\transp  \ , \  {v}_3^0 = \begin{bmatrix} 0 , 0 , 0 \end{bmatrix}\transp  \ , \ {v}_4^0 = \begin{bmatrix} 0 , 0 , \frac{20}{17} \end{bmatrix}\transp  ,
	\end{aligned}
\end{equation}
have been chosen consistently with the constraints \eqref{ex_constraints} and their velocity level counterparts induced by \eqref{eq_vel_constraint}.

\subsubsection{Methods}

\noindent The simulations have been conducted using our discrete gradient scheme for semi-explicit systems \eqref{block_pHDAE_timestepping_2}, implemented as suggested in \eqref{MBS_integrator_DG}. Results obtained with this scheme are labeled ``DG'' . The equations have been solved in each time step using Newton's method with a tolerance of $\epsilon_{\mathrm{Newton}}$. For the discrete gradients and Jacobians, we use the Gonzalez discrete gradient \eqref{eq:DD-Gonzalez}.
Since the constraints \eqref{ex_constraints} are quadratic, the application of the Gonzales discrete Jacobian boils down to a midpoint evaluation. In this example we assume zero inputs.

The generated data along with the source code for the simulations are openly available for verification purposes in the repository \url{https://github.com/plkinon/phdae_discrete_gradients} and are archived at \cite{kinon_2025_15007242}.

\subsubsection{Results \& Discussion}

\begin{table}[t]
	\renewcommand{\arraystretch}{1.2}
	\begin{center}
		\begin{tabular}{l l l l l l l}
			\toprule
			$h$    & $\finaltime$ & $\epsilon_{\mathrm{Newton}}$ & $\{k_{13}, k_{24}\}$ & $m_i$             & $\eta_0$ & $\alpha$ \\
			\midrule
			$0.01$ & $10$         & $10^{-10}$                   & $\{ 50, 500 \}$           & $\{1,3,2.3,1.7\}$ & $1$      & $0.5$    \\
			\bottomrule
		\end{tabular}
		\caption{Simulation parameters for four-particle system.}
		\label{tab_4Psystem}
	\end{center}
\end{table}

\begin{figure}[t]
	\noindent%
	\begin{minipage}{.29\textwidth}
		\includegraphics{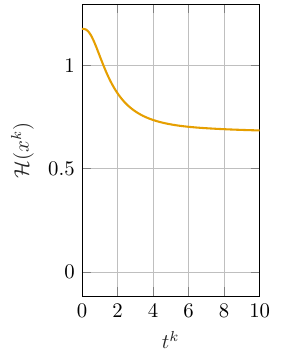}%
	\end{minipage}%
	\begin{minipage}{.39\textwidth}
		\vspace*{-4.3mm}
		\includegraphics{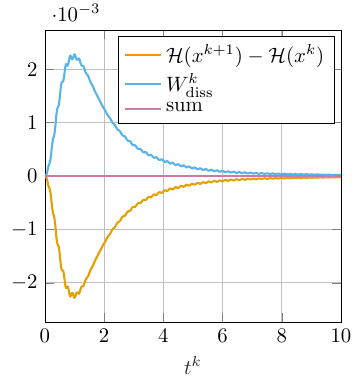}%
	\end{minipage}%
	\begin{minipage}{.32\textwidth}
		\vspace*{-3mm}
		\includegraphics{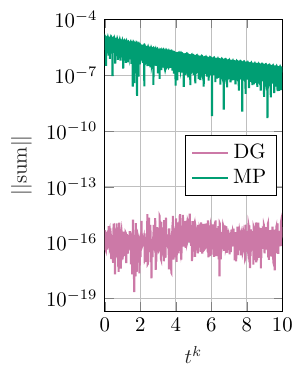}%
	\end{minipage}%
	\caption{Hamiltonian evolution (left), increments (center) and comparison with midpoint approach (right). \lq\lq DG\rq\rq denotes our approach and \lq\lq MP\rq\rq is the midpoint scheme.}
	\label{fig_energy_with_dissipation}
\end{figure}

\noindent We simulate the four-particle system using the parameters comprised in \Cref{tab_4Psystem}.
On the left part of \Cref{fig_energy_with_dissipation} one can observe the discrete evolution of the Hamiltonian in time.
The exact representation of the power balance in discrete time  (see \eqref{MBS_integrator_DG} and \Cref{theorem_block_energy}) is demonstrated in the central part of \Cref{fig_energy_with_dissipation}, since the Hamiltonian increments are always less or equal to zero and the dissipated work in each time step
\begin{equation}
	\discretedissipation :=  h \begin{bmatrix}  \discretecostateone \\ \discretecostate_2(\discretestate\n,\discretestate\none) \end{bmatrix} \transp\discreteR(\discretestate\n,\discretestate\none)  \begin{bmatrix} \discretecostateone \\ \discretecostate_2(\discretestate\n,\discretestate\none)\end{bmatrix} = h \discretecostateonev\transp R_{\mathrm{R}}({q}\nonehalf) \discretecostateonev \geq 0
\end{equation}
is equally large.
The sum of the two terms is numerically zero. 
For comparison, a pure midpoint-based scheme (labeled ``MP'') does not achieve energy-consistency, as depicted on the right part of \Cref{fig_energy_with_dissipation}.

On the left side of \Cref{fig_constraints_with_dissipation} one can observe that the scheme under investigation does not suffer from drift-off, i.e., it accurately captures the constraints on position level \eqref{ex_constraints}, as expected from \eqref{eq_no_drift}. On the right side of the same figure, the kinematic constraint is shown to have order of magnitude of $ 10^{-4}$ for each discrete point in time, due to the intermediate approximation of \eqref{eq_vel_constraint}, as discussed in Remark~\eqref{rem_velocity_constraint_discrete}.
Next, we switch off viscous dissipation by setting $\eta_0=0$ in the dissipation law. The discrete-time energy conservation in the non-dissipative case is verified in \Cref{fig_energy_without_dissipation}.

\begin{figure}[t]
	\begin{minipage}{.45\textwidth}
		\centering
		\includegraphics{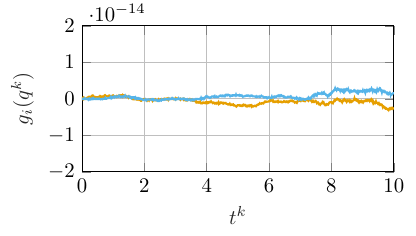}
	\end{minipage}%
	\begin{minipage}{.45\textwidth}
		\centering
		\includegraphics{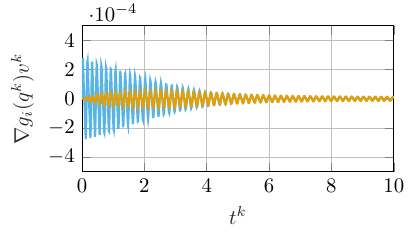}
	\end{minipage}%
	\caption{Constraints on position level (left) and on velocity level (right); \textcolor{color1}{\rule[2.5pt]{15pt}{1pt}} $i=1$, \textcolor{color2}{\rule[2.5pt]{15pt}{1pt}} $i=2$.
	}
	\label{fig_constraints_with_dissipation}
\end{figure}

\begin{figure}[b]
	\begin{minipage}{.45\textwidth}
		\centering
		\includegraphics{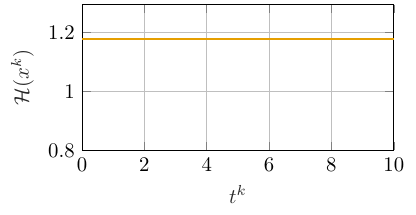}
	\end{minipage}
	\begin{minipage}{.45\textwidth}
		\centering
		\vspace*{-4.3mm}
		\includegraphics{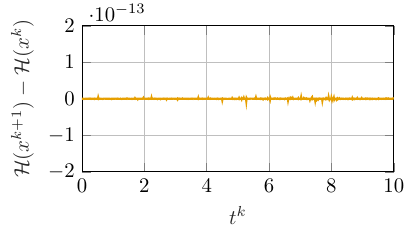}
	\end{minipage}
	\caption{Hamiltonian evolution (left) and increments (right) without dissipation, i.e., $\eta_0 = 0$.}
	\label{fig_energy_without_dissipation}
\end{figure}

\begin{figure}[t]
	\noindent%
	\begin{minipage}{.63\textwidth}
		\centering
		\includegraphics{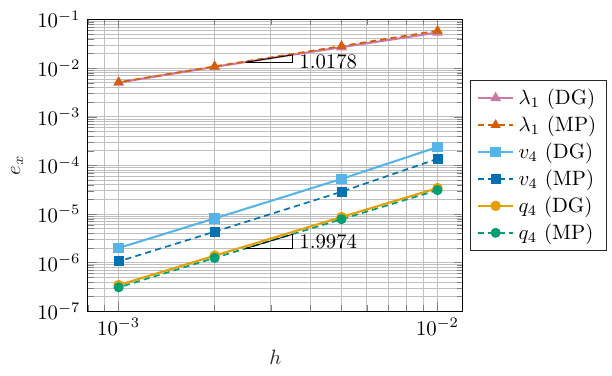}
	\end{minipage}%
	\begin{minipage}{.37\textwidth}
		\centering
		\includegraphics{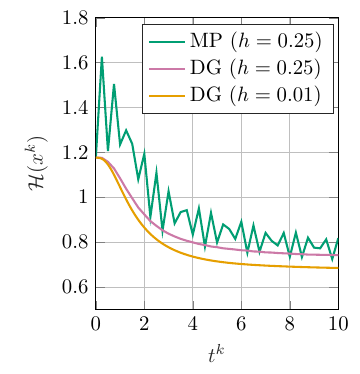}
	\end{minipage}%
	\caption{$h$-convergence (left) and robustness comparison with larger $h$ (right).}
	\label{fig_hconvergence_blowup}
\end{figure}

We have also performed a numerical convergence analysis (see \Cref{fig_hconvergence_blowup}, left side) using the relative error measure
\begin{equation*}
	e_x = \frac{|| x_{\mathrm{ref}} - x||}{|| x_{\mathrm{ref}}||} ,
\end{equation*}
where $x \in \{{q}_4\n, {v}_4\n, \discretetimestep{\lambda_1}\}$ are the solutions evaluated at $t\n = 0.1$ for different time step sizes and methods. The respective reference solution $x_{\mathrm{ref}}$ was obtained using our DG method with $h=10^{-4}$. The scheme exhibits second order accuracy for the differential unknowns of position and velocity and approximately a first order convergence behavior in the Lagrange multiplier.
It should be noted that the order of convergence can be affected by the discretization method used for the equation coefficients.
The compared midpoint scheme MP yields similar convergence results.

Lastly, we investigated the numerical robustness of the proposed DG scheme compared to the MP scheme when choosing larger time step sizes. Both methods did not converge with time step sizes of $0.4$ or larger. However, our DG method provided physically meaningful results up until $h=0.25$, showing a qualitatively similar behavior as with the previous simulation
(see \Cref{fig_hconvergence_blowup}, right side). Contrarily to that, MP exhibits occasional total energy increase, thus violating the dissipativity of the system.
This is in accordance with literature showing that discrete gradient methods are relatively robust compared to the midpoint rule in the nonlinear regime, see e.g.~\cite{gonzalez_1996_stability,kinon_2023_structurepreserving}.

\section{Conclusion}\label{sec_Conclusion}

\noindent In this work, we introduced discrete gradient methods for port-Hamiltonian differential-algebraic equations (pHDAEs), addressing the challenges associated with state-dependent and non-invertible descriptor matrices. We developed a promising time integration method for semi-explicit systems, discussed more general pHDAEs, and explored a method based on an alternative representation of pHDAEs. Additionally, we outlined conditions for constructing discrete gradient pairs for general pHDAEs, analyzed state transformations and the equivalence of different methods.
In particular we proved that, under appropriate regularity assumptions, every pHDAE can be reinterpreted as the combination of a parametrized semi-explicit pHDAE and an unstructured DAE on the time-varying parameter.
Lastly, we applied the proposed framework to the important application case of nonlinear multibody system dynamics, providing convincing simulation results.

Future research could focus on refining the conditions needed to apply discrete gradients to general pHDAEs, improving the numerical efficiency and accuracy. It will be of major interest to apply our framework to large-scale and multiphysics systems, as their modeling is seamlessly possible in the pH framework.
Moreover, a rigorous convergence analysis of the time discretization schemes presented in this paper should be pursued in the future.

\newenvironment{authcontrib}[1]{%
	\subsection*{\textnormal{\textbf{Author Contributions}}}%
	\noindent #1}%
{}%
\newenvironment{acks}[1]{%
	\subsection*{\textnormal{\textbf{Acknowledgements}}}%
	\noindent #1}%
{}%
\newenvironment{funding}[1]{%
	\subsection*{\textnormal{\textbf{Funding}}}%
	\noindent #1}%
{}%
\newenvironment{dci}[1]{%
	\subsection*{\textnormal{\textbf{Declaration of conflicting interests}}}%
	\noindent #1}%
{}%
\newenvironment{code}[1]{%
	\subsection*{\textnormal{\textbf{Code}}}%
	\noindent #1}%
{}%

\begin{acks}
	R. Morandin is funded by the Deutsche Forschungsgemeinschaft (DFG, German Research Foundation) -- 446856041.
	P.~L. Kinon gratefully acknowledges funding by the Research Travel Grant of the Karlsruhe House of Young
	Scientists (KHYS).
\end{acks}
\begin{dci}
	The authors declare no conflict of interest.
\end{dci}
\begin{code}
	The integration methods as well as the definition of the system have been implemented in the openly available Python package \texttt{pydykit}, which can be found at
	\url{https://github.com/pydykit/pydykit}
	and is archived under \cite{kinon_2025_pydykit}. The data generated for this work and the underlying source code for the simulations can be found in the repository
	\url{https://github.com/plkinon/phdae_discrete_gradients} and are archived at \cite{kinon_2025_15007242}.

\end{code}

\printbibliography

@article{aoues_2017_hamiltonian,
  title      = {Hamiltonian Systems Discrete-Time Approximation: {{Losslessness}}, Passivity and Composability},
  shorttitle = {Hamiltonian Systems Discrete-Time Approximation},
  author     = {Aoues, Said and {di Loreto}, Michael and Eberard, Damien and {Marquis-Favre}, Wilfrid},
  year       = {2017},
  month      = dec,
  journal    = {Syst. Control Lett.},
  volume     = {110},
  pages      = {9--14},
  publisher  = {Elsevier BV},
  issn       = {0167-6911},
  doi        = {10.1016/j.sysconle.2017.10.003},
  urldate    = {2024-09-17},
  copyright  = {https://www.elsevier.com/tdm/userlicense/1.0/},
  langid     = {english}
}

@book{arnold_1989_mathematical,
  title      = {Mathematical {{Methods}} of {{Classical Mechanics}}},
  author     = {Arnold, V. I.},
  translator = {Weinstein, A. and Vogtmann, K.},
  year       = {1989},
  month      = may,
  edition    = {2nd edition},
  publisher  = {Springer},
  address    = {New York},
  isbn       = {978-0-387-96890-2},
  langid     = {english},
  doi        = {10.1007/978-1-4757-2063-1}
}

@article{BarS25,
  title   = {Goal-oriented time adaptivity for port-{H}amiltonian systems},
  author  = {Bartel, A. and Schaller, M.},
  year    = {2025},
  journal = {J. Comput. Appl. Math.},
  volume  = {461},
  pages   = {116450},
  doi     = {10.1016/j.cam.2024.116450}
}

@misc{kinon_2025_15007242,
  author    = {Kinon, Philipp L.},
  title     = {{GitHub} repository: plkinon/phdae\_discrete\_gradients: v0.0.1},
  year      = 2025,
  publisher = {Zenodo},
  doi       = {10.5281/zenodo.15007242}
}

@misc{kinon_2025_pydykit,
  title     = {{GitHub} repository: pydykit/pydykit: v0.0.6},
  author    = {Kinon, Philipp L. and Bauer, Julian K.},
  year      = {2025},
  publisher = {Zenodo},
  doi       = {10.5281/zenodo.14849865}
}

@misc{BarDFG23,
  title         = {Operator splitting for semi-explicit differential-algebraic equations and {port-Hamiltonian} {DAEs}},
  author        = {Andreas Bartel and Malak Diab and Andreas Frommer and Michael Günther},
  year          = {2023},
  eprint        = {2308.16736},
  archiveprefix = {arXiv},
  doi           = {10.48550/arXiv.2308.16736}
}

@article{BarDFGM25,
  title    = {Splitting techniques for {DAEs} with {port-Hamiltonian} applications},
  journal  = {Appl. Numer. Math.},
  volume   = {214},
  pages    = {28-53},
  year     = {2025},
  issn     = {0168-9274},
  doi      = {10.1016/j.apnum.2025.03.004},
  author   = {Andreas Bartel and Malak Diab and Andreas Frommer and Michael Günther and Nicole Marheineke}
}

@article{beattie_2018_linear,
  title     = {Linear Port-{{Hamiltonian}} Descriptor Systems},
  author    = {Beattie, Christopher and Mehrmann, Volker and Xu, Hongguo and Zwart, Hans},
  year      = {2018},
  journal   = {Math. Control Signals Syst.},
  volume    = {30},
  pages     = {1--27},
  publisher = {Springer},
  doi       = {10.1007/s00498-018-0223-3}
}

@article{betsch_2000_conservation,
  title      = {Conservation Properties of a Time {{FE}} Method. {{Part I}}: Time-Stepping Schemes {{for $N$-body}} Problems},
  shorttitle = {Conservation Properties of a Time {{FE}} Method. {{Part I}}},
  author     = {Betsch, P. and Steinmann, P.},
  year       = {2000},
  month      = oct,
  journal    = {Int. J. Numer. Methods Eng.},
  volume     = {49},
  number     = {5},
  pages      = {599--638},
  issn       = {0029-5981, 1097-0207},
  doi        = {10.1002/1097-0207(20001020)49:5<599::AID-NME960>3.0.CO;2-9},
  urldate    = {2024-10-23},
  copyright  = {http://doi.wiley.com/10.1002/tdm\_license\_1.1},
  langid     = {english}
}

@article{betsch_2001_conservation,
  title      = {Conservation Properties of a Time {{FE}} Method --- Part {{II}}: {{Time}}-stepping Schemes for Non-linear Elastodynamics},
  shorttitle = {Conservation Properties of a Time {{FE}} Method---Part {{II}}},
  author     = {Betsch, P. and Steinmann, P.},
  year       = {2001},
  month      = mar,
  journal    = {Int. J. Numer. Methods Eng.},
  volume     = {50},
  number     = {8},
  pages      = {1931--1955},
  publisher  = {Wiley},
  issn       = {0029-5981, 1097-0207},
  doi        = {10.1002/nme.103},
  urldate    = {2024-09-17},
  copyright  = {http://onlinelibrary.wiley.com/termsAndConditions\#vor},
  langid     = {english}
}

@article{betsch_2002_conservation,
  title      = {Conservation Properties of a Time {{FE}} Method --- Part {{III}}: {{Mechanical}} Systems with Holonomic Constraints},
  shorttitle = {Conservation Properties of a Time {{FE}} Method?},
  author     = {Betsch, P. and Steinmann, P.},
  year       = {2002},
  month      = apr,
  journal    = {Int. J. Numer. Methods Eng.},
  volume     = {53},
  number     = {10},
  pages      = {2271--2304},
  issn       = {0029-5981, 1097-0207},
  doi        = {10.1002/nme.347},
  urldate    = {2022-01-17},
  langid     = {english}
}

@book{betsch_2016_structurepreserving,
  title     = {Structure-Preserving {{Integrators}} in {{Nonlinear Structural Dynamics}} and {{Flexible Multibody Dynamics}}},
  editor    = {Betsch, Peter},
  year      = {2016},
  series    = {{{CISM International Centre}} for {{Mechanical Sciences}}},
  volume    = {565},
  publisher = {Springer International Publishing},
  address   = {Cham},
  doi       = {10.1007/978-3-319-31879-0},
  urldate   = {2022-01-17},
  isbn      = {978-3-319-31877-6 978-3-319-31879-0},
  langid    = {english}
}

@misc{celledoni2017energy,
  title         = {Energy-preserving and passivity-consistent numerical discretization of port-{H}amiltonian systems},
  author        = {Elena Celledoni and Eirik Hoel Høiseth},
  year          = {2017},
  eprint        = {1706.08621},
  archiveprefix = {arXiv},
  primaryclass  = {math.NA},
  doi           = {10.48550/arXiv.1706.08621}
}

@book{duindam_2009_modeling,
  title      = {Modeling and {{Control}} of {{Complex Physical Systems}}: {{The Port-Hamiltonian Approach}}},
  shorttitle = {Modeling and {{Control}} of {{Complex Physical Systems}}},
  author     = {Duindam, Vincent and Macchelli, Alessandro and Stramigioli, Stefano and Bruyninckx, Herman},
  year       = {2009},
  month      = jan,
  journal    = {Modeling and Control of Complex Physical Systems: The Port-Hamiltonian Approach},
  publisher  = {Springer Berlin Heidelberg},
  doi        = {10.1007/978-3-642-03196-0},
  isbn       = {978-3-642-03195-3}
}

@article{EggHS21,
  title   = {On the Energy Stable Approximation of {{Hamiltonian}} and Gradient Systems},
  author  = {Egger, H. and Habrich, O. and Shashkov, V.},
  year    = {2021},
  journal = {Comput. Methods Appl. Math.},
  volume  = {21},
  number  = {2},
  pages   = {335--349},
  doi     = {10.1515/cmam-2020-0025}
}

@article{eidnes_2022_order,
  title    = {Order Theory for Discrete Gradient Methods},
  author   = {Eidnes, S{\o}lve},
  year     = {2022},
  month    = dec,
  journal  = {BIT Numer. Math.},
  volume   = {62},
  number   = {4},
  pages    = {1207--1255},
  issn     = {0006-3835, 1572-9125},
  doi      = {10.1007/s10543-022-00909-z},
  urldate  = {2024-10-23},
  langid   = {english}
}

@article{falaize_2016_passive,
  title      = {Passive Guaranteed Simulation of Analog Audio Circuits: A Port-Ham\-il\-to\-nian Approach},
  shorttitle = {Passive {{Guaranteed Simulation}} of {{Analog Audio Circuits}}},
  author     = {Falaize, Antoine and H{\'e}lie, Thomas},
  year       = {2016},
  month      = sep,
  journal    = {Appl. Sci.},
  volume     = {6},
  number     = {10},
  pages      = {273},
  publisher  = {MDPI AG},
  issn       = {2076-3417},
  doi        = {10.3390/app6100273},
  urldate    = {2024-09-17},
  copyright  = {https://creativecommons.org/licenses/by/4.0/},
  langid     = {english}
}

@article{fiaz_2013_porthamiltonian,
  title     = {A Port-{{Hamiltonian}} Approach to Power Network Modeling and Analysis},
  author    = {Fiaz, S. and Zonetti, D. and Ortega, R. and Scherpen, J.M.A. and {van der Schaft}, A.},
  year      = {2013},
  month     = dec,
  journal   = {Eur. J. Control},
  volume    = {19},
  number    = {6},
  pages     = {477--485},
  issn      = {09473580},
  doi       = {10.1016/j.ejcon.2013.09.002},
  urldate   = {2024-10-22},
  copyright = {https://www.elsevier.com/tdm/userlicense/1.0/},
  langid    = {english}
}

@misc{FroGLM24,
  title         = {Operator splitting for port-{H}amiltonian systems},
  author        = {Andreas Frommer and Michael Günther and Björn Liljegren-Sailer and Nicole Marheineke},
  year          = {2023},
  eprint        = {2304.01766},
  archiveprefix = {arXiv},
  primaryclass  = {math.NA},
  doi           = {10.48550/arXiv.2304.01766}
}

@misc{GieKT24,
  title         = {Energy-consistent {P}etrov-{G}alerkin time discretization of port-{H}amiltonian systems},
  author        = {Jan Giesselmann and Attila Karsai and Tabea Tscherpel},
  year          = {2024},
  eprint        = {2404.12480},
  archiveprefix = {arXiv},
  primaryclass  = {math.NA},
  doi           = {10.48550/arXiv.2404.12480}
}

@article{gonzalez_1996_stability,
  title    = {On the Stability of Symplectic and Energy-Momentum Algorithms for Non-Linear {{Hamiltonian}} Systems with Symmetry},
  author   = {Gonzalez, O. and Simo, J.C.},
  year     = {1996},
  month    = aug,
  journal  = {Comput. Methods Appl. Mech. Eng.},
  volume   = {134},
  number   = {3-4},
  pages    = {197--222},
  issn     = {00457825},
  doi      = {10.1016/0045-7825(96)01009-2},
  urldate  = {2022-01-17},
  langid   = {english}
}

@article{gonzalez_1996_time,
  title    = {Time Integration and Discrete {{Hamiltonian}} Systems},
  author   = {Gonzalez, O.},
  year     = {1996},
  month    = sep,
  journal  = {J. Nonlinear Sci.},
  volume   = {6},
  number   = {5},
  pages    = {449--467},
  issn     = {1432-1467},
  doi      = {10.1007/BF02440162}
}

@article{gonzalez_1999_mechanical,
  title      = {Mechanical Systems Subject to Holonomic Constraints: {{Differential}}--Algebraic Formulations and Conservative Integration},
  shorttitle = {Mechanical Systems Subject to Holonomic Constraints},
  author     = {Gonzalez, Oscar},
  year       = {1999},
  month      = jul,
  journal    = {Physica D},
  volume     = {132},
  number     = {1-2},
  pages      = {165--174},
  issn       = {01672789},
  doi        = {10.1016/S0167-2789(99)00054-8},
  urldate    = {2022-01-17},
  langid     = {english}
}

@article{gonzalez_2000_exact,
  title     = {Exact Energy and Momentum Conserving Algorithms for General Models in Nonlinear Elasticity},
  author    = {Gonzalez, O.},
  year      = {2000},
  month     = dec,
  journal   = {Comput. Methods Appl. Mech. Eng.},
  volume    = {190},
  number    = {13-14},
  pages     = {1763--1783},
  publisher = {Elsevier BV},
  issn      = {0045-7825},
  doi       = {10.1016/s0045-7825(00)00189-4},
  urldate   = {2024-09-17},
  copyright = {https://www.elsevier.com/tdm/userlicense/1.0/},
  langid    = {english}
}

@article{goren-sumer_2008_gradient,
  title     = {Gradient {{Based Discrete-Time Modeling}} and {{Control}} of {{Hamiltonian Systems}}},
  author    = {{G{\"o}ren-S{\"u}mer}, Leyla and Yal{\c c}{$\iota$}n, Yaprak},
  year      = {2008},
  journal   = {IFAC Proc. Vol.},
  volume    = {41},
  number    = {2},
  pages     = {212--217},
  publisher = {Elsevier BV},
  issn      = {1474-6670},
  doi       = {10.3182/20080706-5-kr-1001.00036},
  urldate   = {2024-09-17},
  copyright = {https://www.elsevier.com/tdm/userlicense/1.0/},
  langid    = {english}
}

@book{Goursat59,
  author    = {Goursat, {\'E}douard},
  publisher = {Dover},
  year      = {1959},
  title     = {A Course in Mathematical Analysis. Vol. 2, Pt. 2: Differential Equations},
  language  = {eng}
}

@book{hairer_2006_geometric,
  title     = {Geometric Numerical Integration},
  author    = {Hairer, Ernst and Lubich, Christian and Wanner, Gerhard},
  year      = {2006},
  publisher = {Springer},
  address   = {Berlin},
  doi       = {10.1007/3-540-30666-8}
}

@book{HirS74,
  title     = {Differential Equations, Dynamical Systems, and Linear Algebra},
  author    = {Hirsch, M. W. and Smale, S.},
  year      = {1974},
  publisher = {Academic Press},
  address   = {New York, NY, USA}
}

@book{holm_2009_geometric,
  title      = {Geometric {{Mechanics}} and {{Symmetry}}: {{From Finite}} to {{Infinite Dimensions}}},
  shorttitle = {Geometric {{Mechanics}} and {{Symmetry}}},
  author     = {Holm, Darryl D. and Schmah, Tanya and Stoica, Cristina},
  year       = {2009},
  month      = oct,
  publisher  = {Oxford University Press},
  address    = {New York},
  isbn       = {978-0-19-921290-3},
  langid     = {english},
  doi        = {10.1093/oso/9780199212903.001.0001}
}

@phdthesis{Jon07,
  title   = {Systems of Linear First Order Partial Differential Equations Admitting a Bilinear Multiplication of Solutions},
  author  = {Jonasson, J.},
  year    = {2007},
  school  = {Link\"oping University},
  address = {Sweden},
  isbn    = {978-91-85895-78-6}
}

@article{JueST19,
  title   = {Two Structure-Preserving Time Discretizations for Gradient Flows},
  author  = {J{\"u}ngel, A. and Stefanelli, U. and Trussardi, L.},
  year    = {2019},
  journal = {Appl. Math. Optim.},
  volume  = {80},
  pages   = {733--764},
  doi     = {10.1007/s00245-019-09605-x}
}

@article{kinon_2023_discrete,
  title   = {Discrete Nonlinear Elastodynamics in a Port-{{Hamiltonian}} Framework},
  author  = {Kinon, Philipp L. and Thoma, Tobias and Betsch, Peter and Kotyczka, Paul},
  year    = {2023},
  journal = {PAMM (Proc. Appl. Math. Mech.)},
  volume  = {23},
  number  = {3},
  pages   = {e202300144},
  doi     = {10.1002/pamm.202300144}
}

@unpublished{MayBetsch2025,
  author = {May, M. and Betsch, P.},
  title  = {Galerkin-based time integration approaches to rigid body dynamics in terms of unit quaternions},
  year   = {2025},
  doi    = {10.21203/rs.3.rs-6354821/v1},
  note   = {{PREPRINT} available at Research Square}
}

@inproceedings{kinon_2023_porthamiltonian,
  title     = {Port-{{Hamiltonian}} Formulation and Structure-Preserving Discretization of Hyperelastic Strings},
  booktitle = {Proceedings of the 11th {{ECCOMAS Thematic Conference}} on {{Multibody Dynamics}}},
  author    = {Kinon, P. L. and Thoma, T. and Betsch, P. and Kotyczka, P.},
  year      = {2023},
  pages     = {1--10},
  address   = {Lisbon, Portugal},
  doi       = {10.48550/arXiv.2304.10957}
}

@article{kinon_2023_structurepreserving,
  title         = {Structure-Preserving Integrators Based on a New Variational Principle for Constrained Mechanical Systems},
  author        = {Kinon, Philipp L. and Betsch, P. and Schneider, S.},
  year          = {2023},
  journal       = {Nonlinear Dyn.},
  volume        = {111},
  pages         = {14231--14261},
  doi           = {10.1007/s11071-023-08522-7}
}

@article{kinon_2024_conserving,
  title    = {Conserving Integration of Multibody Systems with Singular and Non-Constant Mass Matrix Including Quaternion-Based Rigid Body Dynamics},
  author   = {Kinon, Philipp L. and Betsch, Peter},
  year     = {2024},
  journal  = {Multibody Sys. Dyn.},
  volume   = {63},
  number   = {1},
  pages    = {303--340},
  issn     = {1573-272X},
  doi      = {10.1007/s11044-024-10001-9}
}

@article{kinon_2024_generalized,
  title      = {Generalized {{Maxwell}} Viscoelasticity for Geometrically Exact Strings: {{Nonlinear}} Port-{{Hamiltonian}} Formulation and Structure-Preserving Discretization},
  shorttitle = {Generalized {{Maxwell}} Viscoelasticity for Geometrically Exact Strings},
  author     = {Kinon, Philipp L. and Thoma, T. and Betsch, P. and Kotyczka, P.},
  year       = {2024},
  journal    = {IFAC-PapersOnLine},
  volume     = {58},
  number     = {6},
  pages      = {101--106},
  issn       = {24058963},
  doi        = {10.1016/j.ifacol.2024.08.264},
  langid     = {english}
}

@article{KotL19,
  title   = {Discrete-Time Port-{{Hamiltonian}} Systems: A Definition Based on Symplectic Integration},
  author  = {Kotyczka, P. and Lef{\`e}vre, L.},
  year    = {2019},
  journal = {Syst. Control Lett.},
  volume  = {133},
  pages   = {104530},
  doi     = {10.1016/j.sysconle.2019.104530}
}

@book{kundur_1994_power,
  title     = {Power System Stability and Control},
  author    = {Kundur, P.},
  year      = {1994},
  series    = {The {{EPRI}} Power System Engineering Series},
  publisher = {McGraw-Hill},
  address   = {New York},
  isbn      = {978-0-07-035958-1},
  lccn      = {TK1005 .K86 1994}
}

@book{kunkel_2006_differentialalgebraic,
  title      = {Differential-Algebraic Equations: Analysis and Numerical Solution},
  shorttitle = {Differential-Algebraic Equations},
  author     = {Kunkel, Peter and Mehrmann, Volker},
  year       = {2006},
  series     = {{{EMS}} Textbooks in Mathematics},
  publisher  = {European Mathematical Society Publ. House},
  address    = {Z{\"u}rich},
  isbn       = {978-3-03719-017-3},
  langid     = {english},
  doi        = {10.4171/017}
}

@article{KunM23,
  title   = {Discretization of Inherent {{ODEs}} and the Geometric Integration of {{DAEs}} with Symmetries},
  author  = {Kunkel, P. and Mehrmann, V.},
  year    = {2023},
  journal = {BIT Numer. Math.},
  volume  = {63},
  pages   = {29},
  doi     = {10.1007/s10543-023-00966-y}
}

@article{leimkuhler_1994_symplectic,
  title     = {Symplectic {{Numerical Integrators}} in {{Constrained Hamiltonian Systems}}},
  author    = {Leimkuhler, Benedict J. and Skeel, Robert D.},
  year      = {1994},
  month     = may,
  journal   = {J. Comput. Phys.},
  volume    = {112},
  number    = {1},
  pages     = {117--125},
  publisher = {Elsevier BV},
  issn      = {0021-9991},
  doi       = {10.1006/jcph.1994.1085},
  urldate   = {2024-09-17},
  copyright = {https://www.elsevier.com/tdm/userlicense/1.0/},
  langid    = {english}
}

@article{leyendecker_2008_variational,
  title   = {Variational Integrators for Constrained Dynamical Systems},
  author  = {Leyendecker, S. and Marsden, J.E. and Ortiz, M.},
  year    = {2008},
  month   = sep,
  journal = {ZAMM (J. Appl. Math. Mech.)},
  volume  = {88},
  number  = {9},
  pages   = {677--708},
  issn    = {00442267, 15214001},
  doi     = {10.1002/zamm.200700173},
  urldate = {2022-01-17},
  langid  = {english}
}

@article{marsden_2001_discrete,
  title    = {Discrete Mechanics and Variational Integrators},
  author   = {Marsden, J. E. and West, M.},
  year     = {2001},
  month    = may,
  journal  = {Acta Numer.},
  volume   = {10},
  pages    = {357--514},
  issn     = {0962-4929, 1474-0508},
  doi      = {10.1017/S096249290100006X},
  urldate  = {2022-01-17},
  langid   = {english}
}

@article{McLQR99,
  title   = {Geometric Integration Using Discrete Gradients},
  author  = {McLachlan, R. I. and Quispel, G. R. W. and Robidoux, N.},
  year    = {1999},
  journal = {Phil. Trans. R. Soc. A},
  volume  = {357},
  number  = {1754},
  pages   = {1021--1045},
  doi     = {10.1098/rsta.1999.0363}
}

@inproceedings{mehrmann_2019_structurepreserving,
  title     = {Structure-Preserving Discretization for Port-{{Hamiltonian}} Descriptor Systems},
  booktitle = {2019 {{IEEE}} 58th {{Conference}} on {{Decision}} and {{Control}} ({{CDC}})},
  author    = {Mehrmann, Volker and Morandin, Riccardo},
  year      = {2019},
  pages     = {6863--6868},
  publisher = {IEEE},
  doi       = {10.1109/CDC40024.2019.9030180}
}

@article{mehrmann_2023_control,
  title     = {Control of Port-{{Hamiltonian}} Differential-Algebraic Systems and Applications},
  author    = {Mehrmann, Volker and Unger, Benjamin},
  year      = {2023},
  month     = may,
  journal   = {Acta Numer.},
  volume    = {32},
  pages     = {395--515},
  issn      = {0962-4929, 1474-0508},
  doi       = {10.1017/S0962492922000083},
  urldate   = {2024-11-18},
  copyright = {https://www.cambridge.org/core/terms},
  langid    = {english}
}

@article{MonM25,
  title    = {Commutator-based operator splitting for linear {port-Hamiltonian} systems},
  journal  = {Appl. Numer. Math.},
  volume   = {210},
  pages    = {25-38},
  year     = {2025},
  issn     = {0168-9274},
  doi      = {10.1016/j.apnum.2024.12.007},
  author   = {Marius Mönch and Nicole Marheineke}
}

@phdthesis{morandin_phd_2019,
  title  = {Modeling and Numerical Treatment of Port-{{Hamiltonian}} Descriptor Systems},
  author = {Morandin, Riccardo},
  year   = {2023},
  school = {TU Berlin},
  doi    = {10.14279/depositonce-19826}
}

@inproceedings{MorMMN19,
  title     = {Discrete Port-Controlled {{Hamiltonian}} Dynamics and Average Passivation},
  booktitle = {Proceedings of the 58th {{IEEE}} Conference on Decision and Control},
  author    = {Moreschini, A. and Mattioni, M. and Monaco, S. and {Normand-Cyrot}, D.},
  year      = {2019},
  pages     = {1430--1435},
  address   = {Nice, France},
  doi       = {10.1109/CDC40024.2019.9029809}
}

@article{Oet18,
  title   = {{{GENERIC}} Integrators: Structure Preserving Time Integration for Thermodynamic Systems},
  author  = {{\"O}ttinger, H. C.},
  year    = {2018},
  journal = {J. Non-Equilib. Thermodyn.},
  volume  = {43},
  number  = {2},
  pages   = {89--100},
  doi     = {10.1515/jnet-2017-0034}
}

@article{rabier_1994_impasse,
  title   = {On {{Impasse Points}} of {{Quasilinear Differential-Algebraic Equations}}},
  author  = {Rabier, P.J. and Rheinboldt, W.C.},
  year    = {1994},
  month   = jan,
  journal = {J. Math. Anal. Appl.},
  volume  = {181},
  number  = {2},
  pages   = {429--454},
  issn    = {0022-247X},
  doi     = {10.1006/jmaa.1994.1033}
}

@article{sato_2019_linear,
  title    = {Linear Gradient Structures and Discrete Gradient Methods for Conservative/Dissipative Differential-Algebraic Equations},
  author   = {Sato, Shun},
  year     = {2019},
  month    = dec,
  journal  = {BIT Numer. Math.},
  volume   = {59},
  number   = {4},
  pages    = {1063--1091},
  issn     = {1572-9125},
  doi      = {10.1007/s10543-019-00759-2}
}

@article{schaft_2014_porthamiltonian,
  title      = {Port-{{Hamiltonian Systems Theory}}: {{An Introductory Overview}}},
  shorttitle = {Port-{{Hamiltonian Systems Theory}}},
  author     = {van der Schaft, Arjan and Jeltsema, Dimitri},
  year       = {2014},
  month      = jun,
  journal    = {Found. Trends Syst. Control},
  volume     = {1},
  number     = {2-3},
  pages      = {173--378},
  publisher  = {Now Publishers, Inc.},
  issn       = {2325-6818, 2325-6826},
  doi        = {10.1561/2600000002},
  urldate    = {2022-06-15},
  langid     = {english}
}

@misc{schulze_structure_2023,
  title         = {Structure-preserving time discretization of port-{H}amiltonian systems via discrete gradient pairs},
  author        = {Philipp Schulze},
  year          = {2023},
  eprint        = {2311.00403},
  archiveprefix = {arXiv},
  primaryclass  = {math.NA},
  doi           = {10.48550/arXiv.2311.00403}
}

@article{SchM20,
  title   = {Dirac and {L}agrange algebraic constraints in nonlinear port-{H}amiltonian systems},
  author  = {van der Schaft, A. and Maschke, B.},
  year    = {2020},
  journal = {Vietnam J. Math.},
  volume  = {48},
  pages   = {929--939},
  doi     = {10.1007/s10013-020-00419-x}
}

@article{simo_1992_discrete,
  title     = {The Discrete Energy-Momentum Method. {{Conserving}} Algorithms for Nonlinear Elastodynamics},
  author    = {Simo, J. C. and Tarnow, N.},
  year      = {1992},
  month     = sep,
  journal   = {J. Appl. Math. Phys. (ZAMP)},
  volume    = {43},
  number    = {5},
  pages     = {757--792},
  publisher = {{Springer Science and Business Media LLC}},
  issn      = {0044-2275, 1420-9039},
  doi       = {10.1007/bf00913408},
  urldate   = {2024-09-17},
  copyright = {http://www.springer.com/tdm},
  langid    = {english}
}

@misc{simoes_2023_discrete,
  title         = {Discrete gradient methods for irreversible port-{H}amiltonian systems},
  author        = {Alexandre Anahory Simoes and David Martín de Diego and Bernhard Maschke},
  year          = {2023},
  eprint        = {2303.08034},
  archiveprefix = {arXiv},
  primaryclass  = {math.NA},
  doi           = {10.48550/arXiv.2303.08034}
}

@phdthesis{steinbrecher_2006_numerical,
  title        = {Numerical Solution of Quasi-Linear Differential-Algebraic Equations and Industrial Simulation of Multibody Systems},
  author       = {Steinbrecher, Andreas},
  year         = {2006},
  month        = may,
  urldate      = {2025-02-11},
  collaborator = {{Technische Universit{\"a}t Berlin}},
  langid       = {english},
  school       = {Technische Universit{\"a}t Berlin},
  doi          = {10.14279/depositonce-1360}
}

@article{udwadia_2006_explicit,
  title     = {Explicit Equations of Motion for Constrained Mechanical Systems with Singular Mass Matrices and Applications to Multi-Body Dynamics},
  author    = {Udwadia, Firdaus E and Phohomsiri, Phailaung},
  year      = {2006},
  month     = jul,
  journal   = {Proc. R. Soc. A},
  volume    = {462},
  number    = {2071},
  pages     = {2097--2117},
  issn      = {1364-5021, 1471-2946},
  doi       = {10.1098/rspa.2006.1662},
  urldate   = {2024-10-22},
  copyright = {https://royalsociety.org/journals/ethics-policies/data-sharing-mining/},
  langid    = {english}
}

@incollection{vanderschaft_2013_porthamiltonian,
  title     = {Port-{{Hamiltonian Differential-Algebraic Systems}}},
  booktitle = {Surveys in {{Differential-Algebraic Equations I}}},
  author    = {{van der Schaft}, A.},
  editor    = {Ilchmann, Achim and Reis, Timo},
  year      = {2013},
  pages     = {173--226},
  publisher = {Springer Berlin Heidelberg},
  address   = {Berlin, Heidelberg},
  doi       = {10.1007/978-3-642-34928-7_5},
  isbn      = {978-3-642-34928-7}
}

@article{vanderschaft_2018_generalized,
  title    = {Generalized Port-{{Hamiltonian DAE}} Systems},
  author   = {{van der Schaft}, Arjan and Maschke, Bernhard},
  year     = {2018},
  journal  = {Syst. Control Lett.},
  volume   = {121},
  pages    = {31--37},
  issn     = {0167-6911},
  doi      = {10.1016/j.sysconle.2018.09.008}
}

\appendix
\crefalias{section}{appendix}
\section*{Appendix}
\section[Proof details for Lemma~\ref{lem:specifiedDG}]{Proof details for \Cref{lem:specifiedDG}}
\label{appendix_lemma2_7}

\noindent To prove the first claim in \Cref{lem:specifiedDG}, we consider a fixed $\hat{\state}_2\in\statespace_2$ and define $\specified{f}:\statespace_1\to\R,\ x_1\mapsto f(x_1,\hat{\state}_2)$.
Since $\statespace_{2}$ is convex, we deduce that for every fixed $(x_1,x_2)\in\statespace$ the map
\[
	\hat f:[0,1]\to\R,\qquad s\mapsto f\pset[\big]{ x_1 , sx_2 + (1-s)\hat{\state}_2}
\]
is well-defined, continuously differentiable, and its derivative satisfies
\[
	\frac{\mathrm{d}\hat f}{\mathrm{d}s}(s) = \gradient_{x_2}f\pset[\big]{ x_1 , sx_2 + (1-s)\hat{\state}_2} \transp\pset[\big]{x_2-\hat{\state}_2} = 0
\]
for all $s\in[0,1]$, i.e., $\hat f$ is constant.
Thus,
\[
	f(x_1,x_2) = \hat f(1) = \hat f(0) = f\pset[\big]{x_1,\hat{\state}_2} = \specified{f}(x_1).
\]
We observe that for every $x_1\in\statespace_1$ and sufficiently small $h\in\R^{n_1}$ it holds that $(x_1+h,\hat{\state}_2)\in\statespace$, since $(x_1,\hat{\state}_2)\in\statespace$ and $\statespace$ is an open set. In particular, we can write
\[
	\specified{f}(x_1 + h) - \specified{f}(x_1) = f\pset[\big]{x_1+h,\hat{\state}_2} - f\pset[\big]{x_1,\hat{\state}_2}
\]
for every $x_1\in\statespace_1$ and every $h\in\R^{n_1}$ of appropriately bounded norm, from which we immediately deduce that $\specified{f}$ is continuously differentiable and $\gradient\specified{f}(x_1)=\gradient_{x_1}f(x_1,x_2)$ for every $(x_1,x_2)\in\statespace$.
The second claim of \Cref{lem:specifiedDG} is proven in the main part of the manuscript.

\section{Details on the synchronous machine} \label{appendix_synchro}

\noindent Another example for the present framework is the synchronous machine, see \Cref{ex_synchro}. The pH system, as shown in \cite{fiaz_2013_porthamiltonian} reads
\begin{subequations}
	\begin{align}\label{eq:syncMach_dyn}
		\begin{bmatrix}
			\dot\psi_s \\ \dot\psi_r \\ \dot p \\ \dot\theta
		\end{bmatrix}
		 & =
		\begin{bmatrix}
			-R_s & 0 & 0 & 0 \\ 0 & -R_r & 0 & 0 \\ 0 & 0 & -d & -1 \\ 0 & 0 & 1 & 0
		\end{bmatrix}
		\gradient\widetilde{\hamiltonian}(\psi_s,\psi_r,p,\theta) +
		\begin{bmatrix}
			I_3 & 0 & 0 \\ 0 & e_1 & 0 \\ 0 & 0 & 1 \\ 0 & 0 & 0
		\end{bmatrix}
		\begin{bmatrix}
			V_s \\ V_f \\ \tau
		\end{bmatrix}, \\ \label{eq:syncMach_out}
		\begin{bmatrix}
			I_s \\ I_f \\ \omega
		\end{bmatrix}
		 & =
		\begin{bmatrix}
			I_3 & 0 & 0 & 0 \\ 0 & e_1^\top & 0 & 0 \\ 0 & 0 & 1 & 0
		\end{bmatrix}
		\gradient\widetilde{\hamiltonian}(\psi_s,\psi_r,p,\theta),
	\end{align}
\end{subequations}
with Hamiltonian
\begin{equation}
	\widetilde{\hamiltonian}(\psi_s,\psi_r,p,\theta) = \frac{1}{2}
	\begin{bmatrix}
		\psi_s \\ \psi_r
	\end{bmatrix}^\top
	L(\theta)^{-1}
	\begin{bmatrix}
		\psi_s \\ \psi_r
	\end{bmatrix}
	+ \frac{1}{2J_r}p^2,
\end{equation}
from which
\begin{equation}
	\gradient\widetilde{\hamiltonian}(\psi_s,\psi_r,p,\theta) =
	\begin{bmatrix}
		L(\theta)^{-1}
		\begin{bmatrix}
			\psi_s \\ \psi_r
		\end{bmatrix} \\
		J_r^{-1}p       \\
		-\frac{1}{2}
		\begin{bmatrix}
			\psi_s \\ \psi_r
		\end{bmatrix}^\top
		L(\theta)^{-1}
		L'(\theta)
		L(\theta)^{-1}
		\begin{bmatrix}
			\psi_s \\ \psi_r
		\end{bmatrix}
	\end{bmatrix}.
\end{equation}
Therein, $\psi_s,\psi_r\in\R^3$ represent the stator and rotor fluxes, respectively. For the further variable declarations, please refer to Example~\ref{ex_synchro} in the bulk part of this work.
Suppose we now want to rewrite \eqref{eq:syncMach_dyn}, \eqref{eq:syncMach_out} replacing the magnetic fluxes with the corresponding currents, through the constitutive relation
\[
	\begin{bmatrix}
		\psi_s \\ \psi_r
	\end{bmatrix}
	= L(\theta)I .
\]
This change of variables allows us to rewrite the system equivalently in its pHDAE form given in \eqref{eq:syncMachAlt_dyn} and \eqref{eq:syncMachAlt_out} with the corresponding Hamiltonian \eqref{eq_syncMach_Ham}. These fit into the general pHDAE formulation with
\begin{align*}
	 & x =
	\begin{bmatrix}
		I \\ p \\ \theta
	\end{bmatrix}\in\R^8, \qquad
	u =
	\begin{bmatrix}
		V_s \\ V_f \\ \tau
	\end{bmatrix}\in\R^5, \qquad
	y =
	\begin{bmatrix}
		I_s \\ I_f \\ \omega
	\end{bmatrix}\in\R^5,        \\
	 & E(I,\theta) =
	\begin{bmatrix}
		L(\theta) & 0 & L'(\theta)I \\ 0 & 1 & 0 \\ 0 & 0 & 1
	\end{bmatrix} \in \R^{8,8}, \qquad
	z(I,p,\theta) =
	\begin{bmatrix}
		I         \\
		J_r^{-1}p \\
		\frac{1}{2}I^\top L'(\theta)I
	\end{bmatrix} \in \R^8, \\
	 & J =
	\begin{bmatrix}
		0 & 0 & 0 \\ 0 & 0 & -1 \\ 0 & 1 & 0
	\end{bmatrix} \in \R^{8,8}, \qquad
	R =
	\begin{bmatrix}
		R_{s,r} & 0 & 0 \\ 0 & d & 0 \\ 0 & 0 & 0
	\end{bmatrix} \in \R^{8,8}, \qquad
	B =
	\begin{bmatrix}
		I_3 & 0 & 0 & 0 \\ 0 & e_1^\top & 0 & 0 \\ 0 & 0 & 1 & 0
	\end{bmatrix} .
\end{align*}
\section{Further details on the DDR-method with singular descriptor matrix}\label{appendix_counter_example}

\noindent Towards the end of \Cref{sec_version3} we discussed possible limitations of the DDR-approach, questioning whether one can always find a discrete gradient that ensures the solvability of the system equation \eqref{eqn_colsp}, even for a singular descriptor matrix $E$.
This of course also depends on the choice of $\discreteE$, but for simplicity one would hope that for some suitable choice of $\discreteE$, such as the midpoint approximation
\begin{equation}
	\label{eq:midpoint_E}
	\discreteE(x,x') \vcentcolon= E(\tfrac{x+x'}2),
\end{equation}
one can always construct a discrete gradient satisfying \eqref{eqn_colsp2} for all $x,x'\in\R^n$.
The following example demonstrates that this is not in general the case when considering a fixed consistent approximation for $E$, like the one in \eqref{eq:midpoint_E}.

\begin{myex}
	Consider a pHDAE with given
	\begin{align*}
		\hamiltonian(x) & = \exp(\tfrac12x_1^2)-1+\tfrac12x_2^2,\quad \gradient\hamiltonian(x) =
		\begin{bmatrix}
			x_1\exp(\tfrac12x_1^2) \\
			x_2
		\end{bmatrix}
		,                                                                                        \\
		E(x)            & =
		\begin{bmatrix}
			1 \\
			1
		\end{bmatrix}
		\gradient\hamiltonian(x)\transp =
		\begin{bmatrix}
			x_1\exp(\tfrac12x_1^2) & x_2 \\
			x_1\exp(\tfrac12x_1^2) & x_2
		\end{bmatrix}
		,\quad\costate(x) = \tfrac12
		\begin{bmatrix}
			1 \\
			1
		\end{bmatrix}
		.
	\end{align*}
	Here, $E$ and $\costate$ are constructed such that $E\transp\costate=\gradient\hamiltonian$ holds, while the choice of the coefficients $J,R,B$ is free and they can be set e.g.~to zero, since they do not explicitly influence \eqref{eqn_colsp}. Correspondingly, we know that $\mathrm{colsp}(E\transp) = \mathrm{span}(\gradient\hamiltonian)$ holds pointwise.
	For $E$ we consider the midpoint approximation $\discreteE$ as in \eqref{eq:midpoint_E}, which results in
	\begin{equation*}
		\mathrm{colsp}\pset[\big]{\discreteE(x,x')\transp}
		= \mathrm{colsp}\pset[\big]{E(\tfrac{\state+\state'}{2})\transp}
		= \mathrm{span}\pset[\big]{\gradient\hamiltonian(\tfrac{x+x'}2)} .
	\end{equation*}
	A priori, we want to allow the choice of an arbitrary discrete gradient $\DG\hamiltonian$ of $\hamiltonian$.
	By \cite[Proposition~3.2]{McLQR99}, $\DG\hamiltonian$ may be split up into orthogonal contributions as
	\begin{equation}
		\label{eq:discrete_gradient_general_expression}
		\DG\hamiltonian(x,x') = \frac{\hamiltonian(x')-\hamiltonian(x)}{\lVert x'-x\rVert^2}(x'-x)+w(x,x')
	\end{equation}
	for $x\ne x'$, where $w$ satisfies $w(x,x')\in\mathrm{span}(\state'-\state)^\perp$
	and
	\begin{equation*}
		\lim_{x'\to x} \pset[\big]{ w(x,x')-\pi_{\mathrm{span}(x'-x)^\perp}\gradient\hamiltonian(x) } =0,
	\end{equation*}
	where
	\[
		\pi_{\mathrm{span}(x'-x)^\perp} = \frac{I-(x'-x)(x'-x)\transp}{\norm{x'-x}^2}
	\]
	denotes the orthogonal projection onto $\mathrm{span}(x'-x)^\perp$.
	Straightforward calculations yield that for the special choice
	\begin{equation*}
		x =
		\begin{bmatrix}
			a \\
			0
		\end{bmatrix}
		,\quad x' =
		\begin{bmatrix}
			0 \\
			b
		\end{bmatrix} , \quad  a,b \in \R \setminus \{0\},
	\end{equation*}
	where $b:= \pm a \sqrt{\exp(\frac{1}{2}(a/2)^2)}$,
	we have $\hamiltonian(x')\ne\hamiltonian(x)$ and $\gradient\hamiltonian(\tfrac{x+x'}2) \transp (x'-x) = 0$, and therefore
	\begin{equation*}
		x'-x\in \mathrm{span}\pset[\big]{\gradient\hamiltonian(\tfrac{x+x'}2)}^\perp = \mathrm{colsp}(\discreteE(x,x')\transp)^\perp.
	\end{equation*}
	Due to relation \eqref{eq:discrete_gradient_general_expression}, this implies $\DG\hamiltonian(x,x')\notin\mathrm{colsp}(\discreteE(x,x')\transp)$, i.e., there exists no vector $f$ satisfying $\discreteE(x,x')\transp f = \DG\hamiltonian(x,x')$ for this choice of $x,x'$.
	The discrete equation \eqref{eqn_colsp} is therefore unsolvable.
\end{myex}

\section{Composition and inversion of system transformations}\label{app_systemTransformation}

\noindent Suppose that $(\varphi,U):\widetilde\statespace\to\statespace\times\R^{n,n}$ and $(\widetilde\varphi,\widetilde U):\widehat\statespace\to\widetilde\statespace\times\R^{n,n}$ are two system transformations.
Then it is natural to define their \emph{composition} as
\[
	(\varphi,U) \circ (\widetilde\varphi,\widetilde U) = \pset[\big]{\varphi\circ\widetilde\varphi,(U\circ\widetilde\varphi)\widetilde U}:\widehat\statespace\to\statespace\times\R^{n,n},
\]
since applying $(\varphi,U) \circ (\widetilde\varphi,\widetilde U)$ to a system or gradient pair is equivalent to applying $(\varphi,U)$ and $(\widetilde\varphi,\widetilde U)$ consecutively.

Let now $\discretejacobian\varphi$ and $\discretejacobian\widetilde\varphi$ be discrete Jacobians for $\varphi$ and $\widetilde\varphi$, and let $\discrete{U}$ and $\widehat{U}$ be consistent approximations of $U$ and $\widetilde U$, respectively. Because of \eqref{eq:DG_chainRule}, $(\discretejacobian\varphi\circ\widetilde\varphi)\discretejacobian\widetilde\varphi$ is a discrete Jacobian for $\varphi\circ\widetilde\varphi$, while $(\discrete U\circ\widetilde\varphi)\widehat U$ is obviously a consistent approximation of $(U\circ\widetilde\varphi)\widetilde U$.
Note that this choice is consistent with \Cref{thm:discreteGradientPairChainRule}.

Consider now an invertible state transformation $(\varphi,U)$. We define its \emph{inverse} as $(\varphi,U)^{-1} = (\varphi^{-1}$, $U^{-1}\circ\varphi^{-1})$, since the compositions
\[
	(\varphi^{-1},U^{-1}\circ\varphi^{-1})\circ(\varphi,U)=(\mathrm{Id}_{\widetilde\statespace},I_n) \quad\text{and}\quad (\varphi,U)\circ(\varphi^{-1},U^{-1}\circ\varphi^{-1})=(\mathrm{Id}_{\statespace},I_n)
\]
leave all systems and gradient pairs they are applied to invariant.
Let us now denote $\psi=\varphi^{-1}$ and $V=U^{-1}\circ\varphi^{-1}$.
Then, given discrete Jacobians $\discretejacobian\varphi$ and $\discretejacobian\psi$ for $\psi$ and its inverse, and pointwise invertible consistent approximations $\discrete{U}$ and $\discrete{V}$ for $U$ and $V$, by applying \Cref{thm:discreteGradientPairChainRule} to a discrete gradient pair $(\discreteE,\discretecostate)$ for $(\hamiltonian,E,\costate)$ with the system transformation $(\varphi,U)$ and its inverse $(\psi,V)$ consecutively, we obtain the discrete gradient pair
\[
	(\widehat E,\hat\costate) = \pset[\big]{ \discrete V\transp(\discrete U\circ\psi)\transp\discreteE\,(\discretejacobian\varphi\circ\psi)\discretejacobian\psi \; , \; \discrete V^{-1}(\discrete U^{-1}\circ\psi)\discretecostate },
\]
for the same $(\hamiltonian,E,\costate)$ and the same state space.
Since the composition of $(\varphi,U)$ and $(\psi,V)$ leaves the gradient pair unaltered, it is sensible to choose $\discretejacobian\psi$ and $\discrete{V}$ in such a way that $(\widehat E,\hat\costate)=(\discreteE,\discretecostate)$.
If we want this choice to be independent of $(E,\costate)$, this is equivalent to the conditions $\discrete{V}=\discrete{U}^{-1}\circ\psi$ and $(\discretejacobian\varphi\circ\psi)\discretejacobian\psi=I_n$.
While the former condition can always be imposed and ensures that $\discrete{V}$ is a consistent approximation of $V$, the latter condition requires $\discretejacobian\varphi$ to be pointwise invertible, and in that case is equivalent to $\discretejacobian\psi=(\discretejacobian\varphi\circ\psi)^{-1}$.
This leads to the following canonical construction for the inverse discrete Jacobian.

\begin{lemma}\label{lem:discreteJacobianDiffeo}
	Let $\varphi\in\cont^1(\widetilde\statespace,\statespace)$ be a diffeomorphism between two open spaces $\statespace,\widetilde\statespace\subseteq\R^n$, and let $\discretejacobian{\varphi}$ be a pointwise invertible discrete Jacobian for $\varphi$.
	Then a discrete Jacobian for $\varphi^{-1}$, which we call the \emph{inverse discrete Jacobian} of $\varphi$ based on $\discretejacobian\varphi$, is
	$
		\discretejacobian(\varphi^{-1})=(\discretejacobian{\varphi}\circ\varphi^{-1})^{-1}.
	$
\end{lemma}
\begin{proof}
	For $\state,\state'\in\statespace$ and $\tilde\state=\varphi^{-1}(\state),\ \tilde\state'=\varphi^{-1}(\state')$ we have
	\[
		\discretejacobian(\varphi^{-1})(\state,\state')(\state'-\state)
		= \discretejacobian{\varphi}(\tilde\state,\tilde\state')^{-1} \pset[\big]{\varphi(\tilde\state')-\varphi(\tilde\state)}
		= \discretejacobian{\varphi}(\tilde\state,\tilde\state')^{-1} \discretejacobian{\varphi}(\tilde\state,\tilde\state') (\tilde\state' - \tilde\state)
		= \varphi^{-1}(\state') - \varphi^{-1}(\state)
	\]
	and $\discretejacobian(\varphi^{-1})(\state,\state)=\jacobian{\varphi}(\tilde\state)^{-1}=\jacobian(\varphi^{-1})(\state)$.
\end{proof}

\noindent Note that the requirement for $\discretejacobian\varphi$ to be pointwise invertible in the construction of the inverse discrete Jacobian is often in practice not restrictive, since the discrete Jacobians of diffeomorphisms are locally pointwise invertible, as discussed in \Cref{rem:DDR_systemTransformation}. However, it can happen that such discrete Jacobians are not globally pointwise invertible, as the following example highlights.

\begin{myex}\label{ex:counterexampleForInvertibleDiscreteJacobianOfADiffeomorphism}
	Let us consider $\varphi:\R^2\to\R^2,\ \varphi(\state)=\mathrm{Rot}(\state\transp\state)x$, where $\mathrm{Rot}:\R\to\R^{2,2}$ is the rotation matrix function, i.e.,
	\[
		\mathrm{Rot}(\theta) =
		\begin{bmatrix}
			\cos(\theta) & -\sin(\theta) \\ \sin(\theta) & \cos(\theta)
		\end{bmatrix}
	\]
	for all $\theta\in\R$.
	We note that $\varphi$ is $\cont^\infty$ and that it is a bijective map, since it decomposes into rotations of fixed angle on every circle centered in $0\in\R^2$.
	Furthermore, since
	\[
		\frac{\mathrm{d}\mathrm{Rot}}{\mathrm{d}\theta}(\theta) =
		\begin{bmatrix}
			-\sin(\theta) & -\cos(\theta) \\ \cos(\theta) & -\sin(\theta)
		\end{bmatrix}
		=
		\begin{bmatrix}
			\cos(\theta + \tfrac{\pi}{2}) & -\sin(\theta + \tfrac{\pi}{2}) \\ \sin(\theta + \tfrac{\pi}{2}) & \cos(\theta + \tfrac{\pi}{2})
		\end{bmatrix}
		= \mathrm{Rot}(\theta + \tfrac{\pi}{2})
		= \mathrm{Rot}(\theta)\mathrm{Rot}(\tfrac{\pi}{2}),
	\]
	we deduce that
	\[
		\jacobian\varphi(x) = \jacobian\pset[\big]{ \mathrm{Rot}(\state\transp\state)\state }
		= \mathrm{Rot}(\state\transp\state) + \frac{\mathrm{d}\mathrm{Rot}}{\mathrm{d}\theta}(\state\transp\state)\gradient(\state\transp\state)\state\transp
		= \mathrm{Rot}(\state\transp\state) \pset[\big]{ I_2 + 2\mathrm{Rot}(\tfrac{\pi}{2})\state\state\transp }.
	\]
	We now show that $\jacobian\varphi$ is pointwise invertible. Suppose that $\state,w\in\R^2$ satisfy $\jacobian\varphi(\state)w=0$: then $\pset[\big]{I+2\mathrm{Rot}(\tfrac{\pi}{2})\state\state\transp}w = 0$, from which $w = - 2\state\transp w\,\mathrm{Rot}(\tfrac{\pi}{2})\state$.
	In particular,
	\[
		\state\transp w = -2\state\transp w \pset[\big]{ \state\transp\mathrm{Rot}(\tfrac{\pi}{2})\state } = 0,
	\]
	and therefore $w=0$, allowing us to conclude that $\jacobian\varphi$ is pointwise invertible.

	Let now $\discretejacobian\varphi$ denote the Gonzales discrete Jacobian of $\varphi$ (see \eqref{eq:GonzalezJacobian}). In particular, for every $\state,\state'\in\R^2$ and $z\in(\state'-\state)^\perp$ we have
	\[
		\discretejacobian\varphi(\state,\state')(\state'-\state) = \varphi(\state') - \varphi(\state) \quad\text{and}\quad
		\discretejacobian\varphi(\state,\state')z = \jacobian\varphi(\tfrac{\state+\state'}{2})z.
	\]
	For $\state=0$, $\state'=(\sqrt{2\pi},0)$, $e_1=(1,0)$, and $e_2=(0,1)$, we obtain then $\state'=\sqrt{2\pi}e_1$ and $\state'-\state=\state'\perp e_2$, and therefore
	\[
		\discretejacobian\varphi(0,\state')e_1
		= \frac{\discretejacobian\varphi(0,\state')\state'}{\sqrt{2\pi}}
		= \frac{\varphi(\state')-\varphi(\state)}{\sqrt{2\pi}}
		= \frac{\mathrm{Rot}(2\pi)x'-\mathrm{Rot}(0)0}{\sqrt{2\pi}}
		= \frac{x'}{\sqrt{2\pi}} = e_1
	\]
	and
	\[
		\discretejacobian\varphi(0,\state')e_2
		= \jacobian\varphi(\tfrac{x+x'}{2})e_2
		= \mathrm{Rot}(\tfrac{\pi}{2})\pset[\big]{ e_2 + \tfrac{1}{2}\mathrm{Rot}(\tfrac{\pi}{2})x'(x')\transp e_2 }
		= \mathrm{Rot}(\tfrac{\pi}{2})e_2 = -e_1.
	\]
	Thus,
	\[
		\discretejacobian\varphi(\state,\state') = [e_1,-e_1] =
		\begin{bmatrix}
			1 & -1 \\ 0 & 0
		\end{bmatrix},
	\]
	which is singular.
\end{myex}

\noindent Note that, while in \Cref{ex:counterexampleForInvertibleDiscreteJacobianOfADiffeomorphism} the inverse discrete Jacobian of $\varphi$ based on $\discretejacobian\varphi$ cannot be constructed, the inverse diffeomorphism $\varphi^{-1}$ still admits discrete Jacobians, e.g.~the Gonzales one.
This in particular shows that the Gonzales discrete Jacobian of the inverse diffeomorphism does not in general coincide with the inverse discrete Jacobian based on the Gonzales discrete Jacobian.

We also present the following example, that shows that the Gonzales discrete gradient construction does not in general commute with system transformations.

\begin{myex}
	Consider the function $f:\R^2\to\R,\ f(x)=\frac{1}{4}\norm{x^4}$ and the change of variables $\varphi:\R^2\to\R^2,\ \varphi(x_1,x_2)=(x_1+x_2,x_2)$, which yield
	\[
		\gradient f = \norm{x}^2 x, \qquad \jacobian\varphi =
		\begin{bmatrix}
			1 & 1 \\ 0 & 1
		\end{bmatrix},
	\]
	and let $\DG_Gf$ and $\discretejacobian_G\varphi$ denote the corresponding Gonzales discrete gradient and Gonzales discrete Jacobian, noting in particular that $\discretejacobian_G\varphi=\jacobian\varphi$, since it is a constant matrix.
	We have now two natural ways to construct a discrete gradient for the composed map $\tilde f=f\circ\varphi$: either as its Gonzales discrete gradient $\DG_G\tilde f$, or by using the chain rule \eqref{eq:DG_chainRule}, i.e., $\DG\tilde f=\discretejacobian_G\varphi\transp\DG_Gf$.
	We have then
	\[
		\DG_G\tilde f(0,2e_1)\transp e_2 = \gradient\tilde f(e_1)\transp e_2 = \gradient f\pset[\big]{\varphi(e_1)}\transp \jacobian\varphi\, e_2 = e_1\transp
		\begin{bmatrix}
			1 & 1 \\ 0 & 1
		\end{bmatrix}
		e_2 = 1
	\]
	and
	\begin{align*}
		\DG\tilde f(0,2e_1)\transp e_2
		 & = \DG_G\tilde f(0,2e_1)\transp \discretejacobian\varphi\, e_2
		= \DG_G f(0,2e_1)\transp (e_1+e_2)                                         \\
		 & = \tfrac{1}{2} \DG_G f(0,2e_1)\transp 2e_1 + \DG_G f(0,2e_1)\transp e_2
		= \tfrac{1}{2}\pset[\big]{f(2e_1)-f(0)} + \gradient f(e_1)\transp e_2 = 2,
	\end{align*}
	thus $\DG\tilde f$ and $\DG_G\tilde f$ do not coincide.
\end{myex}

\section{Details on the mass-spring multibody system example}\label{appendix_mass_spring}

\noindent The multibody system with singular mass matrix from Example~\ref{ex_singularM},
with masses $m_i$ and springs with constants $k_i$ and resting lengths $l_{i0}$, with $i \in \{1,2 \}$,
takes up the considerations in Section~\ref{modelling_mbs}. The system can be regarded as a modular multibody system comprising two separate subsystems with two degrees of freedom.
However, we decide to use two coordinates for the elongations of the springs ($x_1$ and $x_2$) and one coordinate $q_2$ for the point where the two subsystems are interconnected.
Correspondingly, we use
\begin{align*}
	{q}= \begin{bmatrix}
		     x_1 \\ q_2 \\ x_2
	     \end{bmatrix}
	\quad \text{and} \quad
	{v} = \begin{bmatrix}
		      v_1 \\ v_2 \\ v_3
	      \end{bmatrix} = \begin{bmatrix}
		                      \dot{x}_1 \\ \dot{q}_2 \\ \dot{x}_2
	                      \end{bmatrix}
\end{align*}
and the interconnection constraint $ q_2 = x_1 + l_{10} + w $ arises. The total kinetic energy is given by
\begin{align} \label{ex_kinetic_energy_singularM}
	T({v}) = \frac{1}{2} m_1 v_1^2 + \frac{1}{2} m_2 (v_2 + v_3)^2 = \frac{1}{2} {v}
	\transp M v
\end{align}
and thus we identify the mass matrix $M$ from \eqref{singular_M}, which is constant and singular for all configurations.

Following \Cref{prop:discreteGradientPair_constantE},
we perform a singular value decomposition to arrive at a semi-explicit pHDAE formulation. For the present case, we have
\begin{align}
	E =\begin{bmatrix}
		   I_{3 \times 3} & 0 & 0 \\
		   0              & M & 0 \\
		   0              & 0 & 0
	   \end{bmatrix}.
\end{align}
Here we identify the SVD of $E$ given by $E=U\Sigma V\transp=[U_1,U_2]\diag(\Sigma_1,0)[V_1,V_2]\transp$, where
$\Sigma_1 = \diag{(1,1,1,m_1,2m_2)}$ and
\begin{align}
	V_1 = W_1 = \begin{bmatrix}
		            1 & 0 & 0 & 0 & 0                  \\
		            0 & 1 & 0 & 0 & 0                  \\
		            0 & 0 & 1 & 0 & 0                  \\
		            0 & 0 & 0 & 1 & 0                  \\
		            0 & 0 & 0 & 0 & \frac{1}{\sqrt{2}} \\
		            0 & 0 & 0 & 0 & \frac{1}{\sqrt{2}} \\
		            0 & 0 & 0 & 0 & 0
	            \end{bmatrix} \ , V_2 = W_2 = \begin{bmatrix}
		                                          0                   & 0 \\
		                                          0                   & 0 \\
		                                          0                   & 0 \\
		                                          0                   & 0 \\
		                                          -\frac{1}{\sqrt{2}} & 0 \\
		                                          \frac{1}{\sqrt{2}}  & 0 \\
		                                          0                   & 1
	                                          \end{bmatrix}  .
\end{align}

\section{Transformation to semi-explicit form}
\label{appendix_to_semiexplicit}

\noindent The following theorem provides an SVD-like decomposition for a matrix function with constant rank. This result is helpful in deriving conditions for the existence of a semi-explicit representation of a pHDAE.

\begin{theorem}\label{thm:nullSpaceDec}
	Let $E\in\cont(\statespace,\R^{m,n})$, with $\rank(E(\state))=r$ for all $\state\in\statespace$.
	Then for every $\state_0\in\statespace$ there exists an open neighborhood $\statespace_0\subseteq\statespace$ of $\state_0$ and pointwise unitary functions $U\in\cont(\statespace_0,\R^{m,m})$ and $V\in\cont(\statespace_0,\R^{n,n})$ such that
	\begin{equation}
		U\transp EV = \begin{bmatrix}
			\Sigma & 0 \\ 0 & 0
		\end{bmatrix} ,
	\end{equation}
	with pointwise invertible $\Sigma\in\cont(\statespace_0,\R^{r,r})$.
	Furthermore, if the entries of $E$ are analytic or $\cont^\ell$ for some $\ell\in\N_0\cup\set{\infty}$, then $U$, $V$, and $\Sigma$ can be chosen to have that same regularity.
\end{theorem}
\begin{proof}
	Let
	\[
		U_0\transp E(\state_0)V_0 = \begin{bmatrix}
			\Sigma_0 & 0 \\ 0 & 0
		\end{bmatrix}
	\]
	denote the singular value decomposition of $E(\state_0)$, with $U_0\in\R^{m,m}$ and $V_0\in\R^{n,n}$ unitary matrices, and $\Sigma_0\in\R^{r,r}$ nonsingular, and let us split correspondingly
	\[
		U_0\transp E(\state)V_0 = \begin{bmatrix}
			E_{11}(\state) & E_{12}(\state) \\ E_{21}(\state) & E_{22}(\state)
		\end{bmatrix}
	\]
	for all $\state\in\statespace$. Since $E_{11}(\state_0)=\Sigma_0$ is invertible and $E$ is continuous, there is an open neighborhood $\statespace_0\subseteq\statespace$ of $\state_0$ such that $E_{11}(\state)$ is invertible for all $\state\in\statespace_0$.
	We note that $E_{ij}$ has at least the same regularity as $E$ for all $i,j$, and that, since the determinant of a matrix function clearly has the same regularity as its entries, and the pointwise inverse of $E_{11}$ on $\statespace_0$ can be expressed in the form
	\[
		E_{11}(\state)^{-1} = \frac{\mathrm{adj}(E_{11}(\state))}{\det(E_{11}(\state))},
	\]
	where $\mathrm{adj}(E_{11}(\state))$ denotes the adjugate matrix of $E_{11}(\state)$ (whose entries are determinants of submatrices of $E_{11}(\state)$), and $\det(E_{11}(\state))$ does not vanish on $\statespace_0$, $E_{11}$ will also have the same regularity as $E$ on $\statespace_0$.

	Let now
	\[
		\widetilde U(\state)\transp \coloneqq \begin{bmatrix}
			I_r & 0 \\ -E_{21}(\state)E_{11}(\state)^{-1} & I_{m-r}
		\end{bmatrix}U_0\transp, \qquad
		\widetilde V(\state) \coloneqq V_0\begin{bmatrix}
			I_r & -E_{11}(\state)^{-1}E_{12}(\state) \\ 0 & I_{n-r}
		\end{bmatrix},
	\]
	so that
	\[
		\widetilde U(\state)\transp E(x)\widetilde V(\state) = \begin{bmatrix}
			E_{11}(\state) & 0 \\ 0 & E_{22}(\state) - E_{21}(\state)E_{11}(\state)^{-1}E_{12}(\state)
		\end{bmatrix}
	\]
	for all $\state\in\statespace_0$. Since $E_{11}(\state)\in\R^{r,r}$ is invertible for all $\state\in\statespace_0$ and the rank of $E$ is constantly $r$, we deduce that actually
	\[
		\widetilde U(\state)\transp E(x)\widetilde V(\state) = \begin{bmatrix}
			E_{11}(\state) & 0 \\ 0 & 0
		\end{bmatrix}
	\]
	for all $\state\in\statespace_0$.
	We note that $\widetilde U$ and $\widetilde V$ also have the same regularity as $E$.

	It remains to show that $\widetilde U$ and $\widetilde V$ can be replaced by pointwise unitary matrix functions with the same regularity.
	Let $\widetilde U(\state)=U(\state)L_U(\state)$ and $\widetilde V(\state)=V(\state)L_V(\state)$ be the QL factorizations of $\widetilde U$ and $\widetilde V$, which can be computed for all $\state\in\statespace_0$ with the Gram-Schmidt orthogonalization process, in particular $U$ and $V$ are pointwise unitary, and $L_U$ and $L_V$ are pointwise lower triangular with positive diagonal entries.
	Note that, due to the pointwise invertibility of $\widetilde U$ and $\widetilde V$, the construction of the Gram-Schmidt process ensures that $U,V,L_U,L_V$ have the same regularity as $\widetilde U,\widetilde V$.
	Furthermore, since lower triangular matrices form a multiplicative group, we deduce that
	\begin{align*}
		U(\state)\transp E(\state)V(\state)
		 & = L_U(\state)^{-\top}\widetilde U(\state)\transp E(\state)\widetilde V(\state)L_V(\state)^{-1}               \\
		 & = \begin{bmatrix}
			     \widetilde L_{11}(\state)\transp & \widetilde L_{21}(\state)\transp \\ 0 & \widetilde L_{22}(\state)\transp
		     \end{bmatrix}
		\begin{bmatrix}
			E_{11}(\state) & 0 \\ 0 & 0
		\end{bmatrix}
		\begin{bmatrix}
			\widehat L_{11}(\state) & \widehat L_{21}(\state) \\ 0 & \widehat L_{22}(\state)
		\end{bmatrix}                                 \\
		 & =
		\begin{bmatrix}
			\widetilde L_{11}(\state)\transp E_{11}(\state)\widehat L_{11}(\state) & 0 \\ 0 & 0
		\end{bmatrix},
	\end{align*}
	where $\widetilde L(\state)=L_U(\state)^{-1}$ and $\widehat L(\state)=L_V(\state)^{-1}$ have again the same regularity as $L_U$ and $L_V$, due to the same observations that we made for $E_{11}$.
	We conclude that $U$ and $V$ have all the requested properties.
\end{proof}

\noindent While \Cref{thm:nullSpaceDec} is quite powerful, unfortunately it cannot be directly applied to obtain system transformations that bring a descriptor matrix $E$ into its semi-explicit form, since the matrix function $V$ is not guaranteed to be the Jacobian of a diffeomorphism. However, this result is the first step in the direction of proving \Cref{thm:analyticGradientPair}. The second step is to deduce the following corollary.

\begin{corollary}\label{lem:unitary_permutation}
	Let $E\in\cont(\statespace,\R^{m,n})$ with $\rank(E(\state))=r$ for all $\state\in\statespace$.
	Then for every $\state_0\in\statespace$ there exist an open neighborhood $\statespace_0\subseteq\statespace$ of $\state_0$, a pointwise invertible
	matrix function $U\in\cont(\statespace_0,\R^{m,m})$, and a permutation matrix $\Pi\in\R^{n,n}$, such that
	\begin{equation}
		U\transp E\Pi = \begin{bmatrix}
			I_r & E_{12} \\ 0 & 0
		\end{bmatrix} ,
	\end{equation}
	with
	$E_{12}\in\cont(\statespace_0,\R^{r,n-r})$.
	Furthermore, if the entries of $E$ are analytic or $\cont^\ell$ for some $\ell\in\N_0\cup\set{\infty}$, then $U$ and $E_{12}$ can be chosen to have that same regularity.
\end{corollary}

\begin{proof}
	Because of \Cref{thm:nullSpaceDec}, for every $\state_0\in\statespace$ there exist an open neighborhood $\statespace_0\subseteq\statespace$ of $\state_0$, pointwise unitary $\widetilde U,\widetilde V$, and a pointwise invertible $\Sigma\in\cont(\statespace_0,\R^{r,r})$, such that $\widetilde U,\widetilde V,\Sigma$ have the same regularity as $E$, and
	\[
		\widetilde U\transp E\widetilde V = \begin{bmatrix} \Sigma & 0 \\ 0 & 0 \end{bmatrix}.
	\]
	In particular,
	\[
		\widetilde U\transp E
		=
		\begin{bmatrix}
			\Sigma & 0 \\ 0 & 0
		\end{bmatrix}
		\widetilde V\transp
		=
		\begin{bmatrix}
			\widetilde E_{11} & \widetilde E_{12} \\ 0 & 0
		\end{bmatrix}.
	\]
	Since $\rank(U\transp E(x_0))=\rank(E(x_0))=r$, there is a permutation matrix $\Pi\in\R^{n,n}$ such that
	\begin{equation*}
		\widetilde U\transp E\Pi = \begin{bmatrix}
			\widehat E_{11} & \widehat E_{12} \\ 0 & 0
		\end{bmatrix} ,
	\end{equation*}
	with $\widehat E_{11}(\state_0)$ invertible. Since $\widehat E_{11}$ is continuous, up to further restricting the open neighborhood $\statespace_0$, we obtain
	that $\widehat E_{11}$ is invertible in $\statespace_0$. Then, by choosing
	\[
		U = \widetilde U
		\begin{bmatrix}
			\widehat E_{11}\ntransp & 0 \\ 0 & I_{n-r}
		\end{bmatrix},
	\]
	we obtain that
	\[
		U\transp E\Pi =
		\begin{bmatrix}
			\widehat{E}_{11}^{-1} & 0 \\ 0 & I_{n-r}
		\end{bmatrix}
		\widetilde U\transp E\Pi =
		\begin{bmatrix}
			I_r & E_{12} \\ 0 & 0
		\end{bmatrix}
	\]
	for $E_{12}=\widehat E_{11}^{-1}\widehat E_{12}$.
	Finally, $\widetilde E$ and $E_{12}$ have the same regularity as $E$, by construction.
\end{proof}

\noindent One advantage of \Cref{lem:unitary_permutation} over \Cref{thm:nullSpaceDec} is that $\Pi$ can be interpreted as the Jacobian of a diffeomorphism, in fact $\Pi=\jacobian\sigma$ with $\sigma(\state)=\Pi\state$. We can now proceed with the proof of \Cref{thm:analyticGradientPair}.

\begin{proof}[Proof of \Cref{thm:analyticGradientPair}]
	Let us fix $\state_0\in\statespace$, and let us choose $\statespace_0$, $U$, $\Pi$, $E_{12}$ as in \Cref{lem:unitary_permutation}. Then, up to applying the invertible system transformation $(\sigma,U)$ with $\sigma(\state)=\Pi\state$ to $E$, we assume without loss of generality that
	\[
		E =
		\begin{bmatrix}
			I_r & E_{12} \\ 0 & 0
		\end{bmatrix},
	\]
	where $E_{12}\in\cont(\statespace,\R^{r,n-r})$ is a matrix function with analytic entries.
	Consider now the linear first order PDE system
	\begin{equation}\label{eq:PDE_for_changeOfVar}
		\gradient_{x_2}v(x) = E_{12}(x)\transp\gradient_{x_1}v(x),
	\end{equation}
	and let $f_1,\ldots,f_p\in\cont^1(\statespace,\R)$ be functionally independent functions that generate the solutions of \eqref{eq:PDE_for_changeOfVar}, i.e., such that $\gradient f_1,\ldots,\gradient f_p$ are (pointwise) linearly independent, and that the solutions of \eqref{eq:PDE_for_changeOfVar} are the functions of the form $v(x)=V(f_1(x),\ldots,f_p(x))$ for any $V\in\cont^1(\R^k,\R)$.
	For the existence of such set of solutions for $E_{12}$ with analytic entries, see e.g.~\cite{Jon07,Goursat59}.

	Let us now define $\psi_1=(f_1,\ldots,f_p)\in\cont^1(\statespace,\R^p)$. In particular, we have that $\jacobian_{x_2}\psi_1=(\jacobian_{x_1}\psi_1)E_{12}$.
	Let us now complete $f_1,\ldots,f_p$ to a maximal set of functionally independent functions $f_1,\ldots,f_n\in\cont^1(\statespace,\R)$, which can be done locally e.g.~by selecting an appropriate subset of $\state_1,\ldots,\state_n$, and let
	\[
		\psi=(f_1,\ldots,f_n)=(\psi_1,f_{p+1},\ldots,f_n)\in\cont^1(\statespace,\R^n)
	\]
	be the corresponding local diffeomorphism.
	In particular, up to further restricting the open neighborhood $\statespace_0$ of $\state_0$, we assume that $\psi:\statespace_0\to\psi(\statespace_0)$ is a diffeomorphism.
	Let then $\widetilde\statespace_0=\psi(\statespace_0)$ and let $\varphi=\psi^{-1}\in\cont^1(\widetilde\statespace_0,\statespace_0)$ denote the inverse diffeomorphism.

	Note that, since the original Hamiltonian $\hamiltonian$ is also a solution of \eqref{eq:PDE_for_changeOfVar}, it must be of the form $\hamiltonian=\specified{\widetilde\hamiltonian}\circ\psi_1$ for some $\specified{\widetilde\hamiltonian}\in\cont^1(\pi_1(\widetilde\statespace_0),\R)$, where $\pi_1:\R^n\to\R^p$ denotes the projection onto the first $p$ coordinates.
	Furthermore, it holds that
	\begin{align*}
		\jacobian\psi
		\begin{bmatrix}
			I_r & E_{12} \\ 0 & 0
		\end{bmatrix}
		(\jacobian\varphi\circ\psi)
		=
		\begin{bmatrix}
			\jacobian_{\state_1}\psi_1 & (\jacobian_{\state_1}\psi_1)E_{12} \\
			\star                      & \star
		\end{bmatrix}
		(\jacobian\psi)^{-1}
		=
		\begin{bmatrix}
			\jacobian\psi_1 \\ \star
		\end{bmatrix}
		(\jacobian\psi)^{-1}
		=
		\begin{bmatrix}
			I_p & 0 \\ \star & \star
		\end{bmatrix}.
	\end{align*}
	Then, up to applying the invertible system transformation $(\varphi,(\jacobian\psi\circ\varphi))$,
	we assume without loss of generality that $E$ is of the form
	\[
		E(x) =
		\begin{bmatrix}
			I_p & 0 \\ E_{12}(x) & E_{22}(x)
		\end{bmatrix}.
	\]
	and that $\hamiltonian$ admits a specified Hamiltonian $\specified\hamiltonian\in\cont^1(\pi_1(\statespace_0),\R)$.

	With subsequent left multiplications and right permutations (and therefore system transformations of the form $(\sigma,U)$ with $\sigma(\state)=\Pi\state$), we further bring $E$ to the form
	\[
		E(x) =
		\begin{bmatrix}
			I_p & 0 & 0 \\ 0 & 0 & 0 \\ 0 & E_{32}(x) & E_{33}(x)
		\end{bmatrix},
	\]
	where $E_{32}\in\cont(\statespace_0,\R^{r-p,n-r})$ and $E_{33}\in\cont(\statespace_0,\R^{r-p,r-p})$, with $E_{33}(\state_0)$ invertible.
	In particular,
	up to restricting $\statespace_0$ to an open neighborhood of $\state_0$ where $E_{33}$ is pointwise invertible, and applying an appropriate left multiplication with its inverse, we can assume without loss of generality that $E$ is of the form
	\[
		E(x) =
		\begin{bmatrix}
			I_p & 0 & 0 \\ 0 & 0 & 0 \\ 0 & E_{32}(x) & I_{r-p}
		\end{bmatrix}.
	\]
	Let us partition $x=(x_1,x_2,x_3)\in\R^p\times\R^{n-r}\times\R^{r-p}$ and $z=(z_1,z_2,z_3)\in\R^p\times\R^{n-r}\times\R^{r-p}$ .
	The gradient pair property $E\transp\costate=\gradient\hamiltonian$
	then implies
	\[
		\begin{bmatrix}
			z_1 \\ E_{32}\transp z_3 \\ z_3
		\end{bmatrix}
		=
		\begin{bmatrix}
			I_p & 0 & 0 \\ 0 & 0 & E_{32}\transp \\ 0 & 0 & I_{r-p}
		\end{bmatrix}
		\begin{bmatrix}
			z_1 \\ z_2 \\ z_3
		\end{bmatrix}
		= E\transp z = \gradient\hamiltonian =
		\begin{bmatrix}
			\gradient\specified{\hamiltonian}\circ\pi_1 \\ 0 \\ 0
		\end{bmatrix},
	\]
	i.e., $\costate_1=\gradient\specified{\hamiltonian}\circ\pi_1$ and $z_3=0$.
\end{proof}







\end{document}